\documentclass[12pt,twoside]{book}
\usepackage[latin1]{inputenc}
\usepackage[T1]{fontenc}
\usepackage[english]{babel}
\usepackage{latexsym}
\usepackage{amsmath}
\usepackage{amsfonts}
\usepackage{amssymb}
\usepackage{graphicx}
\usepackage[lmargin=2.5 cm,rmargin=2 cm,tmargin=3.5cm,bmargin=2.5cm,
paper=a4paper]{geometry}

\newcommand{\C}{\mathbb{C}}

\newenvironment{proof}{\noindent{\bf Proof.}\enspace}{
\hfill $\square$ \medskip}

\newtheorem{theorem}{Theorem}[section]
\newtheorem{proposition}{Proposition}[section]
\newtheorem{lemma}{Lemma}[section]
\newtheorem{remark}{Remark}[section]
\newtheorem{corollary}{Corollary}[section]
\newtheorem{definition}{Definition}[section]
\newtheorem{example}{Example}[section]

\numberwithin{equation}{section}

\usepackage{amssymb}
\usepackage{fancyhdr}
\usepackage{t1enc}

\begin{document}

\thispagestyle{empty}

\vskip3cm
\begin{center}
	{\bf \Large
		The construction of $\Psi^\star$-algebra by commutator method from the Bergman projection on the unit ball of $\C^{n}$
	}
\end{center}

\vskip6cm

\begin{center}
	{\bf
		by\\
		Dr. H. ISSA
	}
\end{center}
\vskip6cm

\begin{center}
	{\bf
		Department of Mathematics and Physics\\
		School of Arts and Science\\
		Lebanese International University\\
		Nov. 01, 2018
	}
\end{center}
\thispagestyle{empty}

\newpage
\text{ }
\vskip6cm
\begin{center}
	{\bf \Large Dedication }
	\vskip1.5cm
	Dieses Manuskript ist dem Betreuer meiner Doktorarbeit gewidmet. Ich m\"{o}chte mich herzlich bedanken bei Prof. Dr. Wolfram Bauer f\"{u}r seine  Unterst\"{u}tzung w\"{a}hrend des studiums.\\
	
	\bigskip
	This manuscript is dedicated to my Ph.D. supervisor Prof. Dr. Wolfram Bauer. These notes and results were partially prepared during my stay in Greifswald University in 2008 and continued during my work at LIU in 2018.
\end{center}

\chapter*{Abstract}
\addcontentsline{toc}{chapter}{Introduction}
\chaptermark{Abstract}
The  manuscript is devoted to the construction of $\Psi^\star$- algebras containing the Bergman projection on the unit ball $B_n$ of $\mathbb C^n$. Motivated by the local properties for pseudo-differential operators with symbols  in the H\"{o}rmander classes $S_{\rho,\delta}^{0}$, the notion for  $\Psi^\star$- algebras were introduced. These are Fr\'{e}chet algebras and being closed under holomorphic functional calculus. We consider the $C^\star$-algebra $\mathcal{L}(L^2(B_n))$ of bounded operator acting on the Hilbert space of square integrable functions on $B_n$ with respect to the standard probability measure. We search for  spectral invariant Fr\'{e}chet  subalgebras of $\mathcal{L}(L^2(B_n))$ containing the Bergman projection $P$  which is defined on $L^2(B_n)$ with values in the space of holomorphic functions. We use the commutator method as introduced by Gramsch in \cite{gr1}. This method generalizes the work of R. Beals  for the spectral invariance of pseudodifferential operators (cf. \cite{be,Ue}). Roughly speaking, one need not to work on the H\"{o}rmander classes \cite{hor} but can describe the algebra of pseudodifferential operators as continuous commutator between scales of Sobolev spaces. To connect the analytical properties  to the geometry of the unit ball \cite{grudsky0}, we search for linear tangent vector fields $X$ on the the unit sphere $\partial B_n$ so that the commutator $[X,P]$ has a continuous extension to $L^2(B_n)$. In comparison to the work of Bauer in \cite{B,br} for the case of the Fischer-Fock space it turns out that the results are completely different. Our first contribution states that the set of vector fields which provides the $\Psi^\star$-algebra for the Fock space fails even to provide a continuity with the Bergman projection in the case of the unit ball. Our second contribution, is the construction  of vector fields being tangent to the unit sphere  which provides a new  $\Psi^\star$-algebra containing the Bergman projection of the unit ball. Roughly speaking, we prove that every linear vector field on the  unit sphere commutes with the Bergman projection on the unit ball.

\tableofcontents
\chapter{Introduction}

Motivated by the work of W. Bauer in \cite{B,br} for constructing $\Psi^\star$- algebras starting from the Bergman projection on the Segal-Bargmann space, we consider a similar question in the more complicated situation i.e. in the unit ball $\mathbb B_n$ of $\mathbb C^n$. Following recent results which describes the commutator property between Toeplitz operators by the geometric invariance of these symbols inside the ball and not by smoothness of the symbols, it turns out that algebraic properties are being held beneath the geometry of the balls. Roughly speaking,  symbols which are invariant under the group action of maximal subgroup of the automorphism of the unit ball provides commutative Banach algebras (cf. \cite{grudsky0}). The case of the higher dimensional complex space is completely different and  automorphisms forms an infinite dimensional group. This addresses us to compare the two situation and whether we can use the geometry of the space to obtain some algebraic properties on operators. In \cite{B,br}, W. Bauer found a collection of linear vector fields so that the commutator is being uniquely extended by continuity and a $\Psi^\star$-algebra was constructed starting from the Bergman projection on the complex space.
The importance of $\Psi^\star$-algebras lies behind their topological properties as being Fr\'{e}chet -algebras continuously embedded in a $C^\star$-algebra. Moreover, they are inverse closed and so holomorphic functional calculus can be done. The idea behind introducing these algebras and relaxing the notion of $C^\star$-algebras dates back to the work of R. Beals in \cite{be} for pseudodifferential operators. Roughly speaking, for $\rho,\delta \in [0,1]$ with $\delta<\rho$, pseudodifferential operators with symbols in the H\"{o}rmander classes  $S^{0}_{\rho,\delta}(\Omega)$  defines a collection of (extended) bounded operators on $L^{2}(\Omega)$ forming a Fr\'{e}chet  embedded subalgebra and being closed under holomorphic functional calculus. The approach of Beals did not depend on considering the H\'{o}rmander symbols but rather defining a scale of bounded operators between Sobolev spaces using derivations and multiplication. In fact, he showed that the obtained collection of pseudodifferential operators coincides with the following  
\begin{equation*}
\Psi^{0}_{\rho,\delta}:=\{a\in \mathcal{L}(H^{0})\mid\ ad(M)^{\alpha}ad(\partial)^{\alpha}(a)\in \cap_{s \in \mathbb{R}}\mathcal{L}(H^{s-\rho |\alpha|+\delta|\beta|},H^{s}),\alpha,\beta\in \mathbb{N}^{n} \},
\end{equation*}\par

where $M$  denotes the multiplication by the real coordinates and $H^s$ are the Sobolev spaces. Motivated by the work of Beals , Gramsch introduced the notion of $\Psi^\star$-algebras and gave a general way for construction such algebras (cf. \cite{gr1,gr3}). Indeed, given a Hilbert space $H$  and fix $a\ \in \mathcal{L}(H)$.  Consider a finite family $\mathcal{V}$ of densely defined closed operators (derivations) on $H$ we require that all the iterated commutators
\begin{equation*}
[[a,V_{1}],V_{2},\cdots],\ \ V_{j}\in \mathcal{V}
\end{equation*}
are well defined on a suitable dense subset of $H$ and they admit an extension in $\mathcal{L}(H)$. This produces a $\Psi^\star$-algebra in $\mathcal{L}(H)$ (cf. Section 7.2 for the construction). 

In connection to the analysis Toeplitz operators on the unit ball $\mathbb B_{n} \subset \mathbb{C}^{n}$, we  consider the space $L^{2}(B_{n}):={L}^{2}(\mathbb B_{n},d\mu)$ of square-integrable complex valued functions on $\mathbb B_{n}$, where $\mu$ is the Lebesgue measure on $\mathbb{C}^{n}$ normalizing $\mathbb B_{n}$. Then the space of entire function in $L^{2}(B_{n})$ forms a closed subspace, denoted by $\mathbb{H}^{2}(\mathbb B_n)$, and is called the Bergman space on the unit ball. 

\bigskip
The thesis is devoted to answer   the following problems:

\begin{itemize}
	\item[(I)] Let $a=P$ and the Consider the set of vector fields  generated by linear combinations of $\frac{\partial}{\partial z_{k}}$ (or $\frac{\partial}{\partial\overline{z_{k}}})$. This set generates a $\Psi$-algebra containing the projection onto the Segal Bargmann space. Now let $a=P$ and consider the same set of vector fields.   By studying the extension of continuity of the commutator of these vector fields with $P$, can we obtain a $\Psi^\star$-algebra in $L^{2}(B_{n})$ containing the Bergman projection on the unit ball.

	\item[(II)] Can we construct a collection of vector fields being tangent to the unit sphere and provides a $\Psi^\star$-algebra in $L^{2}(B_{n})$ containing the Bergman projection on the unit ball? Do linear vector fields provide such an algebra?
\end{itemize}

Regarding Problem (I) we show  that  such a $\Psi^\star$-algebra can not be obtained. The detailed proof is given is Section7.3  Addressing Problem (II), we show that every linear vector field on the sphere commutes with Bergman projection.  In Section 7.5, we collect the results and compare the case of the Bergman space over the unit ball to that of the Fock-space and provide some open problems motivated by our work.

\chapter{Some types of algebras}
\hspace{1cm}
In this chapter we investigate some basic properties on several types of algebras. First section is devoted to Banach algebras $\mathcal{B}$ with standard examples. In the second section, we are interested in the invertible elements of a unital Banach algebra. We introduce the basic properties of the spectrum in a Banach algebra.\par

  \section{Algebras and Banach algebras}
 \hskip0.5cm We start by introducing the notion and the basic properties of an algebra.
 \begin{definition}
 	A vector space $\mathcal{E}$ over $\mathbb{C}$ is called an algebra if $\mathcal{E}$ is endowed with a multiplication map $$\mu:\mathcal{E}\times\mathcal{E}\longmapsto \mathcal{E} $$ such that $\mu $ is bilinear and associative.
 	The multiplication between 2 elements $x,y \in \mathcal{E}$ is usually denoted by $x.y$, i.e. $\mu (x,y)= x.y $. Thus for all $x, y, z \in \mathcal{E}$ and $\lambda \in \mathbb{C}$, the conditions of an algebra are given as follows:
 	\begin{enumerate}
 		\item $x.(\lambda y+z)= \lambda xy +xz$ and
 		\item $(x.y).z = x.z + y.z$.
 	\end{enumerate} 
 \end{definition}

From now on, an algebra will be denoted by $\mathcal{A}$.
 \begin{definition}
 	Suppose there is an element $e \in \mathcal{A}$ satisfying $$e.x =x.e =x$$ for all $x \in \mathcal{A}$. Then $\mathcal{A}$ is called a unital algebra and the element $e$ is called identity element or unit.\par
 	It is direct from the definition of the neutral element that in case of existence,  the neutral element is unique. It should be mentioned that not every algebra is a unital algebra. Indeed, the space of all real valued functions defined on $\mathbb{R}$, whose limit at infinity is 0, is an algebra under the point wise multiplication without a unit. 
 \end{definition}
 \begin{example}
 	The above mentioned algebra is denoted by $$\mathcal{A}= \{f: \mathbb{R} \longrightarrow \mathbb{R} \mid \lim_{x\to\infty} f(x)= 0\}$$ with the following operations:
 	\begin{enumerate}
 		\item $(f+g)(x) = f(x) + g(x)$
 		\item $(f.g)(x) =[f(x)].[g(x)]$
 		\item $(\lambda f)(x) = \lambda (f(x))$
 	\end{enumerate}
 	with $\lambda \in \mathbb{R}$ and $f, g \in \mathcal{A}$.\\
 	\\
 	If we consider $f(x) = \frac{1}{1+x^2} \in \mathcal{A}$ then $$\frac{e(x)}{1+x^2}= \frac{1}{1+x^2} \Leftrightarrow e(x)= 1$$ but $e(x)\notin \mathcal{A}$ since $\lim_{x\to\infty}e(x)\neq 0$, this shows that $\mathcal{A}$ has no unital element.
 \end{example}\par
 In the below definition we review the notion of a Banach algebra.
 \begin{definition}
 	Let $\mathcal{A}$ be a unital algebra. Suppose $\mathcal{A}$ is endowed with a submultiplicative norm, i.e. \begin{equation}\label{e1}
 	\|x.y\| \leq \|x\|.\|y\|, \mbox{ for all } x,y \in \mathcal{A}.
 	\end{equation} Then $(\mathcal{A},\|.\|)$ is called a normed algebra. If in addition, $(\mathcal{A},\|.\|)$ is a Banach space then $\mathcal{A}$ is called a Banach algebra.
 \end{definition}\par
 From now on, a Banach will be denoted by $\mathcal{B}$ or $(\mathcal{B},\|.\|)$.
 \begin{remark}
 	If $\mathcal{B}$ is a Banach algebra then it is a normed vector space and hence the collection $$\big\{B(x,r) ,x \in \mathcal{B}, r > 0\big\}$$ form a basis for $\mathcal{B}$. Moreover, the multiplication map defined on $\mathcal{B} \times \mathcal{B}$ by \begin{equation*}
 	\mu (x,y) = xy
 	\end{equation*} is continuous.
 \end{remark}\par
 We note that the submultiplicativity of the norm ensures that $\|e\| \geq 1$. However, the below proposition shows that we can suppose that  $\|e\|=1$.
 \begin{proposition}
 	Let $(\mathcal{B},\|.\|)$ be a Banach algebra, then there is $\|.\|^{'}$ equivalent to $\|.\|$ so that $(\mathcal{B},\|.\|^{'})$ is a Banach algebra and $\|e\|^{'} = 1$.
 \end{proposition}
 \begin{proof}
 	
 	Given $x \in \mathcal{B}$, we define $$\|x\|^{'}= sup\{\|xy\| , \|y\| \leq 1\}$$ It is clear that $\|.\|^{'}$ is a submultiplicative norm on $\mathcal{B}$. Moreover, $$\|x\|= \|x.\frac{e}{\|e\|}\| \leq \|x\|^{'}$$ so that $\|e\|. \|x\|\leq \|x\|^{'}$. Therefore, if $(x_{n})$ Cauchy in $(\mathcal{B},\|.\|^{'})$ then $(x_{n})$ is Cauchy in $(\mathcal{B},\|.\|)$ hence convergent to some  $x$ in $\mathcal{B}$. Therefore $(x_n,y)$ converges to $xy$ and so $(x_{n})$converges to x in $\|.\|^{'}$. As $\|x\| \leq \frac{1}{\|e\|} \|x\|^{'}$, for all $x \in \mathcal{B}$, by an application of the open mapping theorem, the two norms are equivalent.
 \end{proof}\par
 For convenience, we memorize and rearrange the conditions needed for Banach algebra.
 \begin{definition}:
 	Let $(\mathcal{B},\|.\|)$ be a Banach space then it is a Banach algebra if there is a multiplication map $\mathcal{B} \times \mathcal{B} \longmapsto \mathcal{B}$ satisfying the following properties: $\forall x,y,z \in \mathcal{B}$ and $a \in \mathbb{C}$
 	\begin{enumerate}
 		\item $(xy)z=x(yz)$
 		\item $(x+y)z=xz+xy$
 		\item$x(y+z)=xy+xz$
 		\item$a(xy)=(ax)y=x(ay)$
 		\item$\|xy\| \leq \|x\| \|y\|$
 	\end{enumerate}
 	
 	Moreover, if there is a unit element e: $ex=xe=x$ $\forall x \in \mathcal{B}$ and $ \|e\| = 1$ then $\mathcal{B}$ is a unital Banach algebra.
 \end{definition}\par
 The following are some standard examples of Banach algebras in operator theory and functional spaces.
 \begin{example}\label{ex50}
 	Let $\mathcal{E}$ be a Banach space over $\mathbb{C}$ and denote by \begin{equation*}
 	\mathcal{L}(\mathcal{E}):=\{T:\mathcal{E} \longmapsto \mathcal{E} \mid T\ is\ linear\ and\ bounded\}.
 	\end{equation*} On $\mathcal{L}(\mathcal{E})$ the addition is given by point wise addition, the multiplication is a composition between two operators, i.e.\begin{enumerate}
 		\item $(T+S)(f)= Tf +Sf$
 		\item $(\lambda T)(f)= \lambda(Tf)$
 		\item $(T.S)f = T(Sf)$
 	\end{enumerate}
 	We consider the operator norm on $\mathcal{L}(\mathcal{E})$, i.e. \begin{equation*}
 	\|T\|_{op}=sup \{\frac{\|Tf\|}{\|f\|} \mid f \in \mathcal{E}\ and\ f \neq 0\}
 	\end{equation*} since $\mathcal{E}$ is Banach then $(\mathcal{L}(\mathcal{E}), \|.\|_{op})$ is a Banach space. Moreover the identity map on $\mathcal{E}$ is the unit element with $\|id\|_{op}=1$. Furthermore given $T.S$ in $\mathcal{L}(\mathcal{E})$ and $f \in \mathcal{E}\ with\ f \neq 0$, we have \begin{equation*}
 	\frac{\|T.Sf\|}{\|f\|} \leq sup_{Sf \neq 0} \frac{\|T.Sf\|}{\|Sf\|}.sup_{\|f\| \neq 0}\frac{\|Sf\|}{\|f\|}
 	\end{equation*} This shows that $\|TS\|_{op} \leq \|T\|_{op}\|S\|_{op}$ i.e. $\|.\|_{op}$ is submultiplicative. It follows that $(\mathcal{L}(\mathcal{E}),\|.\|_{op})$ is a Banach algebra. The fact that $\mathcal{L}(\mathcal{E})$ is Banach is standard since $\mathcal{E}$ is Banach 
 \end{example}\par
 The below example is a particular case of the previous one in the finite dimensional case.
 \begin{example}
 	Let $n \in \mathbb{N}$ and denote by $\mathcal{M}_{n}(\mathbb{C})$ the set of all square matrices whose entries are in $\mathbb{C}$. It is well known that $\mathcal{M}_{n}(\mathbb{C})$ coincides with the set of all endomorphisms on $\mathbb{C}^{n}$. Given a norm $\|.\|$ on $\mathbb{C}^{n}$, we consider the induced norm on $\mathcal{M}_{n}(\mathbb{C})$ given by \begin{equation*}
 	\|A\|_{ind}= sup_{x \neq 0} \frac{\|Ax\|}{\|x\|}
 	\end{equation*} In fact there is an isometric isomorphism between $(\mathcal{M}_{n}(\mathbb{C}),\|.\|_{ind})\ and\ (\mathcal{L}(\mathbb{C}^{n}),\|.\|_{op})$, hence $\mathcal{M}_{n}(\mathbb{C})$ is a unital Banach algebra with $e=I_n$ the identity matrix. In particular, $\mathbb{C}$ under the usual addition and multiplication by complex numbers (with the norm being the modulus) is a Banach algebra.
 \end{example}
 \begin{remark}
 	We consider the Frobenius norm $\|.\|_{F}$ on $\mathcal{M}_{n}(\mathbb{C})$ given by \begin{equation*}
 	\|A\|_{F}=\sqrt{\sum_{i=1}^{n}\sum_{j=1}^{n}|a_{ij}|^2},
 	\end{equation*} where $A=(a_{ij}) \in \mathcal{M}_{n}(\mathbb{C})$. The submultiplicativity of the Frobenius norm insures that $(\mathcal{M}_{n}(\mathbb{C}),\|.\|_{F})$ is a unital Banach algebra, but $\|e\|_{F} =\|In\|_{F} =\sqrt{n} >\ 1$ for $n\ >\ 1$.
 \end{remark}
 \begin{example}\label{ex4}
 	Let $\mathcal{W}$ be a non-empty set and $\mathcal{L}^{\infty}(\mathcal{W})$ be the space of all bounded complex valued functions defined on $\mathcal{W}$. Under the point wise addition and multiplication, $\mathcal{L}^{\infty}(\mathcal{W})$ is a unital Banach algebra with the supremum norm given by \begin{equation*}
 	\|f\|_{\infty}=\ sup_{x\in \mathcal{W}}|f(x)|
 	\end{equation*} Indeed, given a Cauchy sequence $\{f_{n}\}$ in $(\mathcal{L}^{\infty}(\mathcal{W}),\|.\|_{\infty})$ then for each $x \in \mathcal{W}$ we have $\{f_{n}(x)\}$ is Cauchy in $\mathbb{C}$ and hence convergent.\\
 	If we set $f(x)=\ \lim f_{n}(x)$, then following the boundedness of the Cauchy sequence $\{f_{n}\}$ we obtain the boundedness f and the convergence in the supremum norm.
 \end{example}
 \begin{definition}\label{d3}
 	Let $\mathcal{A}$ be an algebra and $\mathcal{A}_1$ be a non-empty subset of $\mathcal{A}$. $\mathcal{A}_1$ is called a subalgebra of $\mathcal{A}$ if it is an algebra in its own. In other words, $\mathcal{A}_1$ is a subalgebra if it is closed under the operations of scalar multiplication, addition and multiplication defined on $\mathcal{A}$. So that given $x,y\ \in \mathcal{A}_1$ and $\alpha \in \mathbb{C}$, we have:\begin{enumerate}
 		\item $x+y\in \mathcal{A}_1$
 		\item $\alpha x \in \mathcal{A}_1$
 		\item $xy\ \in \mathcal{A}_1$
 	\end{enumerate}
 \end{definition}
 \begin{remark}
 	It is direct that if $\mathcal{S}$ is a non-empty subset of an algebra $\mathcal{A}$, then the intersection of all subalgebras containing $\mathcal{S}$ is the minimal subalgebra containing $\mathcal{S}$.  It is often called the subalgebra generated by $\mathcal{S}$.
 \end{remark}\par
 The following remark is important for the construction of several Banach algebras.
 \begin{remark}
 	Given $(\mathcal{A},\|.\|)$  a Banach algebra and $\mathcal{A}_1$ a subalgebra of $\mathcal{A}$ which is closed in $(\mathcal{A},\|.\|)$, then $(\mathcal{A}_1,\|.\|_{\mathcal{A}_1})$ is also a Banach algebra.
 \end{remark}
 \begin{example}
 	Let $\mathcal{X}$ be a locally compact Hausdorff space, the space of all bounded and continuous functions on $\mathcal{X}$ with values in $\mathbb{C}$, denoted by $\mathcal{C}_b(\mathcal{X})$, is a closed subset of $\mathcal{L}^{\infty}(\mathcal{X})$. Hence $\mathcal{C}_b(\mathcal{X})$ is a unital algebra in its own. Note that if $\mathcal{X}$ is a compact subset of $\mathbb{R}$ $(or\ \mathbb{C})$, then $\mathcal{C}_b(\mathcal{X})$ coincides with $\mathcal{C}_0(\mathcal{X})$, the space of all continuous functions of $\mathcal{X}$.
 \end{example}\par
 Another example on the disk algebra.
 \begin{example}
 	Let $\mathcal{D}$ denote the unit open disk in $\mathbb{C}$. We write $\mathcal{H}(\mathcal{D})$ to denote the set of holomorphic functions in $\mathcal{D}$ and has continuous extension to $\overline{\mathcal{D}}$. As the uniform limit of holomorphic functions is also holomorphic, then $\mathcal{H}(\mathcal{D})$ is closed in $\mathcal{C}_b(\overline{\mathcal{D}})$ and hence a unital Banach algebra.
 \end{example}\par
 The uniform convergence doesn't ensure differentiability and thus we can not obtain a Banach algebra. This is ensured in the following example. 
 \begin{example}
 	$(\mathcal{C}^{'}([-1,1]),\|.\|_{\infty})$is not Banach. Indeed, consider the sequence of functions $\{f_{n}\} \subset \mathcal{C}^{'}([-1,1])$ given by \begin{equation*}
 	f_{n}(x)=\sqrt{\frac{1}{n^2}+x^2}
 	\end{equation*} The sequence converges uniformly to $|x|$. Hence $(\mathcal{C}^{'}([-1,1]),\|.\|_{\infty})$ is not closed in $\mathcal{L}^{\infty}([-1,1])$.
 \end{example}
 \begin{remark}\label{r567}
 	On $\mathcal{C}^{'}([-1,1])$, consider \begin{equation*}
 	\|f\|=\|f\|_{\infty}+\|f^{'}\|_{\infty}
 	\end{equation*} then $(\mathcal{C}^{'}([-1,1]),\|.\|)$ becomes a Banach algebra. We note that this norm is a submultiplicative norm since 
 	\begin{align*}
 	\|fg\|&=\|fg\|_{\infty}+\|(fg)^{'}\|_{\infty}\\
 	&\leq \|f\|_{\infty}.\|g\|_{\infty}+ \|f^{'}g+fg^{'}\|_{\infty}\\
 	&\leq \|f\|_{\infty}\|g\|_{\infty}+\|f^{'}\|_{\infty}\|g\|_{\infty}+\|f\|_{\infty}\|g^{'}\|_{\infty}\\
 	&\leq  \|f\|_{\infty}\|g\|_{\infty}+\|f^{'}\|_{\infty}\|g\|_{\infty}+\|f\|_{\infty}\|g^{'}\|_{\infty}+\|f^{'}\|_{\infty}\|g^{'}\|_{\infty}\\
 	&=(\|f\|_{\infty}+\|f^{'}\|_{\infty})(\|g\|_{\infty}+\|g^{'}\|_{\infty})\\
 	&=\|f\|.\|g\|.
 	\end{align*}
 \end{remark}
\section{The spectrum in Banach algebras}
This section is devoted to introduce the spectrum of an element in a unit Banach algebra. The following properties are to be used for the invariance under holomorphic function calculus.

 \begin{definition}
 	Let $\mathcal{A}$ be a unital algebra with unit e. Given $a \in \mathcal{A}$, we say that a is invertible if there exist $b \in \mathcal{A}$ such that $ab=ba=e$. The set of all invertible elements in $\mathcal{A}$ is denoted by $Inv(\mathcal{A})$, i.e.\begin{equation*}
 	Inv(\mathcal{A})=\{a \in \mathcal{A} \mid a\ is\ invertible\}.
 	\end{equation*}
 	 For a given $a \in Inv(\mathcal{A})$, we write $a^{-1}$ to denote the inverse of a.
 \end{definition}\par
 We recall now the definition of the spectrum of an element in a unital algebra
 \begin{definition}
 	Let $\mathcal{A}$ be a unital algebra with $e=e_{\mathcal{A}}$ the identity element. For a given $a \in \mathcal{A}$, the spectrum of a in $\mathcal{A}$ denoted by $\sigma_{\mathcal{A}}(a)$ is given by \begin{equation*}
 	\sigma_{\mathcal{A}}(a)=\{\lambda \in \mathbb{C}\mid \lambda e_{\mathcal{A}}-a \notin Inv(a)\}
 	\end{equation*}
 \end{definition}\par
 Below is a first example on the spectrum for the endomorphism on $\mathbb{C}^{n}$.
 \begin{example}
 	Given $\mathcal{A} \in \mathcal{M}_{n}(\mathbb{C})$, then \begin{align*}
 	\lambda I_n - \mathcal{A}\ is\ invertible &\Longleftrightarrow (\lambda I_n- \mathcal{A})\neq 0\\ &\Longleftrightarrow \lambda\ not\ eigenvalue\ of\ \mathcal{A}.
 	\end{align*} So that \begin{equation*}
 	\sigma_{\mathcal{M}_{n}(\mathbb{C})}(\mathcal{A})=\{\lambda\mid \lambda is\ an\ eigenvalue\ of\ \mathcal{A}\}.
 	\end{equation*}
 \end{example}
 \begin{example}
 	Let $\mathcal{X}$ be a compact subset of $\mathbb{C}$ and $f \in \mathcal{C}(\mathcal{X})$, then \begin{align*}
 	\lambda -\ f\ \in Inv(\mathcal{C}(\mathcal{X}))& \Longleftrightarrow \exists \ g \in \mathcal{C}(\mathcal{X})\ such\ that\ (\lambda-f)g=1\\ &\Longleftrightarrow g=\frac{1}{\lambda -f} \in \mathcal{C}(\mathcal{X}).
 	\end{align*}
 \end{example}\par
 Note that by Von-Neumann Lemma, we know that if $a \in \mathcal{B}$ satisfying $\|a\|<1$ then $(e-a) \in Inv(\mathcal{B})$ and $(e-a)^{-1} =\ \sum_{n=0}^{\infty}\ a^{n}$. The mentioned lemma together with the submultiplicative property of the norm on a unital Banach algebra insure that the set of invertible elements is open.
 \begin{proposition}
 	Let $(\mathcal{B},\|.\|)$ be a unital algebra, then $(Inv(\mathcal{B}),\|.\|)$ is open in $(\mathcal{B},\|.\|)$.
 	\begin{proof}
 	Given $a \in\ Inv(\mathcal{B})$, we claim that the ball $B(a,\frac{1}{\|a^{-1}\|})$ is contained in $Inv(\mathcal{B})$. In fact, if $\|b-a\|\ <\ \frac{1}{\|a^{-1}\|}$ then $\|a^{-1}\|\|b-a\|\ <\ 1$ so $\|a^{-1}b-e\|\ <\ 1$. Hence by Von-Neumann Lemma, $a^{-1}b$ is invertible and therefore $b \in Inv(\mathcal{B})$.
 	\end{proof}
 \end{proposition}
 \begin{remark}\label{r1}
 	Following the above proof and the fact that $e \in Inv(\mathcal{B})$, we proved $B(e,1) \subset\ Inv(\mathcal{B})$.
 \end{remark}\par
 The fact that $Inv(\mathcal{B})$ is open in $\mathcal{B}$ insures that the spectrum is a closed subset of $\mathbb{C}$. Moreover, the spectrum is a compact subset of $\mathbb{C}$ as stated in the below proposition.
 \begin{proposition}\label{p101}
 	Let $\mathcal{B}$ be a unital subalgebra. Given $a \in \mathcal{B}$, the spectrum of a, $\sigma_{\mathcal{B}}(a)=\sigma(a)$, is compact subset of $\mathbb{C}$.\\
 	\begin{proof}
 			We construct an isometry between the complex plane and the Banach algebra $\mathcal{B}$. Let $\varphi:\ \mathbb{C} \longrightarrow \mathcal{B}$ be the map defined by \begin{equation*}
 			\varphi(\lambda)=\lambda e-a
 			\end{equation*} then 
 			\begin{align*}
 			\|\varphi(\lambda)-\varphi(\mu)\|&=\|(\lambda e-a)-(\mu e-a)\|\\ &=\|(\lambda - \mu)e\|\ \\ &=\ |\lambda -\mu |.
 			\end{align*}
 		Hence $\varphi$ is an isometry and $\varphi^{-1}(Inv(\mathcal{B}))=\{\lambda\mid \lambda-a\ is\ invertible\}=\mathbb{C}$. As $Inv(\mathcal{B})$ is open then $\sigma_\mathcal{B}(a)$ is closed. Moreover, given $|\lambda|\ >\ \|a\|$ then by Von-Neumann Lemma, we know that $e-\lambda^{-1}a$ is invertible. Hence $\lambda \notin \sigma_\mathcal{B}(a)$ and so $\sigma_\mathcal{B} \subset B(0,\|a\|)$ which is bounded, so that $\sigma_\mathcal{B}(a)$ is closed and bounded, hence compact.
 	\end{proof}
 
 \end{proposition}\par

\chapter{Invariance under holomorphic calculus}
	In this chapter we study invertible elements in algebras and their relation to the invertibility in other algebras. The notion of algebras which are invariant under holomorphic calculus is motivated by the Wiener's Tauberian theorem. \par
	Through out the chapter  $\mathcal{B}$ will denote a unital Banach algebra.
 \section{Spectral invariant subalgebras}
  Given a Banach subalgebra $\mathcal{B}$ of a unital algebra $\mathcal{A}$. We are interested in the relationship of elements of $\mathcal{B}$ which are invertible in $\mathcal{A}$ and whether the inverse is in $\mathcal{B}$ or not. This is usually known as the notion of full sub algebras or closed invariant subalgebras. In the following , we aim to study the invertibility of elements in Banach spaces which has a tremendous importance in solving operator equations. Following a motivation from Wiener's theorem, it turns out that invertibility in one Banach algebra has a relationship to invertibility in another one. We start by the following observation.
 \begin{remark}
 	If $\mathcal{A}\ and\ \mathcal{B}$ are two unital algebras with $\mathcal{B} \subset \mathcal{A}$, then $e_{\mathcal{B}}$, the unit element of $\mathcal{B}$, need not to be the unit element of $\mathcal{A}$. For example, $\mathcal{B}=\{\begin{pmatrix} 
 	\lambda & 0 \\
 	0 & 0 
 	\end{pmatrix},\ \lambda \in \mathbb{C}\}$ is a sub algebra of $\mathcal{M}_{n}(\mathbb{C})$ whose unit is $\begin{pmatrix} 
 	1 & 0 \\
 	0 & 0 
 	\end{pmatrix}$.
 \end{remark}
 \begin{definition}\label{d1}
 	Let $(\mathcal{B},\|.\|_{\mathcal{B}})$ and $(\mathcal{A},\|.\|_{\mathcal{A}})$ be two Banach algebras with $\mathcal{B} \subset\mathcal{A}$ and having common unit e. We say that $\mathcal{B}$ is  spectral invariant in $\mathcal{A}$ if for any $x\ \in \mathcal{B}$ which is invertible in $\mathcal{A}$, the element x is also invertible in $\mathcal{B}$, that is \begin{equation*}
 	Inv(\mathcal{A})\cap \mathcal{B}\ =\ Inv(\mathcal{B})
 	\end{equation*} In other words, given $x \in \mathcal{B}$ such that $x^{-1}$ exists in $\mathcal{A}$, then $x^{-1} \in \mathcal{B}$.
 \end{definition}
 \begin{remark}
 	In $\mathcal{K}$ theory, the notion of spectral invariant subalgebra $\mathcal{B}$ of $\mathcal{A}$ is  replaced by the term $\mathcal{B}$ is closed inverse in $\mathcal{A}$ and the couple $(\mathcal{A},\mathcal{B})$ is called Wiener pair.
 \end{remark}
 \begin{example}
 	Let $\mathcal{B}$ denote the set of all upper triangular matrices in $\mathcal{A}=\mathcal{M}_{n}(\mathbb{C})$. As a product of two upper triangular matrices is also an upper triangular and $\mathcal{B}$ is a vector subspace of $\mathcal{A}$, then $\mathcal{B}$ is a subalgebra of $\mathcal{A}$. Moreover, $(\mathcal{A},\mathcal{B})$ is a Wiener pair. Indeed, Let $\mathcal{C}$ be an upper triangular matrix with $det(\mathcal{C}) \neq 0$. Using Gaussian elimination method, it follows that $\mathcal{C}^{-1}$ is also an upper triangular and hence $\mathcal{C}^{-1} \in \mathcal{B}$. This shows that $\mathcal{B}$ is spectral invariant in $\mathcal{A}$.
 \end{example}\par
 Below we provide an example of spectral invariant subalgebra $(\mathcal{B},\|.\|_{\mathcal{B}})$ in $(\mathcal{A},\|.\|_{\mathcal{A}})$ with $\|.\|_{\mathcal{B}}$ being induced from $\|.\|_{\mathcal{A}}$.
 \begin{example}
 	Let $\mathcal{B}=(\mathcal{C}^{'}([-1,1]),\|.\|)$ and $\mathcal{A}=(\mathcal{C}([-1,1]),\|.\|_{\infty})$ with $\|f\|=\|f\|_{\infty}+\|f^{'}\|_{\infty},\ f\ \in \mathcal{B}$ and $\|g\|_{\infty}=\ sup |g(x)|,\ x\in [-1,1]$, then $\mathcal{B}$ is spectral invariant in $\mathcal{A}$. Indeed, given $f \in \mathcal{B}$ such that $f(\epsilon) \neq 0\ \forall \epsilon \in [-1,1]$. Then $\frac{1}{f}$ is differentiable and $(\frac{1}{f})^{'}=-\frac{f^{'}}{f^2}$ is continuous, hence $\frac{1}{f} \in \mathcal{B}$. 
 \end{example}\par
 The following example provides a motivation for the study of the spectral invariant algebras and is due to Wiener.
 \begin{example}
 	Let $\mathcal{T}$ denotes the circle group, i.e. $\mathcal{T}=\mathbb{R}/2\pi \mathbb{Z}$ and consider $\mathcal{B}(\mathcal{T})$ the space of all functions whose Fourier series is absolutely convergent i.e. $f \in \mathcal{B}(\mathcal{T})$ if \begin{equation*}
 	f(t)= \sum_{k \in \mathbb{Z}}a_{k}e^{ikt}\ with\ a=(a_{k})_{k} \in l^1(\mathbb{Z})
 	\end{equation*} and $\|f\|= \|a\|_{1}=\sum|a_{k}|$.\\
 	Then $\mathcal{B}(\mathcal{T})$ is a Banach algebra under point wise multiplication. We remark the submultiplicativity follows from the convolution property:
 	\begin{align*}
 	 f(t)g(t)&=\sum_{k \in \mathbb{Z}}a_{k}e^{ikt}.\sum_{j \in \mathbb{Z}}b_{j}e^{ijt}\\
 	  &=\sum_{k,j \in \mathbb{Z}}a_{k}b_{j}e^{i(k+j)t}=\sum_{n \in \mathbb{Z}}\sum_{k \in \mathbb{Z}}a_{k}b_{n-b}e^{int} 
 	\end{align*}
 	So that \begin{align*}
 	\|fg\|&=\sum_{n \in \mathbb{Z}}|\sum_{k\in \mathbb{Z}}a_{k}b_{n-k}|\\
 	&\leq \sum_{n \in \mathbb{Z}}\sum_{k\in \mathbb{Z}}|a_{k}b_{n-k}|\\
 	&= \|fg\|.
 	\end{align*}
 	Consider $\mathcal{A}(\mathcal{T})=\mathcal{C}(\mathcal{T})$ the algebra of continuous functions on the unit circle. Wiener theorem shows that if $f \in \mathcal{B}(\mathcal{T})$ satisfying that f is nowhere zero, then $\frac{1}{f(t)}$ admits an absolutely convergent Fourier series. In other words, given $f \in \mathcal{B}(\mathcal{T})$ such that f is invertible in $\mathcal{A}(\mathcal{T})$ then f is also invertible in $\mathcal{B}(\mathcal{T})$. Thus $\mathcal{B}(\mathcal{T})$ is spectral invariant in $\mathcal{A}(\mathcal{T})$.
 \end{example}\par
 The spectral invariance of $\mathcal{B}$ in $\mathcal{A}$ is closely related to the spectrum in $\mathcal{B}$ and that in $\mathcal{A}$. This is given in the below theorem.
 \begin{theorem}\label{t3}
 	Let $(\mathcal{B},\|.\|_{\mathcal{B}})\ and\ (\mathcal{A}, \|.\|_{\mathcal{A}})$ be two Banach algebras with $\mathcal{B} \subset \mathcal{A}$ and having common unit e. Then $\mathcal{B}$ is spectral invariant in $\mathcal{A}$  if and only if $\sigma_{\mathcal{B}}(b)=\sigma_{\mathcal{A}}(b)$ for all $b \in \mathcal{B}$. i.e. given $b \in \mathcal{B}$ the spectrum of $b$ in $\mathcal{B}$ coincides with the spectrum of $b$ in $\mathcal{A}$.\\
 \begin{proof}
 		For the necessary condition, let $b \in \mathcal{B}$ be fixed. Given $\lambda \in \sigma_{\mathcal{B}}(b)$ then $(\lambda e-b) \notin Inv(\mathcal{B})$, hence $(\lambda e -b) \notin Inv(\mathcal{A})$ so $\lambda \in \sigma_{\mathcal{A}}(b)$. This shows $\sigma_{\mathcal{B}}(b) \subseteq \sigma_{\mathcal{A}}(b)$. Conversely, if $\lambda \in \sigma_{\mathcal{A}}(b)$ then $(\lambda e-b) \in Inv(\mathcal{A})$ but $(\lambda e-b) \in \mathcal{B}$ so $(\lambda e-b) \in Inv(\mathcal{B})$ and so $\sigma_{\mathcal{A}}(b) \subseteq \sigma_{\mathcal{B}}(b)$.\\
 	For the sufficient condition, given $b \in \mathcal{B}$ with $b \in Inv(\mathcal{A})$ then $0 \notin \sigma_{\mathcal{A}}(b)$, so  $0 \notin \sigma_{\mathcal{B}}(b)$ and therefore $b \in Inv(\mathcal{B})$.
 \end{proof}	
  
 \end{theorem}\par
Below is an application of the above theorem.
 \begin{example}
 	Let $\mathcal{A}$ be the algebra of meromorphic on the complex plane. Denote by $\mathcal{B}$ the sub algebra of $\mathcal{A}$ formed by all entire functions. Then the function $f(z)=z \in \mathcal{B}$ has an inverse in $\mathcal{A}$ and not invertible in $\mathcal{B}$, so that $\mathcal{B}$ is not spectral invariant . Th above fact can be proven by noting $\sigma_{\mathcal{B}}(f)= \phi\ and\ \sigma_{\mathcal{A}}(f)=\mathbb{C}$ because for any $\lambda \in \mathbb{C}$, $(\lambda - z)^{-1}$ is meromorphic but not entire.
 \end{example}\par
\section{Holomorphic Banach algebras valued functions}
 \hskip0.5cm In this section we are interested in holomorphic functions with Banach algebra valued functions. We shall apply these concepts to relate spectral invariance and the invariance under holomorphic functional calculus.
\begin{definition}
	Let $\mathcal{A}$ be a unital Banach algebra over $\mathbb{C}$ and $\mathcal{U}$ be an empty subset of $\mathbb{C}$. Given a function $f: \mathcal{U} \longrightarrow \mathcal{A}$, we say that $f$ is holomorphic on $\mathcal{U}$ if for each $z_{0} \in \mathcal{U}$ there is $r>0$ and $\{c_{n}\} \subset \mathcal{A}$ such that \begin{equation*}
	f(z)=\sum_{n=0}^{\infty}c_{n}(z-z_{0})^{n},\ where\ |z-z_{0}|<r.
	\end{equation*}
\end{definition}\par
In the below, we show that the resolvent operator $R_{a}$ is holomorphic on the resolvent  set $\rho(a):=\mathbb{C} \setminus \sigma(a)$ which is an open set by Proposition \ref{p101}.
\begin{definition}
	Let $\mathcal{A}$ be a unital algebra with identity e. Given $a \in \mathcal{A}$ we denote $\rho(a)$ the resolvent set of a, i.e. \begin{equation*}
	\rho(a)=\{\lambda \in \mathbb{C}\mid \lambda e-a \in Inv(\mathcal{A})\}
	\end{equation*} The resolvent operator is the map defined on $\rho(a)$ with values in the algebra $\mathcal{A}$ and is given by \begin{equation*}
	R_{a}(z)=(ze-a)^{-1},\ for\ z\in \rho(a)
	\end{equation*}
\end{definition}
\begin{proposition}
	The resolvent operator $R_a:\rho(a)\longrightarrow\mathcal{A}$ is a holomorphic map.\\
	\begin{proof}
		Fixed $z_{0}\in \rho(a)$, since $\sigma(a)$ is a closed set then for all $z \in \rho(a)$ with $|z-z_{0}|<\frac{1}{\|(ze-a)^{-1}\|}$ we obtain \begin{equation*}
		\|(z-z_{0})(z_{0}e-a)^{-1}\|<1
		\end{equation*} On the one hand, by Von-Neumann Lemma, we obtain $(e-(z_{0}-z)(z_{0}e-a)^{-1})$ is invertible with
		\begin{align*}
		(e-(z_{0}-z)(z_{0}e-a)^{-1})^{-1}&=\sum_{n=0}^{\infty}(z_{0}-z)^{n}[(z_{0}e-a)^{-1}]^{n}\\
		&=\sum_{n=0}^{\infty}(z_{0}-z)^{n}[R_{a}(z_{0})]^{n}
		\end{align*}
		On the other hand, we plug $z_{0}e$ in the resolvent operator as follows
		\begin{align*}
		R_{a}(z)&=(ze-a)^{-1}\\
		&=[z_{0}e-a-(z_{0}e-ze)]^{-1}\\
		&=[(z_{0}e-a)(e-(z_{0}-z)(z_{0}e-1)^{-1})]^{-1}\\
		&=[(e-(z_{0}-z)(z_{0}e-1)^{-1})]^{-1}(z_{0}e-a)^{-1}\\
		&=[\sum_{n=0}^{\infty}(z_{0}-z)^{n}(R_{a}(z_{0}))^n]R_{a}(z_{0})\\
		&=\sum_{n=0}^{\infty}(z_{0}-z)^{n}(R_{a}(z_{0}))^{n+1}
		\end{align*}
		Taking $c_{n}=(R_{a}(z_{0}))^{n+1}$ it follows that $R_{a}$ is holomorphic on $\rho(a)$.
	\end{proof}
\end{proposition}\par
We are now able to prove that the spectrum of any element is non-empty (and compact by Proposition \ref{p101}). For this we need a $\bf{Banach\ valued\ version}$ of Liouville's theorem.
\begin{theorem}
	Let $(\mathcal{A},\|.\|)$ be a Banach algebra. Given a Banach valued function $f:\mathbb{C}\longrightarrow \mathcal{A}$. Suppose that $f$ is bounded in the following sense: there is $M>0$ such that $\|f(z)\| \leq M$ for all $z \in \mathbb{C}$. Then there is $a_{0} \in \mathcal{A}$ such that $f(z)=a_{0}$ for all $z \in \mathbb{C}$ i.e. $f$ is a constant function. (cf. Theorem 3.1.2. in )
\end{theorem}
\begin{corollary}
	Let $(\mathcal{A},\|.\|)$ be a non-zero unital Banach algebra, the spectrum of any element in $\mathcal{A}$ is a non-empty subset of $\mathbb{C}$.\\
	
	\begin{proof}
		If $a=0$ then 
		\begin{equation*}
		\sigma(0)=\{\lambda \in \mathbb{C}\mid\lambda e \notin Inv(\mathcal{A})\}=\{0\}
		\end{equation*}
		For the case $a \neq 0$, we consider the resolvent operator $R_{a}$. If $\sigma(a)$ is empty, then by the previous proposition we obtain that $R_{a}$ is holomorphic on $\mathbb{C}$ i.e. it is an entire function. We show that $R_{a}$ is also bounded.\\
		It is clear that $R_{a}$ is bounded on the compact set $B[0,\|a\|]$ as it is a continuous function. For $|z|>\|a\|$ we have $\|\frac{a}{z}\|<1$ and hence
		\begin{align*}
		R_{a}(z)&=(ze-a)^{-1}\\
		&=\frac{1}{z}(e-\frac{a}{z})^{-1}\\
		&=\frac{1}{z}\sum_{n=0}^{\infty}\frac{a^{n}}{z^{n}}\\
		&=\sum_{n=0}^{\infty}\frac{a^{n}}{z^{n+1}}
		\end{align*}
		with \begin{align*}
		\|R_{a}(z)\| &\leq \frac{1}{|z|}\sum_{n=0}^{\infty}\|\frac{a}{z}\|^{n}\\
		&=\frac{1}{|z|}\frac{1}{1-\|\frac{a}{z}\|}\\
		&=\frac{1}{|z|-\|a\|}.
		\end{align*}
		This shows that $\lim_{|z|\to\infty} \|R_{a}(z)\|=0$. Therefore $R_{a}$ is bounded and hence constant by Liouville's theorem with $R_{a}\equiv 0$, which is a contradiction.
	\end{proof}
\end{corollary}
\begin{definition}
	Let $(\mathcal{A},\|.\|)$ be a Banach algebra, the spectral of $a \in \mathcal{A}$ is defined to be 
	\begin{equation*}
	\rho(a)= \max_{\lambda \in \sigma(a)}|\lambda|
	\end{equation*}
\end{definition}\par
As $\sigma(a)$ is non-empty then the above definition makes sense. Moreover, following the proof of Proposition \ref{p101} we know that $\sigma(a) \subset B(0,\|a\|)$. Therefore $\rho(a)$ is bounded by $\|a\|$. Furthermore, $\rho(a)$ is given by the limit of the norm of powers of $a$ (cf. \cite{ka}).
\begin{theorem}\label{t1}
	Let $(\mathcal{A},\|.\|)$ be a unital Banach algebra and $a \in \mathcal{A}$ then the spectral radius of a is given by 
	\begin{equation}\label{e5}
	\rho(a)=\lim_{n\to\infty} \|a^{n}\|^{\frac{1}{n}}
	\end{equation}
\end{theorem}
\section{The holomorphic functional calculus}
\hskip0.5cm The holomorphic functional calculus is an important tool in operator theory and functional analysis. It allows us to produce a wide class of operators from a given one. We start by the following motivation.\par
Consider the complex Hilbert space $\mathcal{H}$ and $\mathcal{T} \in \mathcal{L}(\mathcal{H})$, let $P(z)=\sum_{k=0}^{n}a_{k}z^{k}$ be a polynomial in $\mathbb{C}$. We can define 
\begin{equation*}
P(\mathcal{T})= a_{0}Idn+a_{1}\mathcal{T}+\cdot+a_{n}\mathcal{T}^{n},
\end{equation*}
where $\mathcal{T}^{j}$ is the composite of $\mathcal{T}$ with itself $j$-times. As $\mathcal{L}(\mathcal{H})$ is an algebra over $\mathbb{C}$, we then have $P(\mathcal{T}) \in \mathcal{L}(\mathcal{H})$. We note here that P is an entire function and there is no restriction for the definition of $P(\mathcal{T})$.\par
Now consider another example $Q(z)=\frac{1}{z}$. It is natural to assign $Q(\mathcal{T}) \in \mathcal{L}(\mathcal{H})$ to be the inverse of $\mathcal{T}$. This shows that $\mathcal{T}$ should be invertible, i.e. $0 \notin \sigma(\mathcal{T})$. However, Q is holomorphic on $\mathbb{C} \setminus \{0\}$, i.e. Q is holomorph in a neighborhood of $\sigma(\mathcal{T})$.\\
The above discussion encourage us to define $f(\mathcal{T}) \in \mathcal{L}(\mathcal{H})$ and $f$ being holomorph in a neighborhood of $\sigma(\mathcal{T})$ in an abstract way.\par
We start by introducing the curvilinear integral of Banach valued functions.
\begin{definition}
	Let $\mathcal{U}$ be an open subset of $\mathbb{C}$, given $\mu$ a regular curve in $\mathcal{U}$ and $\gamma:[a,b] \longmapsto \mathbb{C}$ be a $\mathbb{C}^{\infty}$ parametrization of $\mu$. Suppose $f:\mathcal{U}\longmapsto\mathcal{A}$ is a continuous Banach algebra valued function, we define the integral of f along $\mu$ as follows:
	\begin{equation*}
	\int_{\mu} f(z)dz= \lim_{n\to\infty} \sum_{k=1}^{n}[\gamma(x_{k})-\gamma(x_{k-1})]f(\gamma(x_{k}))
	\end{equation*}
	where $x_{k}=\frac{k(b-a)}{n}$ is a uniform partition of $[a,b]$.\end{definition}
	Note that, the function $h:[a,b]\longmapsto \mathbb{R}$ defined by 
	\begin{equation*}
	h(x):=\|f(\gamma(x))\|
	\end{equation*}
	is a continuous map and so the limit of the Riemann sum exists, hence the above integral exists in $\mathcal{A}$.

\begin{remark}
	In case $\mu$ is a piecewise smooth curve i.e. $\mu = \gamma_{1} \cup \gamma_{2}\cup...\cup\gamma_{n}$, where $\gamma_{i}$ are regular, we define
	\begin{equation*}
	\int_{\mu}f(z) dz :=\sum_{i=1}^{n}\int_{\mu_{i}}f(z) dz
	\end{equation*}
\end{remark}\par
We move now to the collection of germs of holomorphic functions and fix some notation.\par
Given a unital Banach algebra $\mathcal{A}$, and an element $a \in \mathcal{A}$. We denote by $\mathcal{H}^{1}_{\mathcal{A}}(a)$ the collection of germs of holomorphic functions $f$ defined on an open neighborhood of the spectrum of $a$. i.e.
\begin{equation*}
\mathcal{H}^{1}_{\mathcal{A}}:=\{f:D_{f} \longrightarrow \mathcal{A}\mid D_{f}\ is\ an\ open\ subset\ of\ \mathbb{C}\ with\ \sigma_{\mathcal{A}}(a)\subset D_{f} \}.
\end{equation*}\par
The operations on $\mathcal{H}^{1}_{\mathcal{A}}(a)$ are given by:
\begin{enumerate}
	\item $(f+g)(z)=f(z)+g(z), \quad z \in D_{f}\cap D_{g}$.
	\item $(fg)(z)=f(z).g(z), \quad z \in D_{f}\cap D_{g}$.
	\item $(\lambda f)(z)=\lambda f(z), \quad z \in D_{f}\ and\ \lambda \in \mathbb{C}$.
\end{enumerate}
\begin{remark}
In order to obtain an algebra of the germs using the above operations we need to define an equivalence relation as follows: $f \sim g$ if there is $\mathcal{U}$ open containing $\sigma_{A}(a)$ with 
\begin{equation*}
f \mid_{ \mathcal{U}}=g \mid_{ \mathcal{U}}.
\end{equation*}
It follows that $\mathcal{H}_{\mathcal{A}}(a):=\mathcal{H}_{\mathcal{A}}^{1}(a)/\sim$ is an algebra and is called the algebra of holomorphic functions around $\sigma_{\mathcal{A}}(a)$.
\end{remark}\par
Before introducing the holomorphic functional calculus, we need the following result from complex analysis which can be found in (cf. \cite{ka} Lemma 3.1.2 ), we state it in a way to fit our needs.
\begin{lemma}\label{l1}
	Let $\mathcal{A}$ ba a unital Banach algebra and let $a \in \mathcal{A}$. Given D an open neighborhood of $\sigma(a)$. There are piecewise closed smooth curves $\gamma_{1},...,\gamma_{n} \subset D\setminus \sigma(a)$ satisfying the following conditions: given $f:D \longmapsto \mathbb{C}$ holomorphic, then
	\begin{equation*}
	f(z)=\frac{1}{2\pi i}\sum_{j=1}^{n}\int_{\gamma_{j}}\frac{f(\xi)}{\xi -z} d\xi 
	\end{equation*}
\end{lemma}\par
Now let $a \in \mathcal{A}$ and $f \in \mathcal{H}_{\mathcal{A}}(a)$ we want to associate another element in $\mathcal{A}$ using $f$.
\begin{definition}
	Given a unital Banach algebra $\mathcal{A}$ with $a \in \mathcal{A}$, let D be an open subset of $\mathbb{C}$ containing $\sigma(a)$ and $f:D\longmapsto \mathbb{C}$. we define $\tilde{f}(a) \in \mathcal{A}$ as follows
	\begin{equation*}
	\tilde{f}(a)=\frac{1}{2\pi i}\sum_{j=1}^{n}\int_{\gamma_{j}}f(z)(ze-a)^{-1} dz
	\end{equation*}
	where $\{\gamma_{j} \}_{j=1,...,n}$ are any closed curves satisfying the above lemma.
\end{definition}\par
The below theorem motivates the above definition and clarifies the connection to the discussion done at the beginning of this section.
\begin{theorem}
	Let $(\mathcal{A},\|.\|)$ be a unital Banach algebra and $a \in \mathcal{A}$, given an entire function with $f(z)=\sum_{n=0}^{\infty}c_{n}z^{n}$, $z \in \mathbb{C}$, we have
	\begin{equation*}
	\tilde{f}(a)=\sum_{n=0}^{\infty}c_{n}a^{n} \in \mathcal{A}
	\end{equation*}\\
	\begin{proof}
		Following the proof of Proposition \ref{p101}, we know that $\sigma(a) \subset B(0,\|a\|)$. We choose $\gamma$ to be the circle in $\mathbb{C}$ centered at 0 and radius $r>\|a\|$. Hence $\gamma$ satisfies  Lemma \ref{l1}, and so
		\begin{align*}
		\tilde{f}(a)&=\frac{1}{2\pi i}\int_{\gamma}f(z)(ze-a)^{-1} dz\\
		&=\frac{1}{2\pi i}\int_{\gamma}\sum_{n=0}^{\infty}c_{n}z^{n}(ze-a)^{-1} dz\\
		&=\frac{\sum_{n=0}^{\infty}c_{n}}{2\pi i}\int_{\gamma}\frac{z^{n}}{z}(e-\frac{a}{z})^{-1} dz\\
		&=\sum_{n=0}^{\infty}\frac{c_{n}}{2\pi i}\sum_{k=0}^{\infty}\int_{\gamma}z^{n-1-k}a^{k} dz\\
		&=\sum_{n=0}^{\infty}\frac{c_{n}}{2\pi i}\sum_{k=0}^{\infty}a^{k}\int_{\gamma}z^{n-1-k} dz
		\end{align*}
		Using Cauchy's integral formula for entire functions $g$:
		\begin{equation*}
		g^{n}(z_{0})=\frac{n!}{2\pi i}\int_{\gamma}\frac{g(\xi)}{\xi-z_{0}} d\xi,
		\end{equation*}
		we obtain
		\begin{equation*}
		\frac{1}{2\pi i}\int_{\gamma}\frac{dz}{z^{l}}=
		\begin{cases}
		1, &\mbox{If}\ l=1\\0, &\mbox{elsewhere}
		\end{cases}
		\end{equation*}
		Hence the double sum in the above series reduces to
		\begin{equation*}
		\tilde{f}(a)=\sum_{n=0}^{\infty}\frac{c_{n}}{2\pi i}\int_{\gamma}\frac{a^{n}}{z} dz =\sum_{n=0}^{infty}c_{n}a^{n}
		\end{equation*}
	\end{proof}
\end{theorem}\par
A slight modification of the above theorem is given below.
\begin{proposition}
	Let $(\mathcal{A},\|.\|)$ be a unital Banach algebra and $a \in \mathcal{A}$. Let $f \in \mathcal{H}_{\mathcal{A}}(a)$ and suppose $f: D_{f}\longrightarrow\mathbb{C}$ is holomorphic such that $D_{f}$ contains a disk $D(0,r)$ with $r<\rho(a)$, where $\rho(a)$ is the spectral radius of a. Assume that $f$ has a power series representation $f(z)=\sum_{n=0}^{\infty}c_{n}z^{n}$ with radius of convergence $r_{1}\leq r$ then $\tilde{f}(a)=\sum_{n=0}^{\infty}c_{n}a^{n} \in \mathcal{A}$.\\
	\begin{proof}
		The proof is identical to the previous one. We only need to check the normal convergence of $\sum_{n=0}^{\infty}c_{n}a^{n}$. We apply the $n^{th}$ root test together with Theorem \ref{t1}
		\begin{align*}
		\overline{\lim_{n\to\infty}}\|c_{n}a^{n}\|^{\frac{1}{n}}&=\overline{\lim_{n\to\infty}}|c_{n}|^{\frac{1}{n}}\|a^{n}\|^{\frac{1}{n}}\\
		&=\frac{1}{r_{1}}\rho(a)\\
		&\leq \frac{1}{r}\rho(a)<1
		\end{align*}
		Hence the series converges normally and the interchange of the integration and summation is then guaranteed.
	\end{proof}
\end{proposition}
\begin{remark}
	Let $\mathcal{A}$ be a unital Banach algebra with unit e. Taking $f\equiv1$ on $\mathbb{C}$ we obtain from the above theorem that $\tilde{f}(a)=e\ for\ all\ a\in \mathcal{A}$.
\end{remark}
\begin{example}
	Let $\mathcal{A}=\mathcal{M}_{2}(\mathbb{C})$ be the algebra of $2\times2$ matrices. Let $a=\begin{pmatrix} \lambda & 0 \\ 0 & \mu \end{pmatrix} \in \mathcal{A}$ be a fixed diagonal matrix. Consider the exponential map on $\mathbb{C}$, $f(z)=e^{z}$ which is entire function. With $f(z)=\sum_{n=0}^{\infty}\frac{z^{n}}{n!}$ then $f \in \mathcal{H}_{\mathcal{A}}(a)$ with
	\begin{align*}
	\tilde{f}(a)&=\tilde{f}(\begin{pmatrix} \lambda & 0 \\ 0 & \mu \end{pmatrix})=\sum_{n=0}^{\infty}\frac{a^{n}}{n!}\\
	&=\sum_{n=0}^{\infty}\frac{\begin{pmatrix} \lambda & 0 \\ 0 & \mu \end{pmatrix}^{n}}{n!}=\sum_{n=0}^{\infty}\frac{\begin{pmatrix} \lambda^{n} & 0 \\ 0 & \mu^{n} \end{pmatrix}}{n!}\\
	&=\sum_{n=0}^{\infty}\begin{pmatrix} \frac{(\lambda)^{n}}{n!} & 0 \\ 0 & \frac{(\mu)^{n}}{n!} \end{pmatrix}=\begin{pmatrix} \sum_{n=0}^{\infty}\frac{(\lambda)^{n}}{n!} & 0 \\ 0 & \sum_{n=0}^{\infty}\frac{(\mu)^{n}}{n!} \end{pmatrix}\\
	&=\begin{pmatrix}
	e^{\lambda} & 0 \\ 0 & e^{\mu}
	\end{pmatrix}.
	\end{align*}
\end{example}\par
Below is the definition of the holomorphic calculus
\begin{definition}
	Let $(\mathcal{A},\|.\|)$ be a unital Banach algebra and $\varPhi_{\mathcal{A},a}:\mathcal{H}_{\mathcal{A}}(a)\longmapsto\mathcal{A}$ be the map defined by
	\begin{equation*}
	\varPhi_{\mathcal{A},a}(f)=\tilde{f}(a),\ f \in \mathcal{H}_{\mathcal{A}}(a)
	\end{equation*}
	We call $\varPhi_{\mathcal{A},a}$ the holomorphic functional calculus over a.
\end{definition}\par
\section{Invariance under holomorphic functional calculus}
We introduce the notion of subalgebra which are closed under holomorphic functional calculus and we prove that this notion coincides with the spectral invariance notion.
\begin{definition}
	Let $(\mathcal{A},\|.\|_{\mathcal{A}})$ and $(\mathcal{B},\|.\|_{\mathcal{B}})$ be two Banach algebras with $\mathcal{B}\subset\mathcal{A}$ and having common unit e. We say that $\mathcal{B}$ is closed under holomorphic functional calculus if for every $a \in \mathcal{B}$ the image of $\varPhi_{\mathcal{A},a}$ is entirely in $\mathcal{B}$ i.e. given $a \in \mathcal{B}$ we have $\varPhi_{\mathcal{A},a}:\mathcal{H}_{\mathcal{A}}(a) \longrightarrow \mathcal{B}$.
\end{definition}
\begin{example}
	Let $\mathcal{A}=\mathcal{M}_{2}(\mathbb{C})$ and $\mathcal{B}=\{\begin{pmatrix}
	\alpha & 0 \\ 0 & \beta
	\end{pmatrix}, \alpha,\beta \in \mathbb{C}\}$ be the algebra of diagonal matrices. We want to prove $\mathcal{B}$ is closed under holomorphic functional calculus in $\mathcal{A}$. Fix $a=\begin{pmatrix}
	\alpha & 0 \\ 0 & \beta
	\end{pmatrix} \in \mathcal{B}$, then $\sigma(a)=\sigma[\begin{pmatrix}
	\alpha & 0 \\ 0 & \beta
	\end{pmatrix}]=\{\alpha,\beta \}$. Let f be holomorph on D, where D is open and $\{\alpha,\beta \}\subset D$. Let $\gamma\subset D\setminus \{\alpha,\beta \}$ be a closed curve then
	\begin{align*}
	\frac{1}{2\pi i}\int_{\gamma}f(z)(zI_{2}-a)^{-1} dz &=\frac{1}{2\pi i}\int_{\gamma}f(z)(z\begin{pmatrix} 1 & 0 \\ 0 & 1 \end{pmatrix} - \begin{pmatrix} \alpha & 0 \\ 0 & \beta \end{pmatrix})^{-1} dz\\
	&=\frac{1}{2\pi i}\int_{\gamma}f(z)\begin{pmatrix} z-\alpha & 0 \\ 0 & z-\beta \end{pmatrix}^{-1} dz\\
	&=\frac{1}{2\pi i}\int_{\gamma}f(z)\begin{pmatrix} \frac{1}{z-\alpha} & 0 \\ 0 & \frac{1}{z-\beta} \end{pmatrix} dz\\
	&=\frac{1}{2\pi i}\int_{\gamma}\begin{pmatrix} \frac{f(z)}{z-\alpha} & 0 \\ 0 & \frac{f(z)}{z-\beta} \end{pmatrix} dz\\
	&=\frac{1}{2\pi i}\lim_{n\to\infty}\sum_{j=1}^{n}[\gamma(x_{k})-\gamma(x_{k-1})]\begin{pmatrix} \frac{f(\gamma(x_{k}))}{\gamma(x_{k})-\alpha} & 0 \\ 0 & \frac{f(\gamma(x_{k}))}{\gamma(x_{k})-\beta} \end{pmatrix}\\
	&=\frac{1}{2\pi i}\lim_{n\to\infty}\sum_{j=1}^{n}\begin{pmatrix} \frac{(\gamma(x_{k})-\gamma(x_{k-1}))}{\gamma(x_{k})-\alpha}f(\gamma(x_{k})) & 0 \\ 0 & \frac{(\gamma(x_{k})-\gamma(x_{k-1}))}{\gamma(x_{k})-\beta}f(\gamma(x_{k})) \end{pmatrix}
	\end{align*}
as	$\begin{pmatrix} \frac{(\gamma(x_{k})-\gamma(x_{k-1}))}{\gamma(x_{k})-\alpha}f(\gamma(x_{k})) & 0 \\ 0 & \frac{(\gamma(x_{k})-\gamma(x_{k-1}))}{\gamma(x_{k})-\beta}f(\gamma(x_{k})) \end{pmatrix} \in \mathcal{B}$\\

\bigskip
then $\sum_{j=1}^{n}\begin{pmatrix} \frac{(\gamma(x_{k})-\gamma(x_{k-1}))}{\gamma(x_{k})-\alpha}f(\gamma(x_{k})) & 0 \\ 0 & \frac{(\gamma(x_{k})-\gamma(x_{k-1}))}{\gamma(x_{k})-\beta}f(\gamma(x_{k})) \end{pmatrix} \in \mathcal{B}$,\\

\bigskip
 hence $\lim_{n\to\infty}\sum_{j=1}^{n}\begin{pmatrix} \frac{(\gamma(x_{k})-\gamma(x_{k-1}))}{\gamma(x_{k})-\alpha}f(\gamma(x_{k})) & 0 \\ 0 & \frac{(\gamma(x_{k})-\gamma(x_{k-1}))}{\gamma(x_{k})-\beta}f(\gamma(x_{k})) \end{pmatrix} \in \mathcal{B}$\\
 
 \bigskip
 and therefore $\frac{1}{2\pi i}\int_{\gamma}f(z)(zI_{2}-a)^{-1} dz \in \mathcal{B}.$
\end{example}\par 
The spectral invariance property, introduced in Section 3.2, which is a relation between the invertible elements of $\mathcal{B}$ and those in $\mathcal{A}$ (cf. Definition \ref{d1}) give the rise for the invariance under the holomorphic functional calculus. Moreover, the two notions are the same.
\begin{theorem}\label{t2}
	Let $(\mathcal{A},\|.\|_{\mathcal{A}})$ and $(\mathcal{B},\|.\|_{\mathcal{B}})$ be two Banach algebras with $\mathcal{B} \subset \mathcal{A}$ and having common unit e. Then $\mathcal{B}$ is spectral invariance in $\mathcal{A}$ if and only if $\mathcal{B}$ is closed under holomorphic functional calculus.\\
	\begin{proof}
		We start by proving the necessary condition. Let $a \in \mathcal{B}$ be fixed and f be a holomorphic in $\sigma_{\mathcal{A}}(a)$ i.e. $f \in \mathcal{H}_{\mathcal{A}}(a)$. As $\mathcal{B}$ is spectral invariant in $\mathcal{A}$ then by Theorem \ref{t3}, we know that $\sigma_{\mathcal{A}}(a)=\sigma_{\mathcal{B}}(a)$ therefore f is holomorph in a neighborhood of $\sigma_{\mathcal{B}}(a)$, i.e. $f \in \mathcal{H}_{\mathcal{A}}(a)$. Therefore 
		\begin{equation*}
		\varphi_{\mathcal{A},a}(f)=\tilde{f}(a)=\varphi_{\mathcal{B},a}(f) \in \mathcal{B}
		\end{equation*}
		because $z\longmapsto f(z)(ze-a)^{-1} \in \mathcal{B}$ then $\frac{1}{2\pi i}\int_{\gamma}f(z)(ze-a)^{-1} dz \in \mathcal{B}$.\\
		For the sufficient condition, we fix $a \in Inv(\mathcal{A})\cap\mathcal{B}$ and prove that $a \in Inv(\mathcal{B})$. Since a is invertible in $\mathcal{A}$ then $0 \notin \sigma_{\mathcal{A}}(a)$, let $f:\mathbb{C}\setminus\{0\} \longrightarrow \mathbb{C}$ be the inversion map i.e. $f(z)=\frac{1}{z}$, then f is holomorph in a neighborhood of $\sigma_{\mathcal{A}}(a)$. This shows that $f \in \mathcal{H}_{\mathcal{A}}(a)$. As $a \in \mathcal{B}$, by the invariance under holomorphic functional calculus, we have $\tilde{f}(a)=a^{-1} \in \mathcal{B}$ so $a \in Inv(\mathcal{B})$.
	\end{proof}
\end{theorem}\par
 Let $(\mathcal{A},\|.\|_{\mathcal{A}})$ and $(\mathcal{B},\|.\|_{\mathcal{B}})$ be two Banach algebras with $\mathcal{B} \subset \mathcal{A}$ and having common unit $e$. We say  that $\mathcal{B}$ is locally spectral invariant in $\mathcal{A}$ if there is a positive number $r$ satisfying 
\begin{equation*}
\mathcal{B} \cap B_{\|.\|_{\mathcal{A}}}(e,r) \subset Inv(\mathcal{B})
\end{equation*}
where $B_{\|.\|_{\mathcal{A}}}(e,r)=\{f \in \mathcal{A},\ \|f-e\|<r \}$ is the open ball in $\mathcal{A}$ centered at e with radius r.
\begin{remark}
	If $\mathcal{B}$ is closed under holomorphic functional calculus of $\mathcal{A}$, then $\mathcal{B}$ is locally spectral invariant in $\mathcal{A}$. Indeed, by Remark \ref{r1} , we know that $B_{\|.\|_{\mathcal{A}}}(e,1) \subset Inv(\mathcal{A})$. hence by the previous theorem, we obtain
	\begin{equation*}
	\mathcal{B} \cap B_{\|.\|_{\mathcal{A}}}(e,1) \subset Inv(\mathcal{A}) \cap \mathcal{B}=Inv(\mathcal{B}).
	\end{equation*} 
	Finally we remark that the notion of locally spectral invariance coincides with the spectral invariance of dense subalgebra as proved in [ref[69]of paper diss]
\end{remark}

\chapter{More types of algebras}
In the following chapter we review basics on Fr\'{e}chet  algebras and on $C^*$-algebras. We then introduce the notion of $\Psi^*$-algebras and end by recalling Gramsch commutator method for the construction of $\Psi^*$-algebras.\\
\section{Fr\'{e}chet  spaces and Fr\'{e}chet  algebras}
\hskip0.5cm This section is devoted to introduce the notion of Fr\'{e}chet  algebras which are generalization of Banach algebras. We start by recalling the notion of Fr\'{e}chet  spaces.
\begin{definition}
	Let $\mathcal{E}$ be a vector space over $\mathbb{C}$ and suppose $Q=\{q_{n}\}_{n\in \mathbb{N}}$ is a family of semi-norms on $\mathcal{E}$. We define the topology induced by Q, denoted by $\tau(Q)$ as follows:\\
		Given $(x_{\alpha}) \in \mathcal{E}\ and\ x_{0} \in \mathcal{E}$, we say that $x_{\alpha}$ converges to $x_{0}$ if for any $n \in \mathbb{N}$, we have $q_{n}(x_{\alpha}-x_{0})\longmapsto 0$.
\end{definition} \par
The above topology turns $(\mathcal{E},\tau(Q))$ into a topological vector space.\par
A basis for the topology $\tau(Q)$ is given in the following theorem \cite{mor}.
\begin{theorem}
	Let $Q=\{q_{n}\}_{n \in \mathbb{N}}$ be a family of semi-norms  on $\mathcal{E}$, then a basis for $\tau(Q)$ is given by
	\begin{equation*}
	\{B_{q_{n_{1}}}(x,r_{1})\cap B_{q_{n_{2}}}(x,r_{2})\cap ...\cap B_{q_{n_{t}}}(x,r_{t}),x \in \mathcal{E}\ and\ n_{1},...,n_{t}\in \mathbb{N}, r_{1},...,r_{t}>0\}.
	\end{equation*}
	Here $B_{q_{n_{j}}}(x,r_{j})$ denotes the ball centered at x radius $r_{j}$ in the $q_{n_{j}}$-semi -norm. i.e. $$B_{q_{n_{j}}}(x,r_{j})=\{y\in \mathcal{E}\mid q_{n-{j}}(x-y)<r_{j}\}.$$
\end{theorem}
\begin{remark}
	We can replace the collection $Q=\{q_{n}\}_{n\in \mathbb{N}}$ by the equivalent semi-norms $P=\{p_{n}\}_{n\in \mathbb{N}}$ where $p_{n}=\max\{q_{1},...q_{n}\}$. This allows us to obtain the same topology on $\mathcal{E}$ i.e. $\tau(Q) \equiv \tau(P)$ and to gain the advantage of having an increasing family of semi-norms on $\mathcal{E}$. For this reason we will always assume that the family $Q=\{q_{n}\}_{n \in \mathbb{N}}$ of semi-norms satisfies $q_{1} \leq q_{2} \leq\cdots$
\end{remark}\par
The below lemma shows that $(\mathcal{E},\tau(Q))$ is metrizable whenever it is locally convex Hausdorff (cf. \cite{hor1}).
\begin{lemma}
	Let $\{q_{n}\}_{n\in \mathbb{N}}$ be an increasing family of semi-norms on a vector space $\mathcal{E}$ and assume $\mathcal{E}$ is a locally convex Hausdorff space. Then the map $d:\mathcal{E} \times \mathcal{E} \longmapsto \mathbb{R}_{+}$ given by 
	\begin{equation*}
	d(x,y)=\sum_{n=0}^{\infty}\frac{1}{2^{n}}\frac{q_{n}(x-y)}{1+q_{n}(x-y)}
	\end{equation*}
	defines a metric on $\mathcal{E}$ satisfying $d(x+z,y+z)=d(x,y)$ for all $x,y\ and\ z \in \mathcal{E}$ with $\tau_{d} \equiv \tau(Q)$.
\end{lemma}\par
Due to the above lemma we motivate the following definition for  Fr\'{e}chet  space.
\begin{definition}\label{d6}
	Let $Q=\{q_{n}\}_{n \in \mathbb{N}}$ be a family of semi-norms on $\mathcal{E}$. We say that $\mathcal{E}$ is a Fr\'{e}chet  space if $(\mathcal{E},\tau(Q))$ is complete and Hausdorff.
\end{definition}
\begin{remark}
	The Hausdorff property is equivalent to: given $x \in \mathcal{E}$ with $p_{n}(x)=0$ for all $n \in \mathbb{N}$ then $x=0$ i.e.
	\begin{equation*}
	\cap_{n \in \mathbb{N}}\{x \in \mathcal{E};p_{n}(x)=0\}=0_{\mathcal{E}}
	\end{equation*}
\end{remark}
\begin{remark}
	By a Cauchy sequence $(x_{n})_{n\in \mathbb{N}}$ we mean, given any neighborhood $\mathcal{U}\ of\ 0_{\mathcal{E}}$, there is $n_{0} \in \mathbb{N}$ (depending on $\mathcal{U}$) such that $x_{n}-x_{m} \in \mathcal{U}$ for all $n,m >n_{0}$ (cf. \cite{sch} for the sequentially completeness). Following the previous lemma, the definition of the Cauchy criteria coincides with $(x_{n})_{n \in \mathbb{N}}$ being Cauchy in $(\mathcal{E},\tau_{d})$ i.e. given $\epsilon > 0$ there is $n_{0}=n_{0}(\epsilon) \in \mathbb{N}$ such that $d(x_{n},x_{m})<\epsilon$ for all $n,m \geq n_{0}$.
\end{remark}
\begin{example}\label{ex1}
	Every Banach space is a Fr\'{e}chet  space.
\end{example}
\begin{example}\label{ex2}
	Consider the Schwartz space over $\mathbb{R}^{N}$ i.e.
	\begin{equation*}
	S(\mathbb{R}^{N})=\{f\in C^{\infty}(\mathbb{R}^{N})\mid for\ each\ \alpha,\beta \in \mathbb{N}^{N}\ we\ have\ \sup_{x \in \mathbb{R}^N} |x^{\alpha}D^\beta f(x)|<\infty\}.
	\end{equation*}
	Consider the collection $Q=\{q_{n}\}_{n\in \mathbb{N}}$ of semi-norms on $S(\mathbb{R}^{N})$ given by 
	\begin{equation*}
	q_{n}(f)=\sup_{|\alpha|\leq n}\sup_{x \in \mathbb{R}^{N}}|(1+|x|^{2})^{n}D^{\alpha}f(x)|,\ f \in S(\mathbb{R}^{N})
	\end{equation*}
	then $(S(\mathbb{R}^{N}),\tau(Q)$ is a Fr\'{e}chet  space (cf. \cite{bon,kin}).
\end{example}\par
We have the following theorem which can be found in Zelazko.
\begin{theorem}\label{t4}
	Let $\mathcal{E}$ (respectively $\mathcal{F}$) be two Fr\'{e}chet  spaces whose topology is given by an increasing family of semi-norms $\{p_{n}\}_{n \in \mathbb{N}}$ (respectively $\{q_{n}\}_{n \in \mathbb{N}}$). Given $T:\mathcal{E} \longrightarrow \mathcal{F}$  a linear map, then $T$ is continuous if and only if for each semi norm $q_{n_{0}}$ there is a semi-norm $p_{n_{1}}$ and a positive number $c$ such that 
	\begin{equation*}
	q_{n_{0}}(T(f)) \leq c\ p_{n_{1}}(f),\ for\ all\ f \in \mathcal{E}.
	\end{equation*}
\end{theorem}\par
\begin{definition}
	Let $\mathcal{A}$ be a unital algebra  and $Q=\{q_{n}\}_{n \in \mathbb{N}}$ be an increasing family of semi-norms for which $(\mathcal{A},\tau(Q))$ is a Fr\'{e}chet  space and suppose that the multiplication is (separately or adjointly) continuous i.e. given $n \in \mathbb{N}$, there is $c_{n}>0\ with\ n_{0} \geq n$ satisfying
	\begin{equation*}
	q_{n}(fg) \leq c_{n}\ q_{n_{0}}(f)\ q_{n_{0}}(g),\ for\ all\ f,g \in \mathcal{A}.
	\end{equation*}
\end{definition}
\begin{definition}\label{5151}
	A Fr\'{e}chet  algebra $(\mathcal{A},\tau(Q))$ is said to be submultiplicative if $Q=\{q_{n}\}_{n \in \mathbb{N}}$ is a family of submultiplicative norms, that is 
	\begin{equation*}
	q_{n}(fg) \leq q_{n}(f)\ q_{n}(g),\ for\ all\ f,g \in \mathcal{A}.
	\end{equation*}
\end{definition}\par
It follows from the above definition together with Example \ref{ex1} and the equation
\begin{equation*}\label{e2}
\|xy\| \leq \|x\|.\|y\|
\end{equation*}  that every Banach algebra is submultiplicative Fr\'{e}chet  algebra.\par
Below we give an example on a unital Fr\'{e}chet  algebra which is not a Banach algebra.
\begin{example}
	Consider $C^{\infty}([0,1])$ with operations of pointwise multiplication and addition (cf. \ref{ex4}). Consider the collection of semi-norms 
	$Q=\{q_{n}\}_{n \in \mathbb{N}}$ given by 
	\begin{equation*}
	q_n(f)= 2^{n}\ \sup_{0\leq k \leq n}\sup_{x \in [0,1]}|f^{(k)}(x)|
	\end{equation*} It is clear that $\{q_{n}\}_{n \in \mathbb{N}}$ is an increasing family of semi-norms. The completeness follows by a similar approach to that in Remark \ref{r567} (for a detailed proof that $\mathcal{A}$ is Fr\'{e}chet  space we refer the reader (chapter 2 section 9 example 3))\\
	Note that \begin{equation*}
	\sup_{x \in [0,1]}|f^{(k)}(x)| \leq \frac{q_{n}(f)}{2^{n}},\ for\ all\ k\leq n
	\end{equation*}
	so that
	\begin{equation*}
	q_{n}(f)\ q_{n}(g)\geq 2^{n}\sup_{x \in [0,1]}|f^{(k)}(x)|.\ 2^{n}\sup_{x \in [0,1]}|g^{(k)}(x)|
	\end{equation*}
Using Leibniz Formula we have for each $k \in \{0,...,n\}$,
	\begin{align*}
	|(fg)^{(k)}(x)|&\leq \sum_{l=0}^{k}C_{l}^{k}\ |f^{(l)}(x)|\ |g^{(k-l)}(x)|\\
	&\leq \sum_{l=0}^{k}C_{l}^{k}\ \sup_{l\leq k}|f^{(l)}(x)|\ \sup_{l\leq k}|g^{(k-l)}(x)|\\
	\end{align*}
	Hence
	\begin{align*}
	\sup_{0\leq k \leq n}|(fg)^{(k)}(x)|&\leq \sup_{0\leq k \leq n}\sum_{l=0}^{k}C_{l}^{k}\ \sup_{l\leq k}|f^{(l)}(x)|\ \sup_{l\leq k}|g^{(k-l)}(x)|\\
	&\leq 2^{n}\frac{q_{n}(f)}{2^{n}}.2^{n}\frac{q_{n}(g)}{2^{n}}\\
	&=\frac{q_{n}(f)\ q_{n}(g)}{2^{n}},\ for\ all\ x \in [0,1].\\
	\end{align*}
	Therefore
	\begin{align*}
	q_{n}(fg)&=2^{n}\sup_{0\leq k \leq n}\sup_{x \in [0,1]}|(fg)^{(k)}(x)|\\
	&\leq 2^{n}\frac{q_{n}(f)\ q_{n}(g)}{2^{n}}\\
	&=q_{n}(f).q_{n}(g).
	\end{align*}
	This shows that $\mathcal{A}$ is a submultiplicative Fr\'{e}chet  algebra (where $c_{n}=1\ and\ n_{0}=n$ in the above definition).
\end{example}
\begin{example}
	The Schwartz space $S(\mathbb{R}^{N})$ defined in Example \ref{ex2} is also a Fr\'{e}chet  algebra under pointwise multiplication and addition. This follows also using a similar argument as done in the previous example with the aid of Leibniz formula for $C^{\infty}$ functions on $\mathbb{R}^{N}$
	\begin{equation*}
	\partial^{\alpha}(fg)=\sum_{\mu \in \mathbb{N},\mu \leq \alpha}C_{\mu}^{\alpha} \partial^{\mu} f(x)\partial^{\alpha-\mu}g(x).
	\end{equation*}
	It is clear that this Fr\'{e}chet  algebra is not a unital algebra.
\end{example}\par
The $\sigma$-compactness of $\mathbb{R}^{N}$ allows us to make the collection of smooth functions a Fr\'{e}chet  algebra. 
\begin{example}
	Let $\mathcal{A}=C^{\infty}(\mathbb{R}^{N})$ be the space of all smooth functions on $\mathbb{R}^{N}$, we consider the collection of $\{k_{n}=B[0,n+1]\}_{n \in \mathbb{N}} \subset \mathbb{R}^{N}$ of the closed balls centered at zero and radius $n+1$. For each $n \in \mathbb{N}$
	\begin{equation*}
	p_{n}(f)=\sup\{|D^{\alpha}f(x)|,\ x \in k_{n}\ and\ |\alpha|\leq n \}.
	\end{equation*}
	Then $\{p_{n}\}_{n\in \mathbb{N}}$ is an increasing family of semi-norms giving the topology of uniform convergence on the compact sets. The completeness follows from the fact that $C^{\infty}(B[0,n])$ is complete.  Using Leibniz formula it follows that the multiplication is continuous and hence $C^{\infty}(\mathbb{R}^{N})$ is a Fr\'{e}chet  algebra.
\end{example}\par
The notion of embedding a Fr\'{e}chet  algebra in another is given in the below definition.
\begin{definition}\label{d4}
	Let $(\mathcal{A},\tau_{\mathcal{A}})\ and\ (\mathcal{B},\tau_{\mathcal{B}})$ be two Fr\'{e}chet  algebras with $\mathcal{B} \subset \mathcal{A}$, we say that $\mathcal{B}$ is continuously embedded in $\mathcal{A}$ and we write $\mathcal{B}\hookrightarrow \mathcal{A}$ if the injection map $id_{\mathcal{B}}:(\mathcal{B},\tau_{\mathcal{B}}) \longrightarrow (\mathcal{A},\tau_{\mathcal{A}})$ is continuous. 
\end{definition}
\begin{remark}
	Let $\{q_n\}_{n \in \mathbb{N}}\ and\ \{p_{n}\}_{n \in \mathbb{N}}$ be the semi-norms defined on $\mathcal{B}\ and\ \mathcal{A}$, respectively. Using Theorem \ref{t4} together with the fact that the identity map is linear, it follows then $\mathcal{B}\hookrightarrow\mathcal{A}$ if and only if for each $n \in \mathbb{N}$, there is $m \in \mathbb{N}$ and $c_{n}>0$ such that
	\begin{equation*}
	q_{n}(b) \leq c_{n}\ p_{m}(b),\ for\ all\ b\in \mathcal{B}.
	\end{equation*}
\end{remark}
\begin{example}
	Every closed subalgebra $\mathcal{B}$ of a Fr\'{e}chet  algebra $\mathcal{A}$ is continuously embedded in $\mathcal{A}$. In particular, every closed Banach subalgebra is continuously embedded in the algebra.
\end{example}
\begin{example}
	The Fr\'{e}chet  algebra of rapidly decreasing functions $\mathcal{S}(\mathbb{R}^{N})$ is continuously embedded in the algebra of all smooth functions on $\mathbb{R}^{N}$. Indeed, for any $f \in \mathcal{S}(\mathbb{R}^{N})$ we have
	\begin{align*}
	p_{n}(f)&=\sup_{|\alpha|\leq n;\ x \in B[0,n+1]} |D^{\alpha}f(x)|\\
	&\leq \sup_{|\alpha|\leq n;\ x \in \mathbb{R}^{N}}  |D^{\alpha}f(x)|\\
	&\leq \sup_{|\alpha|\leq n;\ x \in \mathbb{R}^{N}}  |(1+|x|^{2})^{n}D^{\alpha}f(x)|\\
	&=q_{n}(f).
	\end{align*}
	This shows that $\mathcal{S}(\mathbb{R}^{N}) \hookrightarrow C^{\infty}(\mathbb{R}^{N})$.
\end{example}
\section{$C^{*}$-algebras}
\hskip0.5cm This section is devoted to study particular types of Banach algebras, known as $C^{*}$-algebras. The term ${C}^{*}$-algebra refers back to Segal for the study of closed sub-algebras of $\mathcal{L}(\mathcal{H})$ where $\mathcal{H}$ is a Hilbert space. We present these algebras in an abstract way. The terminology of the notion follows from Example \ref{ex0990} below.
 \subsection{Involutions and $\star$-algebras}
\hskip0.5cm We start by the notion of an involution and $\star$-algebra.
 \begin{definition}\label{d2}
 	Let $\mathcal{A}$ be an algebra over $\mathbb{C}$, we say that a map $\star: \mathcal{A} \longmapsto \mathcal{A}$ is an involution if it satisfies the following conditions:
 	\begin{enumerate}
 		\item $(\alpha a+b)^{\star}=\ \overline{\alpha}a^{\star}+b^{\star}, \alpha \in \mathbb{C},\ a,b \in \mathcal{A}$.
 		\item $(a^{\star})^{\star}=a,\ for\ all\ a\in \mathcal{A}\ and$ 
 		\item $(ab)^{\star}=b^{\star}a^{\star},\ for\ all\ a,b\in \mathcal{A}.$
 	\end{enumerate}
 In this case, we say that $(\mathcal{A},\star)$ is a $\star-algebra$.
 \end{definition}
\begin{example}\label{ex3}
	Following the notation in Example \ref{ex4}), consider $f \in l^{\infty}(\mathcal{W})$   and define $f^{\star}$ to be the complex conjugate of f, i.e. $f^{\star}(x)=\overline{f(x)}$ for all x. It is clear that this map is an involution and $f^{\star} \in l^{\infty}(\mathcal{W})$. This shows that $(l^{\infty}(\mathcal{W}),\star)$ is a $\star -algebra$.
\end{example}\par
The transpose conjugate is an involution on the set of square matrices $\mathcal{M}_{n}(\mathbb{C})$. This is proved in the below example.
\begin{example}
	Let $\mathcal{A}=\mathcal{M}_{n}(\mathbb{C})$ and define $\star :\mathcal{M}_{n}(\mathbb{C})  \longrightarrow\mathcal{M}_{n}(\mathbb{C})\ by\ {A}^{\star}=\overline{{A}}^{T}$ i.e.
	\begin{equation*}
	a_{i,j}^{\star}=\overline{a_{j,i}}\ \quad for\ i,j \in \{1,...,n\}.
	\end{equation*}
	Then $(\mathcal{M}_{n}(\mathbb{C}),\star)$ is a $\star -algebra$. In fact, given $\alpha \in \mathbb{C}\ and\ {A},{B}\ in\ \mathcal{M}_{n}(\mathbb{C}) $ we have:
	\begin{enumerate}
		\item $(\alpha{A}+{B})^{\star}=(\overline{\alpha{A}+{B}})^{T}=(\overline{\alpha}\overline{{A}}+\overline{{B}})^{T}=\overline{\alpha}\overline{{A}}^{T}+\overline{{B}}^{T}=\overline{\alpha}{A}^{\star}+{B}^{\star}$
		\item $({A}^{\star})^{\star}=\overline{(\overline{{A}}^{T})}^{T}=((\overline{\overline{{A}}})^{T})^{T}={A}$
		\item $({A}{B})^{\star}=(\overline{{A}{B}})^{T}=(\overline{{A}}\overline{{B}})^{T} =(\overline{{B}})^{T}(\overline{{A}})^{T}={B}^{\star}{A}^{\star}$
  \end{enumerate}
\end{example}\par
  A more general example is given by the adjoint operator on any Hilbert space.
\begin{example} \label{ex47}
	Let $\mathcal{H}$ be a Hilbert space and $\mathcal{L}(\mathcal{H})$ be the space of all bounded operator with $\|.\|_{op}$ defined in Example \ref{ex50}. Let $\star:\mathcal{L}(\mathcal{H}) \longmapsto \mathcal{L}(\mathcal{H})$ be given as follows: for ${T} \in \mathcal{L}(\mathcal{H})$, let ${T}^{\star}:\mathcal{H} \longmapsto \mathcal{H}$ be the map defined by:
	\begin{equation*}
	<{T}f,g>=<f,{T}^{\star}g>,\ for\ all\ f,g \in \mathcal{H}.
	\end{equation*}
	Then $(\mathcal{L}(\mathcal{H}),\star)$ is a $\star-algebra$.\\
	First we prove that ${T}^{\star} \in \mathcal{L}(\mathcal{H})$. For the linearity, given $\alpha \in \mathbb{C}\  with \ f,g\ and\ h \in \mathcal{H}$ we have
	\begin{align*}
	<f,  {T}^{\star}(\alpha g+h)>&=<{T}f, \alpha g+h>\\
	&=<{T}f,\alpha g>+<{T}f,h>\\
	&=\overline{\alpha}<{T}f,g>+<{T}f,h>\\
	&=\overline{\alpha}<f,{T}^{\star}g>+<f,{T}^{\star}h>\\
	&=<f,\alpha{T}^{\star}g>+<f,{T}^{\star}h>\\
	&=<f,\alpha{T}^{\star}g+{T}^{\star}h>
	\end{align*}
	Hence ${T}^{\star}(\alpha g+h)=\alpha{T}^{\star}g+{T}^{\star}h$ i.e.  ${T}^{\star}$ is linear. For the boundedness, given any $g \in \mathcal{H}$  the map
	\begin{equation*}
	f \longmapsto <{T}f,g>
	\end{equation*}
	is bounded since
	\begin{equation*}
	|<{T}f,g>| \leq \|{T}f\|_{\mathcal{H}}\|g\|_{\mathcal{H}} \leq \|{T}\|_{\mathcal{L}(\mathcal{H})}\|g\|\|f\|.
	\end{equation*}
	Hence by Reisz Representation theorem, $\sup_{\|f\|=1}|<{T}f,g>|=\|{T}^{\star}g\|$. Hence
	\begin{equation*}
	\|{T}^{\star}g\| \leq \|{T}f\| \|g\|\ for\ all\ f\ with\ \|f\|=1.
	\end{equation*}
	i.e.
	\begin{equation*}
	\|{T}^{\star}g\| \leq \|{T}\| \|g\|.
	\end{equation*}
	This shows that ${T}^{\star}$ is bounded with
	\begin{equation}\label{e30}
	\|{T}^{\star}\|_{op} \leq \|{T}\|_{op}.
	\end{equation}
	For the condition of $\star-algebra$. Consider $\alpha \in \mathbb{C},\ and\ {T},{S} \in \mathcal{L}(\mathcal{H})$ then for any $f,g \in \mathcal{H}$ we have:
	\begin{enumerate}
		\item 
		$<f,(\alpha T)^{\star}g>=<(\alpha T)f,g>=\alpha<Tf,g>=\alpha<f,T^{\star}g>=<f,\overline{\alpha}T^{\star}g>.$
		\item $<f,(T+S)^{\star}g>=<(T+S)f,g>=<Tf+Sf,g>=<Tf,g>+<Sf,g>\\ =<f,T^{\star}g>+<f,S^{\star}g>=<f,T^{\star}g+S^{\star}g>=<f,(T^{\star}+S^{\star})g>.$
		\item $<f,(T^{\star})^{\star}g>=<T^{\star}f,g>=\overline{<g,T^{\star}f>}=\overline{<Tg,f>}=<f,Tg>.$
		\item $<f,(TS)^{\star}g>=<TSf,g>=<Sf,T^{\star}g>=<f,S^{\star}T^{\star}g>.$
	\end{enumerate}
\end{example}
\begin{proposition}
	Let $\mathcal{A}$ be a unital $\star-algebra$ with identity e, then the following holds true:
	\begin{enumerate}
		\item $e^{\star}=e.$
		\item If $a \in \mathcal{A}$ is invertible then $a^{\star} \in Inv(\mathcal{A})$ with $(a^{\star})^{-1}=(a^{-1})^{\star}$.
		\item $\sigma(a^{\star})=\{\overline{\lambda} \mid \lambda \in \sigma(a) \}$.
	\end{enumerate}
\begin{proof}
	\begin{enumerate}
		\item As $e$ is the identity element of $\mathcal{A}$ then following the second and third condition in Definition \ref{d2}, we have 
		$$e^{\star}=e^{\star}e=e^{\star}(e^{\star})^{\star}=(e^{\star}e)^{\star}=(e^{\star})^{\star}=e$$
		\item Let $a \in Inv(\mathcal{A})$ then $a^{-1}a=e$. Hence by 1),  $a^{\star}(a^{-1})^{\star}=e$ and  $(a^{-1})^{\star}a^{\star}=e$ so that $a^{\star} \in Inv(\mathcal{A})\ and\ (a^{\star})^{-1}=(a^{-1})^{\star}$.
		\item We know that $\lambda \notin \sigma(a)$ if and only if $\lambda e-a \in Inv(\mathcal{A})$. Using 2), this is equivalent to say $(\lambda e-a)^{\star} \in Inv(\mathcal{A})$ i.e. $\overline{\lambda}e-a^{\star}$ is invertible in $\mathcal{A}$, which means $\overline{\lambda} \notin \sigma(a^{\star})$.
	\end{enumerate}
\end{proof}
\end{proposition}
\subsection{$C^\star$-algebras}
\hskip0.5cm We start by the definition of ${C}^{\star}-algebras$ which are special type of Banach algebras.
\begin{definition}
	Let $(\mathcal{A},\|.\|)$ be a Banach algebra with an involution $\star$, we say that $\mathcal{A}$ (or $(\mathcal{A},\|.\|)$) is a ${C}^{\star}-algebra$ if the following condition is satisfied: 
	\begin{equation}\label{e4}
	\|aa^{\star}\|=\|a\|^{2}\ for\ all\ a \in \mathcal{A}
	\end{equation}
\end{definition}
\begin{remark}
	Following the definition of a ${C}^{\star}-algebra$, it is easy to see that the norm being preserved under involution i.e. $\|a^{\star}\|=\|a\|$, indeed by multiplication of semi-norm, we have
	\begin{equation*}
	\|a\|^{2} \leq \|a\|.\|a^{\star}\|
	\end{equation*}
	Replacing a by $a^{\star}$ we get the required equality. 
\end{remark}\par
Below we give some standard examples of ${C}^{\star}$-algebras.
\begin{example}
	Combining the result of Example \ref{ex4} and \ref{ex3}, we know that $l^{\infty}(\mathcal{W})$ is a Banach algebra and also a $\star-algebra$. Moreover,
	\begin{equation*}
	\|ff^{\star}\|=\sup_{x \in \mathcal{W}} |\overline{f(x)}f(x)|=(\sup_{x \in \mathcal{W}}|f(x)|)^{2}=\|f\|^{2}
	\end{equation*}
	So that $l^{\infty}(\mathcal{W})$ is a ${C}^{\star}-algebra$.
\end{example}
\begin{example}\label{ex0990}
	Let $\mathcal{H}$ be a Hilbert space then $(\mathcal{L}(\mathcal{H}),\|.\|_{op})$ is a ${C}^{\star}-algebra$. The fact that it is a $\star-algebra$ and a Banach algebra follows from Example  \ref{ex50} and \ref{ex47}. The condition (\ref{e4}) is also verified. In fact, given $l{T} \in \mathcal{L}(\mathcal{H})$, then by (\ref{e30}), we have $\|{T}^{\star}\| \leq \|{T}\|$, replacing ${T}\ by\ {T}^{\star}$ in the above inequality, we obtain that
	\begin{equation*}
	\|{T}^{\star}\|=\|{T}\|.
	\end{equation*}
	Thus by submultiplicativity of norms we get
	\begin{equation}\label{e6}
	\|{T}^{\star}{T}\| \leq \|{T}^{\star}\| \|{T}\|=\|{T}\|^{2}
	\end{equation}
	Moreover, for $f \in \mathcal{H}\ and\ f \neq 0$, we have 
	\begin{align*}
	\|{T}f\|^{2}&=<{T}f,{T}f>=<{T}^{\star}{T}f,f>\\
	&\leq \|{T}^{\star}{T}\| \|f\| \|f\|
	\end{align*}
	Hence,
	\begin{equation*}
	\frac{\|{T}f\|^{2}}{\|f\|^{2}} \leq \|{T}^{\star}{T}\|
	\end{equation*}
	i.e
	\begin{equation*}
	(\frac{\|{T}f\|}{\|f\|})^{2} \leq \|{T}^{\star}{T}\|
	\end{equation*}
	This shows that $\|{T}\|^{2} \leq \|{T}^{\star}\mathcal{T}\|$. The previous inequality with (\ref{e6}) shows that $\mathcal{L}(\mathcal{H})$ is a ${C}^{\star}$-algebra.
\end{example}
\begin{definition}\label{d5}
	Let $(\mathcal{A},\star)$ be a $\star$-algebra and let $\mathcal{B}$ be a subset of $\mathcal{A}$. We say that $\mathcal{B}$ is symmetric ($\mathcal{B}^{\star}=\mathcal{B}$) if it is closed under the involution $\star$ i.e. given $b \in \mathcal{B}$, then $b^{\star} \in \mathcal{B}$.
\end{definition}\par
We now provide the precise definition of a $C^{\star}$-subalgebra.
\begin{definition}
	Let $\mathcal{A}$ be a $C^{\star}-algebra$ and $\mathcal{B}$ be a subalgebra of $\mathcal{A}$ (cf. Definition \ref{d3}). If $\mathcal{B}$ is symmetric and norm-closed in $\mathcal{A}$ then we say that $\mathcal{B}$ is a $C^{\star}-subalgebra$ of $\mathcal{A}$.
\end{definition}
\begin{example}
	Let $\mathcal{A}=\mathcal{M}_{n}(\mathbb{C})$ be the  $C^{\star}$-algebra of square matrices with entries in $\mathbb C$. Consider  $\mathcal{B}$ be the subalgebra of diagonal matrices in $\mathcal{A}$. It is easy to see that $\mathcal{B}$ is symmetric and closed in $(\mathcal{A},\|.\|_{op})$. Hence $\mathcal{B}$ is a $C^{\star}$-subalgebra of $\mathcal{A}$.
\end{example}\par
To make the topic self contained we recall an important representation of $C^{\star}-algebras$. Let us note that if $\mathcal{B} \subset \mathcal{L}(\mathcal{H})$ is closed $\star-subalgebra$ then it is a $C^{\star}-algebra$, the converse is also true, but it is beyond the scope of the thesis. This is known as Gelfand Naimark theorem (cf.\cite{john}  for a detailed proof and construction).
\begin{theorem}
	Let $\mathcal{A}$ be a $C^{\star}-algebra$, then there is a Hilbert space $\mathcal{H}$ and a $C^{\star}$-subalgebra $\mathcal{B}$ of $\mathcal{L}(\mathcal{H})$ such that $\mathcal{A}$ is $\star$-isomorphic to $\mathcal{B}$.
\end{theorem}
\begin{proposition}
	Let $\mathcal{A}$ be a $\star-algebra$. Assume $(\mathcal{A},\|.\|_{1})$ is a $C^{\star}-algebra$ and let $\|.\|_{2}$ be another norm on $\mathcal{A}$. If $(\mathcal{A},\|.\|_{2})$ is also a $C^{\star}-algebra$ then the two norms are equal.
	\begin{proof}
		Let $b \in \mathcal{A}$ be a self-adjoint element i.e. $b^{\star}=b$. Hence, for $j=1\ or\ 2$ we have
		\begin{equation*}
		\|b^{2}\|_{j}=\|b\|_{j}^{2}
		\end{equation*}
		and by induction we get
		\begin{equation*}
		\|b^{2^{n}}\|_{j}=\|b\|^{2^{n}}_{j},\ for\ all\ n \in \mathbb{N}.
		\end{equation*}
		Then by (\ref{e5}) for the spectral radius, we have
		\begin{align*}
		\rho(b)&=\lim_{n\to\infty}\|b^{n}\|_{j}^{\frac{1}{n}}\\
		&=\lim_{n\to\infty}\|b^{2^{n}}\|_{j}^{\frac{1}{2^{n}}}\\
		&=\lim_{n\to\infty}(\|b\|_{j})^{2^{n}\times\frac{1}{2^{n}}}\\
		&=\|b\|_{j}
		\end{align*}
		Now let $a \in \mathcal{A}$, then $a^{\star}a$ is self-adjoint so $\rho(a^{\star}a)=\|a^{\star}a\|_{j} =\|a\|_{j}^{2}$ i.e. $\|a\|_{1}^{2}=\|a\|_{2}^{2}$ so $\|a\|_{1}=\|a\|_{2}$ then the two norms are equal on $\mathcal{A}$.
	\end{proof}
\end{proposition}\par
We recall the following result on the spectral invariance of $C^{\star}-algebras$ which can be found in( \cite{ger}, Theorem 2.1.11).
\begin{theorem}
	Let $\mathcal{A}$ be a $C^{\star}$-algebra and let $\mathcal{B}$ be a $C^{\star}$-subalgebra of $\mathcal{A}$ with a common unit. Then $\mathcal{B}$ is spectral invariant (or closed under holomorphic functional calculus) in $\mathcal{A}$.
\end{theorem}
\section{$\Psi^{\star}$-algebras}
The term $\Psi^\star$-algebra dates back to the work of Gramsch in \cite{gr1}. These algebras are  Fr\'{e}chet -algebras continuously embedded in a $C^\star$-algebra and being closed under holomorphic functional calculus. The idea of introducing this notion was to generalize  the work of R. Beals  for the case of pseudodifferential operators.

Let us first introduce the notion of a $\Psi^{\star}-algebra$.
\begin{definition}\label{d7}
	Let $(\mathcal{A},\|.\|)$ be a $C^{\star}-algebra$ and let $(\mathcal{B},\tau_{\mathcal{B}})$ be a Fr\'{e}chet  subalgebra of $\mathcal{A}$. Suppose that $\mathcal{B}$ satisfies the following three conditions:
	\begin{enumerate}
		\item $\mathcal{B}$ is inverse closed in $\mathcal{A}$ i.e. $Inv(\mathcal{B})=Inv(\mathcal{A})\cap \mathcal{B}$ (cf. Definition \ref{d1}).
		\item $(\mathcal{B},\tau_{\mathcal{B}})$ is continuously embedded in $(\mathcal{A},\|.\|)$ (cf. Definition \ref{d4}).
		\item $\mathcal{B}$ is symmetric i.e. closed under involution (cf. Definition \ref{d5}).
			\end{enumerate}
Then $\mathcal{B}$ is called a $\Psi^{\star}-algebra\ in\ \mathcal{A}$.
\end{definition}
\begin{remark}
	If the Fr\'{e}chet  algebra $\mathcal{B}$ is a submultiplicative algebra (cf Definition \ref{5151}) and is a $\Psi^{\star}-algebra$ in $\mathcal{A}$. We say that $\mathcal{B}$ is a submultiplicative $\Psi^{\star}-algebra$ in $\mathcal{A}$.
\end{remark}
\begin{example}
	Let $\mathcal{A}$ be a $C^{\star}$-algebra and $\mathcal{B}$ be a $C^{\star}$-subalgebra of $\mathcal{A}$, then $\mathcal{B}$ is a $\Psi^{\star}$-algebra in $\mathcal{A}$.
\end{example}\par
Before we proceed more, we need the following property about the continuity of the inverse map in the case of a Fr\'{e}chet  algebra (which is more general in the case of Banach algebra (cf. Proposition \ref{r1})) and which can be found in \cite{wael}.
\begin{lemma}
	Let $\mathcal{B}$ be a Fr\'{e}chet  algebra then $Inv(\mathcal{B})$ is a $G_{\delta}$ set if and only if the map $x \longmapsto x^{-1}$ is continuous on $Inv(\mathcal{B})$.
\end{lemma}\par
In  connection with the case of $\Psi^{\star}-algebra$, we have the following result.
\begin{proposition}
	Let $\mathcal{B}$ be a $\Psi^{\star}-algebra$ in $\mathcal{A}$, then the map $b \longmapsto b^{\star}$ is continuous on the group of $Inv(\mathcal{B})$.\\
	\begin{proof}
		From condition one of Definition \ref{d7} we know that $id_{\mathcal{B}}:\mathcal{B} \longmapsto \mathcal{A}$ is continuous. As $\mathcal{A}$ is $C^{\star}-algebra$ then $Inv(\mathcal{A})$ is open in $\mathcal{A}$ (cf Proposition \ref{r1}). Therefore the continuity of $id_{\mathcal{B}}$ shows that $id_{\mathcal{B}}^{-1}(Inv(\mathcal{A}))$ is open in $\mathcal{B}$.But
		\begin{equation*}
		id_{\mathcal{B}}^{-1}(Inv(\mathcal{A}))=\mathcal{B} \cap Inv(\mathcal{A})=Inv(\mathcal{B})
		\end{equation*}
		by condition three in definition \ref{d7}. This shows that $Inv(\mathcal{B})$ is open in the Fr\'{e}chet space $\mathcal{B}$. However every open set is a $G_{\delta}$ set. The continuity of the inverse map then follows by the previous lemma.
	\end{proof}
\end{proposition}\par
Based on the above observation, it is then natural to ask the following question:
Given a subalgebra $\mathcal{B}$ of $\mathcal{A}$, can we find a Fr\'{e}chet  topology $\tau_{\mathcal{B}}$ on $\mathcal{B}$ making $\mathcal{B}$ a $\Psi^{\star}-algebra$ in $\mathcal{A}$? In this situation we obtain that the invertible group in $\mathcal{B}$ is being open and $\mathcal{B}$ is also closed under holomorphic functional of $\mathcal{A}$. The above question has no positive answer in general.
\chapter{Pseudodifferential operators}
 \hspace{1cm}In this chapter, we introduce pseudodifferential operators whose symbols are in the H\"{o}rmander classes $S^{m}_{\rho,\delta}$. These operators generalize differential operators as we shall see in Example \ref{exp1515}. The collection of pseudodifferential operators with symbols in $S^{m}_{\rho,\delta}$ is denoted by $\Psi^{m}_{\rho,\delta}$. We show that each class of operators $\Psi^{m}_{\rho,\delta}$ is a Fr\'{e}chet  space. Moreover, for the case $m=0$, we obtain a Fr\'{e}chet t algebra $\Psi^{0}_{\rho,\delta}$. In case $0 \leq \delta < \rho \leq 1$, we recall the Calderon Vaillancourt theorem for the boundedness of the operators on $L^{2}$. We give a proof on the symbols having a compact support. Furthermore, we show that $\Psi^{0}_{\rho,\delta}$ is a submultiplicative algebra and indicate the fact $\Psi^{0}_{\rho,\delta}$ is a $\Psi^{\star}-algebra$ in $\mathcal{L}(L^{2}(\Omega))$.
\section{Oscillatory integrals}
\hskip0.5cm By using Fourier transform, pseudodifferential operators converts the concept of linear partial differential operators with smooth coefficient to the algebraic theory. We start by considering the Fourier transform which is given by:
\begin{equation}
\hat{f}(\xi)=(\mathcal{F}f){\xi}=\int_{\mathbb{R}^{n}}e^{-ix\xi}f(x)dx\label{e7}
\end{equation}
where $dx=dx_{1}.dx_{2}\cdots dx_{n}$ is a Lebesgue measure.\\
In case $f$ is integrable, the Fourier transform is well defined. Moreover, by considering $\mathcal{F}$ on the Schwartz space we obtain a linear continuous map 
\begin{equation*}
\mathcal{F}:\mathcal{S}(\mathbb{R}^{n}) \longmapsto \mathcal{S}(\mathbb{R}^{n})
\end{equation*}
having an inverse operator
\begin{equation*}
\mathcal{F}^{-1}\hat{f}(x)=\int_{\mathbb{R}^{n}}e^{ix\xi}\hat{f}(\xi)d\xi,
\end{equation*}
where $d\xi=(2\pi)^{-n}d\xi_{1}\cdots d\xi_{n}$.
\begin{proposition}
	The adjoint operator $\mathcal{F}^{\star}$ of the Fourier transform $\mathcal{F}$ is the inverse Fourier transform multiplied by the coefficient $(2\pi)^{n}$.\\
	\begin{proof}
		Let $f,g \in \mathcal{S}(\mathbb{R}^{n})$, then
		\begin{align*}
	<\mathcal{F}f,g>&=(2\pi)^{n}\int_{\mathbb{R}^{n}}\mathcal{F}f(\xi)\overline{g(\xi)} d\xi\\
	&=(2\pi)^{n}\int_{\mathbb{R}^{n}}\{\int_{\mathbb{R}^{n}}e^{-ix\xi}f(x)dx \}\overline{g(\xi)}d\xi\\
	&=(2\pi)^{n}\int_{\mathbb{R}^{n}}\int_{\mathbb{R}^{n}}e^{-ix\xi}\overline{g(\xi)}d\xi f(x)dx\\
	&=(2\pi)^{n}\int_{\mathbb{R}^{n}}\overline{\int_{\mathbb{R}^{n}}e^{ix\xi}g(\xi)d\xi} f(x)dx\\
	&=(2\pi)^{n}<f,\mathcal{F}^{-1}g>
	\end{align*}
		So that $\mathcal{F}^{\star}=(2\pi)^{n}\mathcal{F}^{-1}$.
	\end{proof}
\end{proposition}\par
We want to extend the definition of (\ref{e7}) to more general functions. For this reason we fix a positive number $l>0$ and consider the following space:
\begin{equation*}
S^{l}(\mathbb{R}^{n})=\{f:\mathbb{R}^{n}\longrightarrow\mathbb{C} \mid f \mbox{ is continuous and there is } c>0 \mbox{ with } |f(x)|<c<x>^{l} \},
\end{equation*}
where $<x>$ is the Japanese bracket given by
\begin{equation*}
<x>^{l}=(1+|x|^{2})^{\frac{l}{2}}=(1+x_{1}^{2}+\cdots+x_{n}^{2})^{\frac{l}{2}}.
\end{equation*}
Note that a function $f$ is rapidly decreasing $(f \in \mathcal{S}(\mathbb{R}^{n}))$ if and only if it is a smooth function and for all $\alpha \in \mathbb{N}^{n}$ and $N>0$, there is $c_{\alpha,N}$ such that
\begin{equation*}
|\partial_{x}^{\alpha}f(x)| \leq c_{\alpha,N}<x>^{-N},\ for\ all\ x \in \mathbb{R}^{n}.
\end{equation*}
This shows that the Schwartz space $\mathcal{S}(\mathbb{R}^{n})$ is contained in $S^{l}(\mathbb{R}^{n})$. It is also clear that the equation (\ref{e7}) is not defined for functions in $S^{l}(\mathbb{R}^{n})$. Our aim is to extend the definition in (\ref{e7}) to functions in  $S^{l}(\mathbb{R}^{n})$. However, we will obtain only distributions and not usual functions.\\

\bigskip

Let us first set up some notation and recall a  multiplicative property for the Fourier transform on $\mathcal{S}(\mathbb{R}^{n})$ (cf. \cite{dijk}). For $k=1,\cdots,n$ and $\alpha \in \mathbb{N}^{n}$ put $D_{k}=-i\partial_{k}$ and $D=(D_{1},\cdots,D_{n})$.\\
Then
\begin{equation}\label{e8}
D^{\alpha}f(x)=\int_{\mathbb{R}^{n}}e^{ix\xi} \xi^{\alpha}\hat{f}(\xi)d\xi,\ for\ f \in \mathcal{S}(\mathbb{R}^{n}).
\end{equation}
Let $q \in \mathbb{N}$ and $p=2q$ (an even number), we write
\begin{equation*}
<D>^{p}=(1+D_{1}^{2}+\cdots+D_{n}^{2})^{q}.
\end{equation*}
It is easy to verify that
\begin{equation}\label{e9}
e^{-ix\xi}=<x>^{-p}<D_{\xi}>^{p}e^{-ix\xi}.
\end{equation}
If we consider the Laplace operator $\Delta_x=\sum_{i=1}^{n}\frac{\partial^{2}}{\partial x_{i}^{2}}$ then we can check that
	\begin{equation*}
	e^{i\xi(x-y)}=<\xi>^{-2m}(1-\Delta_y)^{m} e^{i\xi(x-y)}=<x-y>^{-2m}(1-\Delta_\xi)^{m}e^{i\xi(x-y)}
	\end{equation*}
	where $m=0,1,2,\cdots$ and $x,y,\xi \in \mathbb{R}^{n}$.\par
	We will regularly use the following version of Peetre's inequality.
	\begin{proposition}\label{p10}
		Let $l \in \mathbb{R}$ then
		\begin{equation*}
		<w+\mu>^{l} \leq 2^{|l|}<w>^{|l|}<\mu>^{l}\ for\ all\ w,\mu \in \mathbb{R}^{n}.
		\end{equation*}\\
		\begin{proof}
			Note that for any $x \in \mathbb{R}^{n}$ we have
			\begin{equation*}
			<x>^{2}=(1+|x|^{2})\leq(1+|x|)^{2} \leq (1+|x|)^{2}+(1-|x|)^{2}=2(1+|x|^{2}).
			\end{equation*}
			Given $l>0$ we have:
			\begin{align*}
			<w+\mu>^{l}&=(1+|w+\mu|^{2})^{\frac{l}{2}}\\
			&\leq (1+|w+\mu|)^{l}\\
			&\leq (1+|w|+|\mu|)^{l}\\
			&\leq (1+|w|)^{l}(1+|\mu|)^{l}\\
			&\leq 2^{l}<w>^{l}<\mu>^{l}
			\end{align*}
			If $l<0$, we replace $w+\mu$ by $\mu$ and w by $-w$ and we obtain:
			\begin{equation*}
			<\mu>^{-l} \leq 2^{-l}<(-w)>^{-l}<\mu+w>^{-l}.
			\end{equation*}
			This shows that
			\begin{equation*}
			<w+\mu>^{l}\leq 2^{|l|}<w>^{|l|}<\mu>^{l}\ for\ all\ w,\mu \in \mathbb{R}^{n}.
			\end{equation*}
		\end{proof}
	\end{proposition}
 In order to regularize (\ref{e7}) for functions in $S^{l}(\mathbb{R}^{n})$, we observe $\mathcal{S}(\mathbb{R}^{n})$ as a subset of its dual space $\mathcal{S}^{'}(\mathbb{R}^{n})$. Recall the Schwartz space on $\mathbb{R}^{n}$ is given by:
\begin{equation*}
	\mathcal{S}(\mathbb{R}^{N})=\{f\in C^{\infty}(\mathbb{R}^{N})\mbox{ for each } \alpha,\beta \in \mathbb{N}^{N}\ \mbox{ we have } \sup_{x \in \mathbb{R}^{N}}|x^{\alpha}D^{\beta}f(x)|<\infty \}
\end{equation*}
with the collection of semi-norms 
\begin{equation*}
	q_{n}(f)=\sup_{|\alpha|\leq n}\sup_{x \in \mathbb{R}^{N}}|(1+|x|^{2})^{n}D^{\alpha}f(x)|,\ f \in \mathcal{S}(\mathbb{R}^{N})
\end{equation*}
The Schwartz space $\mathcal{S}(\mathbb{R}^{n})$ is regarded in $\mathcal{S}^{'}(\mathbb{R}^{n})$ in the following sense: given $u \in \mathcal{S}(\mathbb{R}^{n})$, we have $\hat{u} \in \mathcal{S}(\mathbb{R}^{n})$ and we define:
\begin{equation}\label{e10}
<\hat{u},v>=\int_{\mathbb{R}^{n}}\hat{u}(\xi)v(\xi)d\xi,\ v \in \mathcal{S}(\mathbb{R}^{n}).
\end{equation}
Note that
\begin{align*}
|<\hat{u},v>|&=|\int_{\mathbb{R}^{n}}\hat{u}(\xi)(1+|\xi|^{2})^{n+1}v(\xi)\frac{1}{(1+|\xi|^{2})^{n+1}}d\xi|\\
	&\leq \int_{\mathbb{R}^{n}}|\hat{u}(\xi)| |(1+|\xi|^{2})^{n+1}v(\xi)|\frac{1}{(1+|\xi|^{2})^{n+1}}d\xi\\
	&\leq \int_{\mathbb{R}^{n}}q_{0}(\hat{u})q_{n+1}(v)\frac{1}{(1+|\xi|^{2})^{n+1}}d\xi\\
	&=c\|\hat{u}\|_{\infty}q_{n+1}(v).\\
	\end{align*}
Where $c=\int_{\mathbb{R}^{n}}\frac{1}{(1+|\xi|^{2})^{n+1}}d\xi$. This shows that rapidly decreasing functions defines temperate distribution.\\
Plugging (\ref{e7}) in (\ref{e10}), we obtain
\begin{equation}\label{e11}
<\hat{u},v>=\int_{\mathbb{R}^{n}}(\int_{\mathbb{R}^{n}}e^{-ix\xi}u(x)dx)v(\xi)d\xi.
\end{equation}
By using Fubini's theorem, we can write (\ref{e11}) in the following way
\begin{equation}\label{e12}
\int_{\mathbb{R}^{n}}\int_{\mathbb{R}^{n}}e^{-ix\xi}u(x)v(\xi)dxd\xi,\ u,v \in \mathcal{S}(\mathbb{R}^{n}).
\end{equation}\par
We now extend the definition given in equation (\ref{e12}) to functions $f=u$ in $S^{l}(\mathbb{R}^{n})$.
\begin{theorem}
	Let $l>0$ and $f \in S^{l}(\mathbb{R}^{n})$. Fix $p>l+n$, then
	\begin{equation}\label{e13}
	<\hat{f},g>=\int_{\mathbb{R}^{n}}\int_{\mathbb{R}^{n}}e^{-ix\xi}f(x)<x>^{-p}<D_{\xi}>^{p}g(\xi)dxd\xi,\ g \in \mathcal{S}(\mathbb{R}^{n})
	\end{equation}
	extends the definition of (\ref{e12}) for rapidly decreasing functions. Moreover, equation (\ref{e13}) defines a linear continuous distribution $\hat{f} \in \mathcal{S}^{'}(\mathbb{R}^{n})$.\\
	\begin{proof}
		If $f=u \in \mathcal{S}(\mathbb{R}^{n})$ then by equation (\ref{e12}) we have
		\begin{equation*}
		<\hat{f},g>=\int_{\mathbb{R}^{n}}\int_{\mathbb{R}^{n}}e^{-ix\xi}f(x)g(\xi)dxd\xi
		\end{equation*}
		plugging (\ref{e9}) in the above equation, we get
		\begin{equation*}
		<\hat{f},g>=\int_{\mathbb{R}^{n}}\int_{\mathbb{R}^{n}}<x>^{-p}(<D_{\xi}>^{p}e^{-ix\xi})f(x)g(\xi)dxd\xi
		\end{equation*}
		Using integration by parts, we obtain
		\begin{equation*}
			<\hat{f},g>=\int_{\mathbb{R}^{n}}\int_{\mathbb{R}^{n}}e^{-ix\xi}f(x)<x>^{-p}(<D_{\xi}>^{p}g(\xi))dxd\xi.
		\end{equation*}
		Now for $f \in S^{l}(\mathbb{R}^{n})$ and $p>l+n$, we have
		\begin{align*}
		|<\hat{f},g>|&=|\int_{\mathbb{R}^{n}}\int_{\mathbb{R}^{n}}e^{-ix\xi}f(x)<x>^{-p}(<D_{\xi}>^{p}g(\xi))dxd\xi|\\
		&\leq \int_{\mathbb{R}^{n}}\int_{\mathbb{R}^{n}}|f(x)<x>^{-p}||<D_{\xi}>^{p}g(\xi)|dxd\xi\\
		&\leq c\int_{\mathbb{R}^{n}}<x>^{l-p}dx\int_{\mathbb{R}^{n}}|<D_{\xi}>^{p}g(\xi)|d\xi\\
		&=c\mu\int_{\mathbb{R}^{n}}|<\xi>^{n+1}<D_{\xi}>^{p}g(\xi)|\frac{1}{<\xi>^{n+1}}d\xi\\
		&\leq c\mu\int_{\mathbb{R}^{n}}p_{m}(g)\frac{1}{<\xi>^{n+1}}\\
		&=c\mu c_{1}p_{m}(g)\\
		&=c_{2}p_{m}(g),
		\end{align*}
		where $\mu=\int_{\mathbb{R}^{n}}\frac{dx}{(1+|x|^{2})^{p-l}}$, $m=max\{n+1,p\}>n$ and $c_{1}=\int_{\mathbb{R}^{n}}\frac{1}{<\xi>^{n+1}}d\xi$.
	\end{proof}
\end{theorem}
\section{Pseudodifferential operators}
\hskip0.5cm We introduce H\"{o}rmander classes which are function spaces whose symbols define pseudodifferential operators.
\begin{definition}
	Let $\Omega$ be an open subset of $\mathbb{R}^{n}$ and consider $m \in \mathbb{R}$ and $\rho,\delta \in [0,1]$. The space $S_{\rho,\delta}^{m}(\Omega\times\mathbb{R}^{n})$ is defined as follows:
	\begin{multline*}
		S_{\rho,\delta}^{m}(\Omega\times \mathbb{R}^{n})=\Big\{a:\Omega\times\mathbb{R}^{n} \longrightarrow\mathbb{C}\mid a\in C^{\infty}(\Omega\times\mathbb{R}^{n}) \mbox{  and for } \alpha,\beta \in \mathbb{N}^{n}\\
		\mbox{ there is } c_{\alpha,\beta} \mbox{ such that } |\partial^{\alpha}_{\xi}\delta_{x}^{\beta}a(x,\xi)|\leq c_{\alpha,\beta}<\xi>^{m-\rho|\alpha|+\delta|\beta|},\ x \in \Omega\Big \}.
		\end{multline*}
\end{definition}\par
The above classes generalize symbols of differential operators with smooth coefficients. We will be more interested in the case where $m=0$, which is denoted by $S_{\rho,\delta}(\Omega\times\mathbb{R}^{n})$ so that 
\begin{multline*}
S_{\rho,\delta}(\Omega\times\mathbb{R}^{n})=\Big\{a:\Omega\times\mathbb{R}^{n} \longrightarrow\mathbb{C}\mid a\in C^{\infty}(\Omega\times\mathbb{R}^{n}) \mbox{  and for }  \alpha,\beta \in \mathbb{N}^{n},\\
\mbox{ there is }\ c_{\alpha,\beta} \mbox{ such that } |\partial^{\alpha}_{\xi}\delta_{x}^{\beta}a(x,\xi)|\leq c_{\alpha,\beta}<\xi>^{-\rho|\alpha|+\delta|\beta|},\ x \in \Omega\Big \}.
\end{multline*}
Below we introduce pseudodifferential operators associated to symbols in the H\"{o}rmander classes.
\begin{definition}
	Let $\Omega$ be an open subset of $\mathbb{R}^{n}$. For $m \in \mathbb{R}\ and\ \rho,\delta \in [0,1]$. Given $a \in S_{\rho,\delta}^{m}(\Omega\times\mathbb{R}^{n})$, the pseudodifferential operator $\mathcal{A}$ associated to the Kohn-Neirenberg symbol $a$ is defined on the Schwartz space $\mathcal{S}(\mathbb{R}^{n})$ as follows:
	\begin{equation*}
	\mathcal{A}u(x)=\int_{\mathbb{R}^{n}}\int_{\mathbb{R}^{n}}e^{i(x-y)\xi}a(x,\xi)u(y)dyd\xi,\ u \in \mathcal{S}(\mathbb{R}^{n})\ and\ x\in \Omega
	\end{equation*}
\end{definition}
\begin{remark}
	For each $x \in \Omega$, we note that
	\begin{equation}\label{e14}
	\mathcal{A}u(x)=\int_{\mathbb{R}^{n}}e^{ix\xi}a(x,\xi)\hat{u}(\xi)d\xi.
	\end{equation}
	since
	\begin{equation*}
	a(x,.)\hat{u} \in \mathcal{S}(\mathbb{R}^{n}),
	\end{equation*}
	then
	\begin{equation*}
	\mathcal{A}u(x)=\mathcal{F}^{-1}(a(x,.)\hat{u})(x)
	\end{equation*}
	is well defined.
\end{remark}\par
The following example motivates the above definition for pseudodifferential operators.
\begin{example}\label{exp1515}
	Let $p(x,D)$ be a differential operator of order $m \in \mathbb{N}$ with smooth coefficient whose derivatives are bounded i.e. 
	\begin{equation*}
	p(x,D)=\sum_{|\alpha|\leq m}f_{\alpha}(x)D^{\alpha}_{x},\ \partial_{x}^{\beta}f_{\alpha}(x) \in C^{\infty}(\Omega)\cap L^{\infty}(\Omega).
	\end{equation*}
	Let
	\begin{equation*}
	a(x,\xi)=\sum_{|\alpha|\leq m}f_{\alpha}(x)\xi^{\alpha}.
	\end{equation*}
	Given  $\alpha,\ \beta \in \mathbb{N}^{n}$, we have
	\begin{align*}
	|\partial_{\xi}^{\gamma}\partial_{x}^{\beta}a(x,\xi)|&\leq\sum_{|\alpha|\leq m}|\partial_{x}^{\beta}f_{\alpha}(x)|\ |\partial_{\xi}^{\gamma}\xi^{\alpha}|\\
	&\leq c_{\gamma,\beta}<\xi>^{m}.
	\end{align*}
	Hence $a \in S^{m}_{0,0}(\Omega\times\mathbb{R}^{n})$ and by (\ref{e8})
	\begin{align*}
	\mathcal{A}u(x)&=\int_{\mathbb{R}^{n}}e^{ix\xi}a(x,\xi)\hat{u}(\xi)d\xi\\
	&=\int_{\mathbb{R}^{n}}e^{ix\xi}(\sum_{|\alpha|\leq m}f_{\alpha}(x)\xi^{\alpha})\hat{u}(\xi)d\xi\\
	&=\sum_{|\alpha|\leq m}f_{\alpha}(x)\int_{\mathbb{R}^{n}}e^{ix\xi}\xi^{\alpha}\hat{u}(\xi)d\xi\\
	&=\sum_{|\alpha|\leq m}f_{\alpha}(x)D^{\alpha}u(x)\\
	&=p(x,D)u.
	\end{align*}
\end{example}
\begin{remark}
	The above example show that if $a=1$ then $\mathcal{A}u(x)=u(x)$. Moreover, if $a(x,\xi)=f(x)$ i.e. $a$ is independent $\xi$ then $\mathcal{A}$ is a multiplication operator. Note that if $a(x,\xi)=a(\xi)$ is independent of x then following equation (\ref{e14}) we have
	\begin{equation*}
	\mathcal{A}u(x)=\int_{\mathbb{R}^{n}}e^{ix\xi}a(\xi)\hat{u}(\xi)d\xi
	\end{equation*}
	is the convolution by $a$.
\end{remark}
\section{The algebra of Pseudodifferential operators}
\hskip0.5cm Fix $m \in \mathbb{R}\ and\ \rho,\delta \in [0,1]$, and let $\Omega$ be an open subset of $\mathbb{R}^{n}$. We write $\Psi^{m}_{\rho,\delta}$ for the class of pseudodifferential operators with symbols in the H\"{o}rmander class $S^{m}_{\rho,\delta}(\Omega\times\mathbb{R}^{n})$.\\
On $S^{m}_{\rho,\delta}(\Omega\times\mathbb{R}^{n})$ we consider the collection of semi norms given by 
\begin{equation*}
Q_{k}(a)=\sup\{|<\xi>^{-m+\rho|\alpha|-\delta|\beta|}\partial_{\xi}^{\alpha}\partial_{x}^{\beta}a(x,\xi)|,\ x\in \Omega,\ \xi \in \mathbb{R}^{n},\ |\alpha|\leq k,\ |\beta|\leq k  \}.
\end{equation*}\par
The following theorem shows that Pseudodifferential operators are bounded on the Schwartz space.
\begin{theorem}
	Let $a \in S^{m}_{\rho,\delta}(\Omega\times\mathbb{R}^{n})$ and $\mathcal{A}$ be the associated pseudodifferential operator. Giving $u\in \mathcal{S}(\mathbb{R}^{n})$ we have $\mathcal{A}u \in \mathcal{S}(\mathbb{R}^{n})$ and the map $\mathcal{A}:\mathcal{S} \longrightarrow \mathcal{S}$ is continuous.\\
	\begin{proof}
		Let $u \in \mathcal{S}(\mathbb{R}^{n})$ be fixed. We first show that $\mathcal{A}$u is a bounded function
		\begin{align*}
		|\mathcal{A}u(x)|&=|\int e^{ix\xi}a(x,\xi)\hat{u}(\xi)d\xi|\\
		&\leq \int |a(x,\xi)||\hat{u}(\xi)|d\xi\\
		&=\int |<\xi>^{-m}a(x,\xi)||<\xi>^{m+n+1}\hat{u}(\xi)|\frac{d\xi}{<\xi>^{n+1}}\\
		&=cq_{|m|+n+1}(u)+Q_{0}(a).
		\end{align*}
		It is easy to check that the derivation under integral sign can be done and we obtain:
		\begin{align*}
		\partial_{x_{j}}\mathcal{A}u(x)&=\int e^{ix\xi}i\xi_{j}a(x,\xi)\hat{u}(\xi)d\xi +\int e^{ix\xi} \partial_{x_{j}}a(x,\xi)\hat{u}(\xi)d\xi\\
		&=i\mathcal{A}(\xi_{j},\hat{u})+\mathcal{B}u(x)
		\end{align*}
		where $b(x,\xi)=\partial_{x_{j}}a(x,\xi)\in S^{m+\delta|\beta|}_{\rho,\delta}(\Omega\times\mathbb{R}^{n})$, so that
		\begin{align*}
		\partial_{x_{j}}\mathcal{A}u(x)&\leq cq_{|m|+n+1}(\xi_{j},u)Q_{0}(a)+cq_{|m|+n+1}(u)Q_{0}(\partial_{x_{j}}a)\\
		&\leq cq_{|m|+n+2}(u)Q_{1}(a).
		\end{align*}
		Moreover, using integration by parts, we obtain
		\begin{align*}
		|x_{j}\mathcal{A}u(x)|&=|\int e^{ix\xi}\partial_{\xi_{j}}a(x,\xi)\hat{u}(\xi)d\xi+\int e^{ix\xi} a(x,\xi)\partial_{\xi_{j}}\hat{u}(\xi)d\xi|\\
		&\leq cq_{|m|+n+1}(u)Q_{0}(\partial_{\xi_{j}}a)+cq_{|m|+n+1}(x_{j}u)Q_{0}(a)\\
		&\leq cq_{|m|+n+2}(u)Q_{1}(a).
		\end{align*}
		Continuing in this way we obtain $\mathcal{A}u \in \mathcal{S}(\mathbb{R}^{n})$ and $\mathcal{A}$ is continuous on $\mathcal{S}(\mathbb{R}^{n})$.
	\end{proof}
\end{theorem}\par
We aim to show that the space $\Psi_{\rho,\delta}^{0}$ forms an operator algebra. For this reason we start by investigating some properties of the H\"{o}rmander classes $S_{\rho,\delta}^{m-\rho|\alpha|+\delta|\beta|}(\Omega\times \mathbb{R}^{n})$.
\begin{proposition}
	Let $a\in S_{\rho,\delta}^{m}$ then $\partial_{\xi}^{\alpha}\partial_{x}^{\beta}a(x,\xi) \in S_{\rho,\delta}^{m-\rho|\alpha|+\delta|\beta|}(\Omega\times \mathbb{R}^{n})$.\\
	\begin{proof}
		Given $\gamma,\mu \in \mathbb{N}^{n}$ we have
		\begin{align*}
		|\partial_{\xi}^{\gamma}\partial_{x}^{\mu}(\partial_{\xi}^{\alpha}\partial_{x}^{\beta}a(x,\xi))|&=|\partial_{\xi}^{\gamma+\alpha}\partial_{x}^{\mu+\beta}a(x,\xi)|\\
		&\leq c<\xi>^{m-\rho|\gamma+\alpha|+\delta|\mu+\beta|}\\
		&=c<\xi>^{(m-\rho|\alpha|+\delta|\beta|)-\rho|\gamma|+\delta|\mu|}.
		\end{align*}
	\end{proof}
\end{proposition}
\begin{proposition}
	Let $a\in S_{\rho_{1},\delta_{1}}^{m_{1}}$ and $b\in S_{\rho_{2},\delta_{2}}^{m_{2}}$ then $ab\in S_{\rho,\delta}^{m_{1}+m_{2}}$ where $\rho=\min\{\rho_{1},\rho_{2} \}$ and $\delta=\max \{\delta_{1},\delta_{2} \}$.\\
	\begin{proof}
		Let $\alpha,\beta \in \mathbb{N}^{n}$, then using Leibniz formula we have:
		\begin{equation*}
		\partial_{x}^{\beta}a(x.\xi)b(x,\xi)=\sum_{\gamma \leq \beta}C_{\beta}^{\gamma}\partial_{x}^{\gamma}a(x,\xi)\partial_{x}^{\beta-\gamma}b(x,\xi)
		\end{equation*}
		so that
		\begin{align}\label{a1}
		|\partial_{\xi}^{\alpha}\partial_{x}^{\beta}a(x,\xi)b(x,\xi)|&=|\sum_{\gamma \leq \beta}C_{\beta}^{\gamma}\partial_{\xi}^{\alpha} \partial_{x}^{\gamma}a(x,\xi)\partial_{x}^{\beta-\gamma}b(x,\xi)|\notag\\
		&=|\sum_{\gamma \leq \beta}\sum_{\mu \leq \alpha}C_{\beta}^{\gamma}C_{\alpha}^{\mu}\partial_{\xi}^{\mu}\partial_{x}^{\gamma}a(x,\xi)\partial_{\xi}^{\alpha-\mu}\partial_{x}^{\beta-\gamma}b(x,\xi)|\notag\\
		&\leq \sum_{\gamma \leq \beta}\sum_{\mu \leq \alpha}C_{\beta}^{\gamma}C_{\alpha}^{\mu}|\partial_{\xi}^{\mu}\partial_{x}^{\gamma}a(x,\xi)\partial_{\xi}^{\alpha-\mu}\partial_{x}^{\beta-\gamma}b(x,\xi)|\notag\\
		&\leq \sum_{\gamma \leq \beta}\sum_{\mu \leq \alpha}C_{\beta}^{\gamma}C_{\alpha}^{\mu}c_{\mu,\gamma}c_{\alpha-\mu,\beta-\gamma}<\xi>^{m_{1}-\rho_{1}|\mu|+\delta_{1}|\gamma|}<\xi>^{m_{2}-\rho_{2}|\alpha-\mu|+\delta_{2}|\beta-\gamma|}.
		\end{align}
		We note that
		
		$$m_{1}-\rho_{1}|\mu|+\delta_{1}|\gamma|+m_{2}-\rho_{2}|\alpha-\mu|+\delta_{2}|\beta-\gamma|$$$$=m_{1}+m_{2}-(\rho-{1}|\mu|+\rho_{2}|\alpha|-\rho_{2}|\mu|)+(\delta_{1}|\gamma|+\delta_{2}|\beta|-\delta_{2}|\gamma|).$$
	We also have
		\begin{equation*}
		\delta_{1}-\delta_{2} \leq \delta-\delta_{2}\ and\ |\gamma|\leq |\beta|
		\end{equation*}
		so that
		\begin{equation*}
		(\delta_{1}-\delta_{2})|\gamma| \leq (\delta-\delta_{2})|\beta|
		\end{equation*}
		hence
		\begin{equation}\label{a2}
		\delta_{1}|\gamma|-\delta_{2}|\gamma|+\delta_{2}|\beta| \leq \delta|\beta|.
		\end{equation}
		Similarly
		\begin{equation*}
		\rho_{2}-\rho_{1} \leq \rho_{2}-\rho\ and\ |\mu|\leq |\alpha|
		\end{equation*}
		so that 
		\begin{equation*}
		(\rho_{2}-\rho_{1})|\mu| \leq (\rho_{2}-\rho)|\alpha|
		\end{equation*}
		hence
		\begin{equation}\label{a3}
		\rho_{2}|\mu|-\rho_{1}|\mu|-\rho_{2}|\alpha|\leq -\rho|\alpha|
		\end{equation}
		Using (\ref{a2}) and (\ref{a3}) in the estimate (\ref{a1}), we obtain
		\begin{equation*}
		|\partial_{\xi}^{\alpha}\partial_{x}^{\beta}a(x,\xi)b(x,\xi)|\leq c<\xi>^{m_{1}+m_{2}-\rho|\alpha|+\delta|\beta|}
		\end{equation*} 
	\end{proof}
\end{proposition}
\begin{corollary}
	The collection of all H\"{o}rmander symbol classes $S_{\rho,\delta}^{-\infty}=\cup_{m\in \mathbb{R}}S_{\rho,\delta}^{m}$ is an algebra over the pointwise multiplication and addition.
\end{corollary}
\begin{corollary}
	The symbol class $S_{\rho,\delta}^{0}$ is a Fr\'{e}chet  algebra under pointwise multiplication and addition. The topology is given by the collection of semi-norms $\{Q_{k}\}_{k\in\mathbb{N}}$ defined by
	\begin{equation*}
	Q_{k}(a)=\sup\{|<\xi>^{-m+\rho|\alpha|-\delta|\beta|}\partial_{\xi}^{\alpha}\partial_{x}^{\beta}a(x,\xi)|,\ x\in \Omega,\ \xi \in \mathbb{R}^{n},\ |\alpha|\leq k\ and\ |\beta|\leq k \}.
	\end{equation*}
\end{corollary}\par
We are interested in the continuity of Pseudodifferential operator on the Lebesgue space $L^{2}(\Omega)$. Of course this is equivalent to the fact of existence of a positive number $c$ satisfying
\begin{equation*}
\|\mathcal{A}u\|_{L^{2}}\leq c\|u\|_{L^{2}},\ for\ all\ u\in \mathcal{S}(\mathbb{R}^{n}).
\end{equation*}
The continuity is then understood as the (unique) extension from $\mathcal{S}(\mathbb{R}^{n})$ to $L^{2}(\mathbb{R}^{n})$.\\
The below theorem can be found in \cite{ege}. However the original proof for the case $\delta = 0\ and\ \rho > 0$ is due to H\"{o}rmander \cite{hor}.
\begin{theorem}
	Let $\rho,\delta \in [0,1]$ with $\delta<\rho$ and $a=a(x,\xi) \in S_{\rho,\delta}(\Omega\times\mathbb{R}^{n})$ with the $supp\ a \subset\{x \in \Omega,\ |x|\leq c \}$. Then $\mathcal{A}:L^{2}(\mathbb{R}^{n})\longmapsto L^{2}(\mathbb{R}^{n})$ is continuous.\\
	\begin{proof}
		We give the proof for the case $\delta=0$. We consider the Fourier transform for a in the x-variable, we have:
		\begin{equation*}
		[\mathcal{F}a(.,\xi)](\mu)=\int_{\mathbb{R}^{n}}a(x,\xi)e^{-ix\mu}dx=\int_{K}a(x,\xi)e^{-ix\mu}dx,\ where\ supp\ a \subset K\ and\ K\ is\ compact.
		\end{equation*}
		Hence, for $\alpha \in \mathbb{N}^{n}$ we have
		\begin{align}\label{e15}
		|\mu^{\alpha}[\mathcal{F}a(.,\xi)](\mu)|&=|\mathcal{F}[\partial^{\alpha}a(.,\xi)](\mu)|\notag\\
		&=|\int_{K}\partial_{x}^{\alpha}a(x,\xi)e^{-ix\mu}dx|\notag\\
		&\leq c<\xi>^{-\rho|\alpha|}.
		\end{align}
		As $\rho>0$ we choose $\alpha \in \mathbb{N}^{n}$ such that $\rho|\alpha|>n$. By the inequality (\ref{e15}) we have
		\begin{equation*}
		|\mathcal{F}a(.,\xi)(\mu)|\leq c<\xi>^{-\rho|\alpha|}<\mu>^{-|\alpha|}
		\end{equation*}
		But
		\begin{equation*}
		a(x,\xi)=\int_{\mathbb{R}^{n}}[\mathcal{F}a(.,\xi)](\mu)e^{ix\mu}d\mu.
		\end{equation*}
		So that
		\begin{align*}
		|a(x,\xi)|&\leq \int_{\mathbb{R}^{n}}|\mathcal{F}a(.,\xi)|(\mu)d\mu\\
		&\leq c\int_{\mathbb{R}^{n}}<\xi>^{-\rho|\alpha|}<\mu>^{-|\alpha|}d\mu\\
		&\leq c<\xi>^{-\rho|\alpha|}\int_{\mathbb{R}^{n}}<\mu>^{-|\alpha|}d\mu.
		\end{align*}
		Hence
		\begin{align*}
		|\mathcal{A}u(x)|^{2}&=|\int_{\mathbb{R}^{n}}e^{ix\xi}a(x,\xi)\hat{u}(\xi)d\xi|^{2}\\
		&=(\int_{\mathbb{R}^{n}}|a(x,\xi)||\hat{u}(\xi)d\xi)^{2}\\
		&\leq \int_{\mathbb{R}^{n}}|a(x,\xi)|^{2}d\xi.\int_{\mathbb{R}^{n}}|\hat{u}(\xi)|^{2}d\xi\\
		&=\|\hat{u}\|^{2}_{L^{2}}\int_{\mathbb{R}^{n}}|a(x,\xi)|^{2}d\xi\\
		&\leq \|\hat{u}\|^{2}_{L^{2}}\int_{\mathbb{R}^{n}}(c<\xi>^{-\rho|\alpha|}\int_{\mathbb{R}^{n}}<\mu>^{-|\alpha|}d\mu)^{2}d\xi\\
		&=c^{2}\|\hat{u}\|^{2}_{L^{2}}\int_{\mathbb{R}^{n}}<\xi>^{-2\rho|\alpha|}(\int_{\mathbb{R}^{n}}<\mu>^{-|\alpha|}d\mu)^{2}d\xi\\
		&=C\|\hat{u}\|_{L^{2}}^{2}
		\end{align*}
		Where $C=c^{2}\int_{\mathbb{R}^{n}}<\xi>^{-2\rho|\alpha|}(\int_{\mathbb{R}^{n}}<\mu>^{-|\alpha|}d\mu)^{2}d\xi<\infty$ because $n<\rho|\alpha|\leq |\alpha|$. Using Plancheral theorem we obtain
		\begin{equation*}
		\|\mathcal{A}u\|_{L^{2}}^{2}=\int_{K}|\mathcal{A}u(x)|^{2}dx\leq C\|\hat{u}\|^{2}_{L^{2}}=C\|u\|_{L^{2}}^{2}.
		\end{equation*}
	\end{proof}
\end{theorem}
\begin{remark}
	The above theorem is still true even when removing the boundedness of support of $a$ in the x-variable. But the condition $\delta<\rho$ is essential. However, the proof is more technical and is due to Calderon-Vaillancourt \cite{cal}.
\end{remark}
\begin{theorem}\label{t5}
	Let $\mathcal{A} \in \Psi^{0}_{\rho,\delta}$ with $0\leq \delta<\rho\leq 1$, then $\mathcal{A}$ is bounded on $L^{2}(\Omega)$.
\end{theorem}\par
In order to investigate wether the composite of two pseudo-differential operators is also a pseudo-differential operator we need to go through pseudo-differential operator with double symbols.
\begin{theorem}
	Let $a,b \in S^{0}_{\rho,\delta}(\Omega \times \mathbb{R}^{n})$ then
	\begin{equation}
	c(x,\mu):=\int\int e^{izw} a(x,w+\mu)b(z+x,\mu)dzdw\ \in\ S^{0}_{\rho,\delta}(\Omega \times \mathbb{R}^{n})
	\end{equation}\\
	\begin{proof}
		Using Leibniz rule together with Peetre's inequality (\ref{p10}), for each $\alpha,\beta \in \mathbb{N}^{n}$
		\begin{equation*}
		\partial_{x}^{\beta}a(x,w+\mu)b(z+x,\mu)=\sum_{\gamma \leq \beta}C_{\beta}^{\gamma}\partial_{x}^{\gamma}a(x,w+\mu)\partial_{x}^{\beta-\gamma}b(z+x,\mu).
		\end{equation*}
		then
		\begin{align*}
		|\partial_{\mu}^{\alpha}\partial_{x}^{\beta}a(x,w+\mu)b(z+z,\mu)|
		\leq \sum_{\gamma \leq \beta}\sum_{r \leq \alpha} C_{\beta}^{\gamma}C_{\alpha}^{r}|\partial_{\mu}^{r}\partial_{x}^{\gamma}a(x,w+\mu)|\ |\partial_{\mu}^{\alpha-r}\partial_{x}^{\beta-\gamma}b(z+x,\mu)|\\
		\leq \sum_{\gamma \leq \beta}\sum_{r \leq \alpha} C_{\beta}^{\gamma}C_{\alpha}^{r}c_{r,\gamma}c_{\alpha-r,\beta-\gamma}<w+\mu>^{-\rho|r|+\delta|\gamma|}<\mu>^{-\rho|\alpha-r|+\delta|\beta-\gamma|}\\
		\leq \sum_{\gamma \leq \beta}\sum_{r \leq \alpha} C_{\beta}^{\gamma}C_{\alpha}^{r}c_{r,\gamma}c_{\alpha-r,\beta-\gamma}2^{|-\rho|r|+\delta|\gamma||}<w>^{|-\rho|r|+\delta|\gamma||}<\mu>^{-\rho|r|+\delta|\gamma|}<\mu>^{-\rho|\alpha-r|+\delta|\beta-\gamma|}\\
		\leq c<w>^{|-\rho|r|+\delta|\gamma||}<\mu>^{-\rho|\alpha|+\delta|\beta|}.
		\end{align*}
		this shows that
		\begin{equation*}
		|\partial_{\mu}^{\alpha}\partial_{x}^{\beta}c(x,\mu)|\leq C<\mu>^{-\rho|\alpha|+\delta|\beta|}.
		\end{equation*}
	\end{proof}
\end{theorem}
\begin{theorem}
	Let $\mathcal{A},\mathcal{B}\in \Psi_{\rho,\delta}^{0}$ with $0\leq \delta<\rho\leq 1$, then $\mathcal{A}\mathcal{B} \in \Psi_{rho,\delta}^{0}$.\\
	\begin{proof}
		Let $a,b\in S_{\rho,\delta}^{0}(\Omega\times\mathbb{R}^{n})$ be the Kohn-Neirenberg symbol associated to $\mathcal{A}$ and  $\mathcal{B}$, respectively. Given  $u\in \mathcal{S}(\mathbb{R}^{n})$ then
		\begin{align*}
		\mathcal{A}\mathcal{B}u(x)&=\mathcal{A}(\mathcal{B}u)(x)\\
		&=\mathcal{A}[\int e^{ix\mu}b(y,\mu)\hat{u}(\mu)d\mu](x)\\
		&=\int e^{ix\xi}a(x,\xi)\mathcal{F}(\int e^{iy\mu}b(y,\mu)\hat{u}(\mu)d\mu)(\xi)d\xi\\
		&=\int e^{ix\xi}a(x,\xi)\int e^{-iy\xi}\int e^{iy\mu}b(y,\mu)\hat{u}(\mu)d\mu dyd\xi\\
		&=\int e^{ix\xi}a(x,\xi)\int e^{-iy\xi}\int e^{iy\mu}b(y,\mu)\int e^{-i\mu \gamma}u(\gamma)d\gamma d\mu dyd\xi\\
		&=\int\int e^{i\xi(x-y)}a(x,\xi)\int\int e^{i\mu(y-\gamma)}b(y,\mu)u(\gamma)d\gamma d\mu dyd\xi\\
		&=\int\int e^{i\xi(x-y)}a(x,\xi)\int\int e^{i\mu\theta}b(y,\mu)u(\theta+y)d\theta d\mu dyd\xi\\
		&=\int\int\int\int e^{-i\xi z-i\mu \theta}a(x,\xi)b(z+x,\mu)u(z+\theta+x)d\theta d\mu dzd\xi\\
		\end{align*}
where we use the change of variables $\gamma-y=\theta$ and $y-x=z$. Introduce the new change of variable $z+\theta=v$ and $\xi-\mu=w$ we obtain $z\xi+\theta\mu=zw+v\mu$ hence
		\begin{align*}
		\mathcal{A}\mathcal{B}u(x)&=\int\int\int\int e^{-izw-iv\mu}a(x,w+\mu)b(z+x,\mu)u(v+x)dvd\mu dzdw\\
		&=\int\int e^{-iv\mu}\int\int e^{-izw}a(x,w+\mu)b(z+x,\mu)dzdw\ u(v+x)dvd\mu\\
		&=\int\int e^{-iv\mu}c(x,\mu)\ u(v+x)dvd\mu,
		\end{align*}
where $c(x,\mu)$ is the symbol given by the above theorem. Using the change of variable $t=v+x$ we obtain:
		\begin{align*}
		\mathcal{A}\mathcal{B}u(x)&=\int\int e^{-i(t-x)\mu}c(x,\mu)u(t)dtd\mu\\
		&=\mathcal{C}\ u(x).
		\end{align*}
		where $\mathcal{C}$ is the Pseudo-differential operator with Kohn-Neirenberg symbol $c\in S_{\rho,\delta}^{0}(\Omega\times\mathbb{R}^{n})$ i.e. $\mathcal{C}=\mathcal{A}\mathcal{B}\in \Psi_{\rho,\delta}^{0}$.
	\end{proof}
\end{theorem}\par
Combining theorem (\ref{t5}) and the above theorem, we obtain:
\begin{corollary}
	Given $\rho,\delta\in [0,1]$ with $\delta<\rho$ then the collection of all Pseudo-differential operators of order zero $\Psi_{\rho,\delta}^{0}(\Omega)$ is an operator subalgebra in $\mathcal{L}(L^{2}(\Omega))$.
\end{corollary}\par
We now prove that $\Psi_{\rho,\delta}^{0}(\Omega)$ is a subalgebra of $\mathcal{L}(L^{2}(\Omega))$.
\begin{theorem}
	Let $\mathcal{A}\in \Psi_{\rho,\delta}^{0}(\Omega)$, then the adjoint operator $\mathcal{A}^{\star}\in \Psi_{\rho,\delta}^{0}(\Omega) $\\
	\begin{proof}
		Let $u,v\in \mathcal{S}(\mathbb{R}^{n})$ then
		\begin{align*}
		<\mathcal{A}u,v>_{L^{2}}&=\int\mathcal{A}u(x)\overline{v(x)}dx\\
		&=\int\int e^{ix\xi}a(x,\xi)\hat{u}(\xi)d\xi \overline{v(x)}dx\\
		&=\int\mathcal{F}(u)(\xi)\overline{\int e^{-ix\xi}\overline{a(x,\xi)}v(x)dx}d\xi\\
		&=\int\int e^{-iy\xi}u(y)dy\overline{\int e^{-ix\xi}\overline{a(x,\xi)}v(x)dx}d\xi\\
		&=\int u(y)\overline{\int\int e^{-i(y-x)\xi}\overline{a(x,\xi)}v(x)dxd\xi}dy
		\end{align*}
		so that
		\begin{align*}
		\mathcal{A}^{\star}v(y)&=\int\int e^{i(y-x)\xi}\overline{a(x,\xi)}v(x)dxd\xi
		\end{align*}
	\end{proof}
\end{theorem}
Combining the last corollary and theorem we obtain
\begin{corollary}
	Let $\rho,\delta \in [0,1]$ with $\delta<\rho$ then $\Psi^{0}_{\rho,\delta}$ is a $\Psi^{\star}-algebra$ in the $C^{\star}-algebra\ \mathcal{L}(L^{2}(\Omega))$.
\end{corollary}\par
For the spectral invariance proof, we refer the reader to Weiner's lemma for pseudodifferential operators which can be found in \cite{fors,hor,sjo}.
\chapter{Bergman spaces}
\section{RKHS and Bergman spaces}
Reproducing Kernel Hilbert Spaces (RKHS) have arise in various areas in functional analysis, representation theory, statistics and approximation theory. Through this chapter you will study reproducing kernels spaces with its properties and Bergman spaces as a particular example of it.\\
\begin{definition}
	A reproducing kernel Hilbert space $H$ is a Hilbert Space of complex valued functions on an open subset $\Omega$ of $\mathbb{C}^n$ such that the Dirac evaluation function at  every point is continuous.
\end{definition}
\begin{remark}
	Let	$H$ be a RKHS then  by Reisz Representation theory there is a unique $K_z \in H $ such that 
	$$\delta_z(f)=f(z)=<f,K_z> ;\forall f \in H.$$
\end{remark}
\begin{definition}
	The reproducing kernel of $H$ is a 2-variable function over $\Omega\times\Omega$ defined by $K(w,z)=K_z(w)$.
\end{definition}
\begin{proposition}
	Let $H$ be a RKHS of holomorphic  functions. Then the reproducing of $H$ if given by $K(w,z) = \overline{\rm K(z,w)}$ i.e. $K(z,w)$ is holomorphic in $z$ and anti-holomorphic in $w$. 
\end{proposition}
\begin{definition}
	Let $U$ be an open subset of $\mathbb C^n$ and $\alpha$ be a weight function such that $\alpha(z)dv(z)$ is a probability measure on $U$. The Bergman space over $U$, denoted by $HL^2(U,\alpha)$,   is the space  of holomorphic functions on $U$ which are square integrable with the respect $\alpha(z)dv(z)$ i.e. 
	$$HL^2(U,\alpha)=\{f:U\longrightarrow\mathbb{C}\mid \int_U |F(z)|^2 \alpha(z)dz<\infty \}$$ 
\end{definition}
\begin{lemma}
	The Bergman space  $ HL^2(U,\alpha)$ is a reproducing kernel Hilbert space of holomorphic functions.
\end{lemma}
\begin{proof} We provide the proof for the case $n=1$. For a fixed $z \in U$. Let $B(z,r)$ be the open disk of center $z$ and radius $r$  so that $\overline{B(z,r)}\subset U$, then we can write $$F(z) = \frac{1}{2\pi}\int_{B(z,r)}F(v)dv$$ by Mean Value Theorem for holomorphic functions.
	Using Taylor series at $v=z$, $F(v)=F(z)+\sum_{n=1}^{\infty}a_n(v-z)^n$. As $\overline{B(z,r)}\subset U$ is a compact set then this series converges uniformly on $\overline{B(z,r)}$ to $F$.$$\frac{1}{\pi r^2}\int_{B(z,r)} F(v) dv=\frac{1}{\pi r^2}\int_{B(z,r)} F(z)+\sum_{n=1}^{\infty}a_n(v-z)^n dv$$
	$$=\frac{1}{\pi r^2}\int_{B(z,r)} F(z)dv + \sum_{n=1}^{\infty}a_n\int_{B(z,r)} (v-z)^n dv$$ $$= \frac{1}{\pi r^2}\pi r^2 F(z) + \sum_{n=1}^{\infty}a_n\int_{0}^{2\pi}\int_{0}^{1} r^n e^{in\theta}rdrd\theta$$ $$=F(z) + \sum_{n=1}^{\infty}a_n\int_{0}^{2\pi}\int_{0}^{1} r^{n+1} e^{in\theta} drd\theta= F(z)+\sum_{n=1}^{\infty}a_n\int_{0}^{2\pi}e^{in\theta} d\theta$$ $$=F(z) + 0 = F(z).$$
	This shows that the pointwise evaluation map is  continuous linear function on $HL^2(U,\alpha)$.
\end{proof}
\begin{proposition}
	If $F\in L^2(U,\alpha)$, the orthogonal projection of F into the closed subspace $HL^2(U,\alpha)$ denoted by $PF$ is given by $$PF(z) = \int_U K(z,w) F(w)\alpha(w) dw.$$
	\begin{proof}
		Given 	$F\in L^2$  write $F = f + g$ where $f \in HL^2$ and $g \in (HL^2)^\perp.$
		As $f \in HL^2(U,\alpha)$ then $$Pf(z)=f(z)=\int_U K(z,w) f(w)\alpha(w) dw.$$  However, as $k_z\in HL^2$  $$\int_U K(z,w) g(w) \alpha(w)dw = <g,k_z>_{L^2(U,\alpha)} = 0.$$
	\end{proof}
\end{proposition}
Below, we collect some basic properties of the kernel $K(w,z)$
\begin{itemize}
	\item[1.] $\|K_z \| = \sqrt{K(z,z)}$
	\item[2.] $|K(w,z)| \leq \sqrt{K(z,z)}\sqrt{K(w,w)}$
	\item[3.] $|F(z)|^2 \leq K(z,z) \|F\|^2$
	\item[4.] $\int_U K(z,w) K(w,u) \alpha(w) dw= K(z,u)$
\end{itemize}
The reproducing kernel of a separable can be calculated as the following proposition ensures.
\begin{proposition}
	For any complete orthogonal system $\{e_n\}_{n\in \mathbb{N}}$ of a separable RKHS  we have $$K(w,z) = \sum_{n \in \mathbb{N}} e_n(w)\overline{e_n(z)}$$ which converges absolutely on $U\times U$.\\
	\begin{proof}
		For any $z \in U$,  we have $K_z = \sum_{n \in \mathbb{N}} <K_z, e_n> e_n$ and hence
		$$K_z(w) = <K_z,K_w> = \sum_{n \in \mathbb{N}} <K_z,e_n><e_n,K_w> = \sum_{n \in \mathbb{N}} \overline{e_n(z)} e_n(w).$$
	\end{proof}
\end{proposition}
\begin{remark}
	Consider the Bergman space $HL^2(U,\alpha)$ and let $K(w,z)$ denotes the reproducing kernel. Then the following holds true
	\begin{itemize}
		\item[1.] $K(w,z)$ is analytic in w and anti-analytic in z.
		\item[2.] H is separable.
		\item[3.] $K(w,z)$ converges uniformly on compact sets of $U\times U$.
	\end{itemize}
\end{remark}
\section{Bergman space over the unit ball}
In this section we consider the case of the Bergman space over the unit ball of $\mathbb C^n$. We are interested in calculating the reproducing kernel starting from the fact that the space of holomorphic polynomials are dense in the Bergman space. Our calculations are done for the unit disc (with standard measure) and by a slight generalization we obtain the kernel in the higher dimensional cases. The results of this section are well known in literature (cf. \cite{zhu}). 
\begin{theorem}
	Let $D$ be the unit disk, i.e. $D = \{z \in \mathbb{C} ; |z| < 1\}$ and $HL^2 (D, \frac{1}{\pi}(1-|z|^2)^a), a > -1$ is the weighted Bergman space. Then the reproducing kernel  is given by \begin{equation*}
	K\left( z,w\right) =\frac{1}{\left( 1-w\overline{z}\right) ^{2+a }}
	\end{equation*}\end{theorem}
	\begin{proof}
		The prove of this needs few steps. We start by proving $\{z_n\}^\infty_{n=0}$ is an orthogonal basis for $HL^2 (D, (1-|z|^2)^a)$.$$<z^n, z^m> = \int_D \overline{z^n} z^m \alpha (z)dz = \int_{0}^{2 \pi} \int_{0}^{1} r^n e^{-in\theta} r^m e^{im\theta} (1-r^2)^a r dr d\theta$$ $$= \int_{0}^{1} r^{n+m+1} (1-r^2)^a \int_{0}^{2 \pi} e^{i(m-n)\theta}d\theta dr$$
		For $m\neq n$ we know that $\int_{0}^{2 \pi} e^{i(m-n)\theta} d\theta = [\frac{1}{m-n} e^{i(m-n)\theta}]^{2 \pi}_0 = 0$. Now let $F \in HL^2 (D, (1-|z|^2)^a)$ with $<z^n, F> = 0$ then $F(z) = \sum_{n=0}^{\infty} c_n z^n$. As $F(z) = \sum_{n=0}^{\infty} c_n z^n$ converges uniformly on a compact subset of $D$ we obtain 
		\begin{align*}
		<z^m, F> &= \int_{0}^{1} \int_{0}^{2 \pi} r^m e^{-im\theta} F(re^{i\theta}) (1-r^2)^a r dr d\theta\\
		&=\lim_{x\to 1} \int_{0}^{x} \int_{0}^{2 \pi} r^m e^{-im\theta} F(re^{i\theta}) (1-r^2)^a r dr d\theta\\
		&=\lim_{x\to 1} \int_{0}^{x} \int_{0}^{2 \pi} r^m e^{-im\theta} \sum_{n=0}^{\infty}(c_n r^n e^{in\theta}) (1-r^2)^a r dr d\theta\\
		&=\lim_{x\to 1} \sum_{n=0}^{\infty} \int_{0}^{x} \int_{0}^{2 \pi} r^m e^{-im\theta} c_n r^n e^{in\theta} (1-r^2)^a r dr d\theta\\
		&=\lim_{x \to 1} \sum_{n=0}^{\infty}\int_{0}^{x} r^{m+n+1} (1-r^2)^a c_n \int_{0}^{2 \pi} e^{i(m-n)\theta}d\theta dr\\
		&=\lim_{x \to 1} \sum_{n=0}^{\infty}\int_{0}^{x} r^{m+n+1} (1-r^2)^a c_n \int_{0}^{2 \pi} e^{i(m-n)\theta}d\theta dr\\
		& = \int_{0}^{1} r^{m+n+1} (1-r^2)^a c_m (2\pi) dr\\
		&=2 \pi c_m \int_{0}^{1} r^{2m+1}(1-r^2)^a dr.
		\end{align*}
		This shows that $c_m = 0$ for all $m$ and therefore $F=0$. Hence, $\{z^m\}$ is an orthogonal basis for $HL^2 (D, (1-|z|^2)^a)$.\\
		Our next step is to normalize the basis.\\
		$$\|z^n\| =\frac{1}{\pi} \int_{0}^{1} \int_{0}^{2 \pi} r^n e^{-in\theta}r^n e^{in\theta}(1-r^2)^a rd\theta dr$$ $$=\int_{0}^{1} 2\ r^{2n+1} (1-r^2)^a dr = 2 \frac{\Gamma(a+1)\Gamma(n+1)}{2 \Gamma(a+n+2)}.$$  So that $$\{z^n\sqrt{\frac{\Gamma(a+n+2)}{\Gamma(a+1)\Gamma(n+1)}}\}$$ is an orthonormal basis for $HL^2(D, (1-|z|^2)^a)$.\\
		Our final step is computing the reproducing kernel
		\begin{align*}
		K(z, w)&= \sum_{n=0}^{\infty} z^n \sqrt{\frac{\Gamma(a+n+2)}{\Gamma(a+1)\Gamma(n+1)}}{\overline{w}}^n\sqrt{\frac{\Gamma(a+n+2)}{\Gamma(a+1)\Gamma(n+1)}}\\
		&=\frac{1}{\Gamma(a+1)}\sum_{n=0}^{\infty} \frac{\Gamma(a+n+2)}{\Gamma(n+1)}(z\overline{w})^n\\
		&=\frac{1}{\left( 1-w\overline{z}\right) ^{2+a }}.
		\end{align*}
	\end{proof}
The generalization for higher dimensional case is given in the below definition and remark.
\begin{definition}
	Let $B_{n} \subset \mathbb{C}^{n}$ denotes the unit open ball of $\mathbb{C}^{n}$, and consider the space $\mathbf{L}^{2}(B_{n}):=\mathbf{L}^{2}(B_{n},d\mu)$ of square-integrable complex valued functions on $B_{n}$, where $\mu$ is the Lebesgue measure on $\mathbb{C}^{n}$ normalizing $B_{n}$ i.e. $\mu(B_{n})=1$. Then the space of entire function in $\mathbf{L}^{2}(B_{n})$ is called the Bergman space over the unit ball of $\mathbb{C}^{n}$ and is denoted by $\mathbb{H}^{2}(B_n)$.
\end{definition}	
\begin{remark}
	The Bergman space $\mathbb{H}^{2}(B_n)$ has the reproduction kernel $k(z,w)=\dfrac{1}{(1-w.\overline{z})^{n+1}}$ and the Bergman projection \begin{equation*}
	P:\mathbf{L}^{2}(B_{n}) \longrightarrow \mathbb{H}^{2}(B_n)
	\end{equation*}
	is given by
	\begin{equation}\label{eq67}
	(Pf)(z)f= <f,k_{z}>= \int_{B^{n}} \dfrac{f(w)}{(1-z.\overline{w})^{(n+1)}} d\mu(w)
	\end{equation}
\end{remark}
 \section{The Segal Bargmann space}
 This section is devoted to introduce the Segal Bargmann space which is a Bergman space over an unbounded domain. The aim of introducing this space in the thesis is to compare the results of $\Psi^\star$-algebra constructed by Bauer to the construction we shall do in the case of the unit ball.
\begin{definition}
The Segal Bargmann space is defined to be the space	of entire functions on $\mathbb C^n$ which are   square-integrable with respect to the Gaussian measure 
$$\alpha_t(z) = \frac{1}{(\pi t)^n}e^{-\frac{|z|^2}{t}},$$ 
where $t$ is a positive parameter. This space is denoted by $HL^2(\mathbb{C},\alpha_t(z))$.
\end{definition}
\begin{theorem} \label{th56}
	The reproducing kernel of the Segal Bargmann space  $HL^2(\mathbb{C},\alpha_t(z)) $ is given by $$K(z,w) = e^\frac{z\overline{w}}{t}.$$\end{theorem}
	\begin{proof}
		We provide the calculations for the case $n=1$. Given  $n \neq m$, we have
		\begin{align*}
		<z^n, z^m> = \int_{\mathbb{C}} \overline{z^n} z^m \alpha_t(z) dz &= \int_{0}^{\infty}\int_{0}^{2\pi}r^n r^{-in\theta} r^m e^{im\theta}(\frac{e^{-\frac{r^2}{t}}}{\pi t})r d\theta dr\\
		&=\int_{0}^{\infty} r^{n+m+1}(\frac{-e^\frac{r^2}{t}}{\pi t}) \int_{0}^{2 \pi} e^{i(m-n)\theta} d\theta dr = 0.
		\end{align*}
		
		Let $F(z)= \sum_{n=0}^{\infty} c_n z^n$ with $<z^m, F> =0$ for all $m\in\mathbb N$ then 
		\begin{align*}
		<z^m, F>&= \int_{0}^{\infty} \int_{0}^{2\pi} r^m e^{-im\theta}F(re^{i\theta})\frac{1}{\pi t}e^{-\frac{r^2}{t}}r d\theta dr\\
		&=\int_{0}^{\infty}\int_{0}^{2\pi} \sum_{n,m}^{\infty}\frac{1}{\pi t} r^{m+n+1} e^{-\frac{r^2}{t}} c_n e^{i(n-m)\theta} d\theta dr\\
		& =\int_{0}^{\infty}\frac{c_m}{\pi t}(2\pi) r^{2m+1}e^{-\frac{r^2}{t}} dr
		\end{align*}
		This shows that $F\equiv0$. So $\{z^m\}$ is an orthogonal basis of $HL^2(\mathbb{C}, \alpha_t)$. Moreover,
		for $n = 0$, we have
		\begin{align*}
		\|z^n\|^2 &= \int_{0}^{\infty}\int_{0}^{2\pi} 1(\frac{1}{\pi t}) e^{-\frac{r^2}{t}}r d\theta dr\\
		&=\lim_{x \to \infty} \int_{0}^{x} 2r e^{-\frac{r^2}{t}} dr\\
		&= \lim_{x \to \infty}-e^{-\frac{r^2}{t}}]^x_0 = \lim_{x \to \infty} -e^{\frac{x^2}{t}} +e^0 = 1.
		\end{align*}
		 Furthermore, for $n>1$,  we obtain 
		 \begin{align*}
		 \|z^n\|^2 &=\int_{0}^{\infty} \int_{0}^{2\pi} r^n e^{-in\theta} r^n e^{in\theta }(\frac{1}{\pi t}) e^{-\frac{r^2}{t}} r d\theta dr = \frac{2}{t} \int_{0}^{\infty} r^{2n} e^{-\frac{r^2}{t}} r dr\\
		 &= \frac{2}{t}[-\frac{t}{2} r^{2n} e^{-\frac{r^2}{t}}]^\infty_0 + tn \int_{0}^{\infty} r^{2n-1} e^{-\frac{r^2}{t}} dr\\
		 &= 0+ tn (\frac{2}{t}\int_{0}^{\infty}r^{2(n-1)+1} e^{-\frac{r^2}{t}}dr)\\
		 & = tn \|z^{n-1}\|^2.
		 \end{align*} Therefore by induction we obtain $\|z^n\|^= n! t^n$ i.e.  $\{\frac{z^n}{\sqrt{n! t^n}}\} $ is an orthonormal basis. Hence the reproducing kernel is given by $$K(z, w)=\sum_{n=0}^{\infty} \frac{z^n}{\sqrt{n! t^n}} \frac{\overline{w^n}}{\sqrt{n!t^n}} = \sum_{n=0}^{\infty}\frac{1}{n!}(\frac{z\overline{w}}{t})^n = e^{\frac{z\overline{w}}{t}}.$$		
	\end{proof}
	
\chapter{$\Psi^\star$-algebra and the Bergman projection}
This chapter is devoted to solve the problem considered in the thesis. The obtained results are new contribution in the analysis of Toeplitz operators and $\Psi^\star$-algebras.
\section{Motivation and the main problem}
\hskip1cm The characterization of the algebra of pseudodifferential operators $\Psi^{0}_{\rho,\delta}(\Omega)$, ($0\leq \rho \leq \delta \leq 1$ and $\delta<1)$, was given by the method of Beals (cf. \cite{be}) and Meyer(cf. \cite{Me}). Instead of defining the H\"{o}rmander classes $\Psi^{0}_{\rho,\delta}$ as an operators on the symbol classes $S_{\rho,\delta}$, one can describe $\Psi^{0}_{\rho,\delta}$ only by commutator method(cf. \cite{be,Ue}) Roughly speaking:
\begin{equation*}
\Psi^{0}_{\rho,\delta}:=\{a\in \mathcal{L}(H^{0})\mid\ ad(M)^{\alpha}ad(\partial)^{\alpha}(a)\in \bigcap_{s \in \mathbb{R}}\mathcal{L}(H^{s-\rho |\alpha|+\delta|\beta|},H^{s}),\alpha,\beta\in \mathbb{N}^{n} \},
\end{equation*}\par
where $M$ is the multiplication by the coordinates, and $H^{s}$ are the sobolev spaces. It was shown that $\Psi^{0}_{\rho,\delta}\subset \mathcal{L}(L^{2}(\mathbb{R}^{n}))$ is spectral invariant Fr\'{e}chet  algebra.  As a generalization of this case, Gramsch introduced the notion of $\Psi_{0}$ and $\Psi^{\star}$ algebras in the abstract setting (cf \cite{gr1,gr2}).\\
Subsequent to the work of Gramsch, many other works were done in this field (cf\cite{br,co,ka,sh1,sh2}). It should be remarked here that due to the work of Rickard and Waelbroek, every $\Psi^{\star}$ algebra in $\mathcal{L}(H)$, where $H$ is a Hilbert space, contains its holomorphic functional calculus (cf. \cite{la}) in the sense of Taylor(cf \cite{TA}).\par

Different methods for generalizing and $\Psi^{\star}$ algebras were studied in different papers. For example the method using the unitary group actions was studied in \cite{br} and \cite{TA}. Another method called the commutator method, can be found in (cf. \cite{gr1,gr3}). More precisely, starting from a finite family $\mathcal{V}$ of densely defined closed operators on $H$, and $a\ \in\ \mathcal{L}(H)$ we require that all the iterated commutators
\begin{equation*}
[[a,V_{1}],V_{2},\cdots],\ \ V_{j}\in \mathcal{V}
\end{equation*}
are well defined on a suitable dense subset of H and they admit an extension in $\mathcal{L}(H)$ (cf section 2). Using this method, we are interested in generating $\Psi_{0}\ and\ \Psi^{\star}$ algebras in $\mathcal{L}(L^{2}(B_{n}))$ containing the Bergman projection $P$ on the unit open ball $B_n \subset \mathbb{C}^{n}$ by using a family of linear vector fields. Linear vector field $X$ is defined by a matrix $A\in M_{2n}(\mathbb{C})$ given by 
\begin{equation*}
X=\sum_{i,j}^{2n}a_{i,j}x_{j}\frac{\partial}{\partial x_{i}}
\end{equation*}\par
where $(x_{1},\cdots,x_{2n})$ are the real coordinates of $\mathbb R^{2n}\approx\mathbb{C}^{n}$. And so for constructing new $\ \Psi^{\star}$ algebras containing $P$, a natural question arises:

\noindent
 What properties should $A$ have so that the commutator $[X,P]$ admits a continuous extension in $\mathcal{L}(L^{2}(B_{n}))$?
 
 \noindent
To answer this question, we write $X$ in the complex coordinates $(z_{1},\overline{z_{1}},\cdots,z_{n},\overline{z_{n}})$  and we define the property $\Im$ as a relation between the entries of $A$.We need to find a dense subset in $L^{2}(B_{n})$ which is invariant under $P$. We prove that the set of polynomials with complex coefficients is invariant under $P$.\\
After this step, we prove that a necessary condition for $[X,P]$ to admit a continuous extension is that the matrix $A$ should satisfies the property $\Im$. A nice result we then prove, any vector field $X$ with this property commutates with $P$. In fact, we prove that they commute on the set of smooth function with compact support. And by this result, its obviously to see that the higher commutators $[X,[X,P],\cdots]$ also vanishes on $L^{2}(B_{n})$ and therefore we obtain a $\Psi^\star$ algebras from any finite family of vector fields with this property.\par
In (cf), the author proved that the commutator of the Segal-Bergman projection $P_{1}$ (cf. Theorem \ref{th56}) with vector fields having bounded coefficients admits a continuous extension on $L^{2}(\mathbb{C}^{n})$ and $\Psi^\star$-algebras were constructed. For example $[P_{1},\frac{\partial}{\partial z_{l}}]=0$ and $[P_1,\frac{\partial}{\partial\overline{z_{l}}}]=P_1\frac{\partial}{\partial\overline{z_{l}}}$ admits a continuous extension.
\section{The commutator method for $\Psi^\star$-algebra}

The method we follow for constructing  a $\Psi^\star$-algebra containing the Bergman projection on the unit ball is known as the commutator method and is due to  Gramsch. We restrict our work to the operator algebra $\mathcal{B}=\mathcal{L}(H)$ where $(H,\|.\|_{H})$ is a Hilbert space (in our case $H=L^2(\mathbb B_n)$), and we sketch the method for constructing new  $\Psi^{\star}$-algebras.\\
Let $\mathcal{V}$ be a finite set of closed densely defined operators. For each $A \in \mathcal{V}$, $A:H\supset D(A) \longmapsto H$ we define:
\begin{enumerate}
	\item  $\mathcal{T}(A):=\{a\in \mathcal{F}\mid\ a(D(A))\subset D(A) \}$.
	\item $\mathcal{B}(A):=\{a \in \mathcal{T}(A)\mid\ ad[A]a:=Aa-aA:H \supset D(A) \mapsto H\ extends\ to\ \partial_{A}(a)\in \mathcal{F} \}$.
\end{enumerate}
It is not so hard to prove that for each $A\in \mathcal{V}$, the operator $\partial_{A}:\mathcal{F}\supset\mathcal{B}(A)\longmapsto\mathcal{F}$ is a closed derivation between algebras. Moreover, if $A$ is symmetric and we define
\begin{equation*}
\mathcal{B}^{\star}(A):=\{a\in \mathcal{B}(A)\mid a^{\star}\in \mathcal{B}(A) \}
\end{equation*}
then $\partial_{A}:\mathcal{B}(A) \longmapsto \mathcal{F}$ is closed $\star-derivation$ between symmetric algebras. A method generating new $\Psi_{0}\ and\ \Psi^{\star}$-algebras in $\mathcal{F}$, from a given finite family of closed derivations and $\star-deviation$ was discussed by Gramsch in cf. Applying those constructions to our case we define a decreasing sequence of subalgebras in $\mathcal{F}$:
\begin{equation*}
\Psi_{0}^{\mathcal{V}}[\mathcal{F}]:=\mathcal{F}\supset\Psi_{1}^{\mathcal{V}}[\mathcal{F}]:=\cap_{\mathcal{A}\in \mathcal{V}}\mathcal{B}(\mathcal{A})\supset\cdots\supset\Psi_{k}^{\mathcal{V}}[\mathcal{F}]\supset\cdots\supset\Psi_{\infty}^{\mathcal{V}}[\mathcal{F}]
\end{equation*}
where
\begin{enumerate}
	\item For $k\geq 2,\Psi_{k}^{\mathcal{V}}[\mathcal{F}]:=\{a\in \Psi_{k-1}^{\mathcal{V}}[\mathcal{F}];\ \delta_{\mathcal{A}}(a)\in \Psi_{k-1}^{\mathcal{V}}[\mathcal{F}]\ \forall \mathcal{A} \in \mathcal{V} \}$\\
	\item $\Psi_{\infty}^{\mathcal{V}}[\mathcal{F}]:=\cap_{k \in \mathbb{N}}\Psi_{k}^{\mathcal{V}}[\mathcal{F}]$
\end{enumerate}
Now, we endow each $\Psi_{k}^{\mathcal{V}}[\mathcal{F}]$, $k\geq 1$, with the family of submultiplicative semi-norms $\{q_{k,j} \}_{j\in\mathbb{N}}$:
\begin{equation*}
q_{k,j}(.):=q_{k-1,j}(.)+\sum_{\mathcal{A}\in\mathcal{V}}q_{k-1,j}(\delta_{\mathcal{A}}.)
\end{equation*}
As a result to the above construction, we obtain[ref]:
\begin{proposition}\label{ppww}
	For each fixed k, $(\Psi_{k}^{\mathcal{V}}[\mathcal{F}],\{q_{k,j}\}_{j\in\mathbb{N}})\hookrightarrow \mathcal{F}$ is continuously embedded submultiplicative Fr\'{e}chet  algebra. And $(\Psi_{\infty}^{\mathcal{V}}[\mathcal{F}],\{q_{k,j}\}_{k,j\in\mathbb{N}})$ is a submultiplicative $\Psi_{0}$ algebra in $\mathcal{L}(H)$. Moreover, if $\mathcal{A}\in\mathcal{V}$ is symmetric then replacing $\mathcal{B}(\mathcal{A})\ by\ \mathcal{B}^{\star}(\mathcal{A})$, gives that $\Psi_{\infty}^{\mathcal{V}}[\mathcal{F}]$ is a submultiplicative $\Psi^{\star}$ algebra in $\mathcal{L}(H)$.
\end{proposition}\par
Now we define the higher order commutators with a finite system of operators, to get a sufficient condition for $a\in\mathcal{F}$ to belong to the algebra $\Psi_{k}^{\mathcal{V}}[\mathcal{F}]$:
\begin{definition}
	Let M be linear subspace of H, L(M) denotes the set of linear operators on M. Consider $B\in L(M)$ and a finite system:
	\begin{equation*}
	\mathcal{A}_{k}=[A_{1},\cdots,A_{k}]\ where\ A_{j}\in L(M),
	\end{equation*}
	We define inductively:
	\begin{equation*}
	ad[A_{1}](B)=[A,B],\ \ \ \ ad[a_{1},\cdots,A_{j}](B)=ad[A_{j}](ad[A_{1},\cdots,A_{j-1}](B))
	\end{equation*}
	Then $ad[A_{k}](B)$ is called the commutator with the system $\mathcal{A}_{k}$.
\end{definition}\par
Let $A:H\supset D \longmapsto H$ be a closable densely defined operator on $H$. And let $A:H\supset D(A)\longmapsto H$ be its minimal closed extension in $H$. That is $D$ is dense in $D(A)$ with respect to the graph norm $(\|.\|_{gr}:=\|.\|_{H}+\|A.\|_{H})$. Moreover, suppose there is $a\in \mathcal{L}(H)$, such that $a(D)\subset D(A)$ then the following lemma holds:
\begin{lemma}
	Suppose that the commutator $[A,a]:D\longmapsto H$ admits a continuous extension $\delta_{A}(a)\in \mathcal{L}(H)$. Then
	\begin{enumerate}
		\item D(A) is invariant under a, i.e. $a(D(A))\subset D(A)$.
		\item $\delta_{A}(a)$ is the extension of the well defined operator: $[A,a]:D(A)\longmapsto H$
	\end{enumerate}
       \begin{proof}
	Let $x\in D(A)$, then there is a sequence $(y_{n})\in D$ such that $y_{n}\longmapsto x$ and $ay_{n}\longmapsto Ax$ with respect to $\|.\|_{H}$. Since $a\in \mathcal{L}(H)$, then $ay_{n}\mapsto ax(\star)$ and $aAy_{n}\mapsto aAx$, then
	\begin{equation*}
	Aay_{n}=\delta_{A}(a)y_{n}+aAy_{n}
	\end{equation*}
	converges in H. Since A is closed on D(A), and using $(\star)$ we get $ax\in D(A)$. Hence $[A,a]:D(A)\longmapsto H$ is well defined, and
	\begin{equation*}
	[A,a](x)=\lim_{n\to\infty}[A,a]y_{n}=\lim_{n\to\infty}\delta_{A}(a)y_{n}=\delta_{A}(a)(x).
	\end{equation*}
       \end{proof}
\end{lemma}
\begin{proposition}\label{plo}
	Let $k \in \mathbb{N} \bigcup\{{\infty}\}$,$a\in\mathcal{F}$, and $\mathcal{V}$ be a finite set of closed densely defined operators. Such that there is $D\subset H$ dense, with $D(A)=(\overline{D},\|.\|_{gr}:=\|.\|_{H}+\|A.\|_{H}$ for all $A\in \mathcal{V}$. Moreover, suppose the following property $(E_k)$ holds:\\
	$(E_k): D$ is invariant under $a$, and under all $A \in \mathcal{V}$. And suppose for any system $\mathcal{A}$ in\\
	$$S_k(\mathcal{V}):= \{[A_1,\ldots,A_j], A_j \in \mathcal{V}, j \leq k\}$$\\
	The commutators $ad[\mathcal{A}](a): D \longrightarrow H$, admits a continuous extension $\delta_\mathcal{A}(a)\in \mathcal{F}$.\\
	Then $a \in \Psi^\mathcal{V}_k[\mathcal{F}]$, and $\delta_\mathcal{A}(a)$ is the bounded extension of $ad[\mathcal{A}](a): H \supset D(A) \longrightarrow H$ for all $A \in \mathcal{V}$.\\
	
\end{proposition}\par
\begin{proof}
	The previous Lemma proves the proposition for $k = 1$. Now assume that we proved $a \in \Psi ^\mathcal{V}_{k-1}[\mathcal{F}]$, and consider $A \in \mathcal{V}$. Then by our assumption $ad[A](a)$ admits a continuous extension $\delta_A (a) \in \mathcal{F}$. However for any system $\mathcal{A}_j$, $j\leq k-1$, we have\\
	$$ad[\mathcal{A}_j,A](a)= ad[\mathcal{A}_j](\delta_A(a))$$
	hence by induction $\delta_A(a) \in \Psi ^\mathcal{V} _{k-1}$, i.e. $ a \in \Psi ^ \mathcal{V} _k [\mathcal{F}]$.
\end{proof}

\section[Continuity results: case of constant coefficients]{Continuity results for derivations with constant coefficients}
\hskip1cm This section is devoted to study the extension by continuity for the commutator $[X,P]$ whenever $X$ is formed of partial derivatives with respect to $z_i$ or with respect to $\overline{z_i}$.
 The motivation started by the following results we obtained: let $Y=\frac{\partial}{\partial \overline{z_{i}}}$ then the commutator $[P,Y]$ is not continuous. Before proving our result we need to
test the behavior of $P$ on the monomials.
\begin{proposition}\label{p3}
	Let $\alpha,\beta\in\mathbb N^n$, then the Bergman projection of $z^{\alpha}\overline{z}^{\beta}$, $P(z^{\alpha}\overline{z}^{\beta})$ is given by
	\begin{equation*}
	\begin{cases}
	\dfrac{\alpha!}{(n+|\alpha|)!}\dfrac{[|\alpha-\beta|+n]!}{(\alpha-\beta)!}z^{\alpha-\beta},\ \ \alpha \geq \beta\\
	0 \hskip6cm else
	\end{cases}
	\end{equation*}
\end{proposition}
\begin{proof}
	We start by regarding the following series representation. Given  $\lambda$ non negative integer, we have 
	\begin{equation*}
	(1+x)^{-\lambda}=\sum_{k=0}^{\infty}\dfrac{\Gamma(k+\lambda)}{k!\Gamma( \lambda)}x^{k},\ \ \ for |x|< 1 	
	\end{equation*}
Hence,	\begin{align*}
 \dfrac{1}{(1-z.\overline{w})^{n+1}}&=\sum_{k=0}^{\infty}\dfrac{\Gamma(k+n+1)}{k!\Gamma(n+1)}(z.\overline{w})^{k}\\
	&=\sum_{k=0}^{\infty}\dfrac{(k+n)!}{k!n!}(z.\overline{w})^{k}\\
	&=\sum_{k=0}^{\infty}
	\dfrac{(k+n)!}{k!n!}\sum_{|\gamma|=k} \dfrac{k!}{\gamma!}z^{\gamma} \overline{w}^{\gamma}
	\end{align*}
	Now using the fact:$$\int_{B_n} z^{\alpha}\overline{z}^{\beta}d\mu(z)=\begin{cases}
	\dfrac{n!\alpha!}{(n+|\alpha|)!}, \hskip2cm \alpha=\beta\\
	0,\hskip4cm else,
	\end{cases}$$
	we obtain 
	\begin{align*}
	P(z^{\alpha}\overline{z}^{\beta})&	
	=\int_{B_n}\dfrac{w^{\alpha}w^{-\beta}}{(1-z.\overline{w})^{n+1}}d\mu(w)\\
	&=\sum_{k=0}^{\infty}\dfrac{(k+n)!}{k!n!}\sum_{|\gamma|=k}\dfrac{k!}{\gamma!}z^{\gamma}\int_{B_n} w^{\alpha}\overline{w}^{\beta}\overline{w}^{\gamma}d\mu(w).
	\end{align*}
	Choosing $ \beta+\gamma=\alpha$ we obtain $k=|\alpha-\beta|$. Therefore,
	\begin{align*}
	P(z^{\alpha}\overline{z}^{\beta})&=\dfrac{(|\alpha-\beta|+n)!}{|\alpha-\beta|!n!}\dfrac{|\alpha-\beta|!}{(\alpha-\beta)!}\int_{B_n}w^{\alpha}\overline{w}^{\alpha}z^{\alpha-\beta}d\mu(w)\\
	&=\dfrac{(|\alpha-\beta|+n)!}{n!(\alpha-\beta)!}\dfrac{n!\alpha!}{(n+|\alpha|)!}z^{\alpha-\beta}\\
	&=\dfrac{(|\alpha-\beta|+n)!}{(\alpha-\beta)!}\dfrac{\alpha!}{(n+|\alpha|)!}z^{\alpha-\beta},
	\end{align*}
\end{proof}
\begin{remark}
	By the above calculations, we have:
	\begin{equation*}
	\|z^{\alpha}\overline{z}^{\beta}\|^{2}=\frac{n!(\alpha+\beta)!}{(n+|\alpha+\beta|)!}
	\end{equation*}
\end{remark}
\begin{remark}
	Let $D$ denotes the set of all polynomials with complex coefficient, then $D$ is dense, closed under linear vector fields, and by the above proposition it is also closed under the Bergman projection.
\end{remark}
Below is our first result concerning the continuity of the commutator of a derivation and the Bergman projection on the unit ball of $\mathbb C^n$.
\begin{theorem}\label{lqu}
	For all $i=1,\cdots,n$ the commutator $[\frac{\partial}{\partial z_{i}},P]$ is does not admit a continuous extension to $L^{2}(B_{n})$.
\end{theorem}
	\begin{proof}
		Consider $z_{i}^{p}\overline{z_{i}}^{q}$ with $p>q+1$, then by Proposition \ref{p3} we obtain
		\begin{equation*}
		P(\frac{\partial}{\partial z_{i}}(z_{i}^{p}\overline{z_{i}}^{q}))=pPz_{i}^{p-1}\overline{z_{i}}^{q}=\frac{p![p-q-1+n]!}{(n+p-1)!(p-q-1)!}z_{i}^{p-q-1}
		\end{equation*}
		By another application of Proposition \ref{p3} we get
		\begin{equation*}
		\frac{\partial}{\partial z_{i}}(P(z_{i}^{p}\overline{z_{i}}^{q}))=	\frac{\partial}{\partial z_{i}}(\frac{p![p-q+n]!}{(n+p)!(p-q)!}z_{i}^{p-q})=\frac{p![p-q+n]!}{(n+p)!(p-q-1)!}z_{i}^{p-q-1}
		\end{equation*}
		Therefore
		\begin{align*}
		[\frac{\partial}{\partial z_{i}},P](z_{i}^{p}\overline{z_{i}}^{q})&=p!\{\frac{[p-q+n]!-(n+p)[p-q-1+n]!}{(n+p)!(p-q-1)!} \}z_{i}^{p-q-1}\\
		&=-p!q\frac{[p-q-1+n]!}{(n+p)!(p-q-1)!}z_{i}^{p-q-1}
		\end{align*}
		Now suppose  that the commutator is continuous, then there is a real positive number $c$ such that
		\begin{equation*}
		\|	[\frac{\partial}{\partial z_{i}},P]f\|^{2}\leq c\|f\|^{2}  \mbox{ for all } f\in L^{2}.
		\end{equation*}
	This is equivalent to say that	\begin{equation*}	\|[\frac{\partial}{\partial z_{i}},P](z_{i}^{p}\overline{z_{i}}^{q})\|^{2}\leq c\|(z_{i}^{p}\overline{z_{i}}^{q})\|^{2}.
		\end{equation*}
	By the above calculations we get	\begin{equation*}
		\{p!q\frac{[p-q-1+n]!}{(n+p)!(p-q-1)!} \}^{2}\|z_{i}^{p-q-1}\|^{2}\leq c\|(z_{i}^{p}\overline{z_{i}}^{q})\|^{2}
		\end{equation*}
Hence,
		\begin{equation*}
\{p!q\frac{[p-q-1+n]!}{(n+p)!(p-q-1)!} \}^{2}\frac{n!(p-q-1)!}{(n+p-q-1)!}\leq c\frac{n!(p+q)!}{(n+p+q)!}
		\end{equation*}
		so that
		\begin{equation*}
	q^{2}\frac{p!p!(p-q-1+n)!}{(n+p)!(n+p)!(p-q-1)!}\frac{(n+p+q)!}{(p+q)!}\leq c.
		\end{equation*}
		Let $p=2m$ and $q=m$ then:
		\begin{equation*}
		m^{2}\frac{(2m)!(2m)!(m-1+n)!}{(n+2m)!(n+2m)!(m-1)!}\frac{(3m+n)!}{(3m)!}\leq c
		\end{equation*}
		i.e.
		\begin{equation*}
		m^{2}\frac{(m-1+n)...m}{(2m+n)...(2m+1)(2m+n)...(2m+1)}(3m+n)...(3m+1)\leq c.
		\end{equation*}
		We obtain a contradiction by taking $m$ big enough.
	\end{proof}

As a consequence of the above result, we obtain
\begin{corollary}\label{lquu}
	Let $X=\sum_{i=1}^{n}a_{i}\frac{\partial}{\partial z_{i}}$ where $a_{i}$ are constants not all zeros, then $[X,P]$ does not admit a continuous extension to $L^{2}(B_{n})$.
\end{corollary}
	\begin{proof}
		Without loss of generality, suppose $a_{j}\neq 0$, and suppose there exist $c>0$ such that
		\begin{equation*}
		\|\sum_{i=1}^{n}a_{i}[\frac{\partial}{\partial z_{i}},P]f\|^{2}\leq c\|f\|^{2} \mbox{ for all } f\in L^{2}.
		\end{equation*}
		Consider the functions $(z_{j}^{2m}\overline{z_{j}}^{m})$, then $[\frac{\partial}{\partial z_{i}},P](z_{j}^{2m}\overline{z_{j}}^{m})=0$  for all $i\neq j$ so that
		\begin{equation*}
	 \|\sum_{i=1}^{n}a_{i}[\frac{\partial}{\partial z_{i}},P](z_{j}^{2m}\overline{z_{j}}^{m})\|^{2}=a_{j}^{2}\|[\frac{\partial}{\partial z_{i}},P](z_{j}^{2m}\overline{z_{j}}^{m})\|^{2}\leq c\|z_{j}^{2m}\overline{z_{j}}^{m}\|^{2}
		\end{equation*}
		which is not true by the previous theorem.
	\end{proof}\\
Here is our second result for the criteria of the continuity of the commutator.
\begin{theorem}\label{lquuu}
	For all $i=1,\cdots,n$ the commutator $[P,\frac{\partial}{\partial\overline{z_{i}}}]$ is does not admit a continuous extension to $L^{2}(B_{n})$.
\end{theorem}
	\begin{proof}
		Since $\frac{\partial}{\partial\overline{z_{i}}}P=0$, then
		\begin{equation*}
		[P,\frac{\partial}{\partial\overline{z_{i}}}]=P\frac{\partial}{\partial\overline{z_{i}}}
		\end{equation*}
		Suppose that the commutator is continuous then there is $c>0$ such that
		\begin{equation*}
		\|P\frac{\partial}{\partial\overline{z_{i}f}}\|^{2}\leq c\|f\|^{2}\ \forall\ f\in L^{2}.
		\end{equation*}
		For $p>q$, we have
		\begin{equation*}
		P\frac{\partial}{\partial\overline{z_{i}}}(z_{i}^{p}\overline{z}_{i}^{q})=qP(z_{i}^{p}\overline{z}_{i}^{q-1})=q\frac{p![p-q+1+n]!}{(n+p)!(p-q+1)!}z_{i}^{p-q+1}
		\end{equation*}
		so that	
		\begin{equation*}
\{q\frac{p![p-q+1+n]!}{(n+p)!(p-q+1)!} \}^{2} \|z_{i}^{p-q+1}\|^{2}\leq c\|(z_{i}^{p}\overline{z}_{i}^{q})\|^{2}
		\end{equation*}
	i.e.	\begin{equation*}
		\{q\frac{p![p-q+1+n]!}{(n+p)!(p-q+1)!} \}^{2}\frac{n![p-q+1]!}{(p-q+1+n)!}\leq c\frac{n![p+q]!}{(n+p+q)!}
		\end{equation*}
		and
		\begin{equation*}
		 q^{2}\frac{p!p![p-q+1+n]!}{(n+p)!(n+p)!(p-q+1)!}\frac{(n+p+q)!}{[p+q]!}\leq c
		\end{equation*}
		Taking $p=2m=2q$ we obtain
		\begin{equation*}
		m^{2}\frac{(2m)!(2m)!(m+1+n)!}{(n+2m)!(n+2m)!(m+1)!}\frac{(n+3m)!}{(3m)!}\leq c
		\end{equation*}
		i.e.
		\begin{equation*}
		m^{2}\frac{(m+n+1)...(m+2)}{(2m+n)..(2m+1)(2m+n)..(2m+1)}(3m+n)..(3m+1)\leq c.
		\end{equation*}
		This is equivalent to say that whenever $m$ is large enough we have
		\begin{equation*}
		m^{2}\frac{m^{2n}}{m^{2n}}\leq c.
		\end{equation*}
	\end{proof}\\
Using the above theorem and following a similar technique to that used in Corollary \ref{c2}, we obtain:
\begin{corollary}\label{c2}
	Let $Y=\sum_{i=1}^{n}a_{i}\frac{\partial}{\partial \overline{z_{i}}}$, where $a_{i}$ are constants not all zeros, then $[Y,P]$ is not continuous on $L^{2}(B_{n})$.
\end{corollary}
\section[Continuity results: case of linear vector fields]{Continuity results for the commutator defined by linear vector fields}
In the sequel, let $X=(Ax)^{t}\frac{d}{dx}$ be a  vector field where $A\in M_{2n}(\mathbb R)$ is a given matrix is a skew-symmetric matrix. We study the problem of extension by continuity for  $[X,P]$ to $L^2(B_n)$.

\begin{proposition}
Let $A\in M_{2n}(\mathbb R)$ and $X$ be the vector field defined by $X=(Ax)^{t}\frac{d}{dx}$. Denote by   $S^{2n-1}$ the unit sphere in $\mathbb{C}^n$. Then $X\in\mathfrak{X}(S^{2n-1})$ if and only if $A$ ia an antisymmetric matrix.
\end{proposition}
	\begin{proof}
		First note that, for a given $x\in S^{2n-1}$ we know that $Ax\in T_x S^{2n-1}$ if and only if $<x, Ax> =0$. This shows the sufficient condition. For the necessary condition,  consider arbitrary vectors $x,y \in S^{2n-1}$ then
		\begin{align*}
<x+y, A(x+y)>&= <x,A(x+y)> +<y,A(x+y)>\\
&=<x,Ax>+<x,Ay>+<Ax,y>+<Ay,y>\\
&=<x,Ay>+<Ax,y>\\
&=<x,(A+A^t)y>=0.
		\end{align*}
This shows thatsatisfying $A+A^t=0$ i.e. $A$ is an antisymmetric matrix.
	\end{proof}
\begin{remark}
	Let $X=(Ax)^{t}\frac{d}{dx} \in\mathfrak{X}(S^{2n-1})$ then  the one parameter group $e^{As}$ satisfies
	\begin{enumerate}
		\item $det (e^{As}) = e^{trace As} =e^0 =1$
		\item $e^{As}$ is a unitary matrix
		\item Given $f\in C^\infty(S^{2n-1})$ we have $X(f)= \frac{d}{ds} (f(e^{As}.))_{s=0}$
	\end{enumerate}
\end{remark}
Below is our result concerning the commuting property between linear vector fields and the Bergman Projection on the unit ball.
\begin{theorem}\label{ttt111}
	Let  $X=(Ax)^{t}\frac{d}{dx}\in\mathfrak{X}(S^{2n-1})$ and let $P$ be the Bergman projection on the unit ball of $\mathbb C^n$. Then the commutator $[X,P]$ vanishes on $L^{2}(B_{n})$. In particular, $[X,P]$ is continuous on $L^{2}(B_{n})$.
\end{theorem}
	\begin{proof}
		Let $f\in L^{2}(B_{n})$ and assume that $f$ admits a $\mathbb{C}^\infty$ extension for some neighborhood of the closed unit ball. It is then direct that $g_s=f\circ e^{As}\in L^{2}(B_{n})$ and we have
		
		\begin{align*}
		P(g_s)(z)&= \int_{B_{n}} \frac{f(e^{As}w)}{(1-z^t\overline{w})^{n+1}} d\mu (w)= \int_{B_{n}} \frac{f(u)|det(Du)|}{(1-z^t\overline{e^{As}u})^{n+1}} d\mu (u)\\
		&=\int_{B_{n}} \frac{f(u)}{(1-z^te^{-As}\overline{u})^{n+1}} d\mu (u)=\int_{B_{n}}\frac{f(u)}{(1-(e^{As}z)^t\overline{u})^{n+1}} d\mu (u) = (Pf)(e^{As}z),
		\end{align*} 
		 where we used the change of variable $u=e^{-As}w$. Using the above observation, together with the fact that $P$ is bounded we obtain
		   $$X(P(f)) =\frac{d}{ds}[(Pf)(e^{As}.)]_{s=0} =\frac{d}{ds}[P(g_s))]_{s=0} = P\frac{dg_s}{ds}_{\mid s=0}= P(X(f)).$$
	\end{proof}

\section{Conclusion and open problems}
We summarize the results obtained in the thesis and compare to the case of Segal-Bargmann space.

\begin{enumerate}
	
	\item Following Lemma 2.2.1  in \cite{br}, the author proved that the Bergman projection $P_1$ from $L^2(\mathbb C^n, d\mu)$ to the Segal Bargmann space commutes with $\frac{\partial}{\partial z_{i}}$. However, following Theorem \ref{lqu} and Corollary \ref{lquu}  we showed that $[\frac{\partial}{\partial z_{i}},P]$  is not zero and does not admit a continuous extension to $L^2(\mathbb B_n)$.
	\item  In the above mentioned lemma the author also proved that $[P_1,\frac{\partial}{\partial\overline{z_{i}}}]$ admits  a continuous extension to $L^2(\mathbb C^n, d\mu)$. Following Theorem \ref{lquuu} and Corollary \ref{c2} we showed that $\sum_ka_k[P,\frac{\partial}{\partial\overline{z_{k}}}]$, (not all the constants are zeros), can not be extended by continuity to $L^2(\mathbb B_n)$.\\
	\item Based on the above observation for the continuity of the commutator with respect to the Wirintger derivatives for the case of the Segal-Bargmann space, W. Bauer constructed a collection of vector fields so that the iterated commutators with $P_1$ admits a continuous extension to $L^2(\mathbb C^n, d\mu)$. The main point was to test the behavior of high order commutators and to check the continuity. In Theorem \ref{ttt111} we showed that every linear vector field tangent to the unit sphere commutes with the Bergman projection on the unit ball.  This ensures that the higher commutator with $P$ vanishes.\\
	\item Following the above mentioned results, the construction of a $\Psi^\star$ is then ensured by an application of Proposition \ref{ppww} i.e. the commutator method in sense of Gramsch. In the case of the Segal-Bargmann this results a $\Psi^\star$-algebra in $\mathcal L(L^2(\mathbb C^n, d\mu))$ containing $P_1$. In our case, we obtained a $\Psi^\star$-algebra in $\mathcal L(L^2(\mathbb B_n))$ containing $P$ (cf. Proposition \ref{plo}) .
\end{enumerate}
Motivated  by our results, we provide below some open problems.
\begin{enumerate}
	\item Following the work of Vasilevski, we know that the commuting problem of Toeplitz operator on $\mathbb B_n$ is equivalent to another problem on the geometry of the unit sphere (for details cf. \cite{grudsky0}). In our case we think that the analysis problem for the extension is also related to the geometric structure of the sphere.
	\item We did not apply the results of the construction of $\Psi^\star$-algebra neither for the embedding nor for the spectral invariance. However, we think that this gives rise to control and provide some sharp estimation for the norm of the Bergman projection on Sobolev spaces.
\end{enumerate}

\end{document}